\documentclass[10pt,reqno,final]{amsart}

\usepackage[utf8]{inputenc}
\usepackage[T1]{fontenc}
\usepackage{fixltx2e}
\usepackage{microtype}
\usepackage{amssymb}
\usepackage{bbm}
\usepackage[letterpaper,hmargin=1in,vmargin=1in]{geometry}
\usepackage{mathtools}
\usepackage{mathrsfs}
\usepackage{enumerate}
\usepackage{ifdraft}
\usepackage{booktabs}
\usepackage{suffix}

\newcommand*{\Ccal}{\mathcal{C}}
\newcommand*{\Dcal}{\mathcal{D}}
\newcommand*{\Tcal}{\mathcal{T}}
\newcommand*{\Cscr}{\mathscr{C}}
\newcommand*{\Fscr}{\mathscr{F}}
\newcommand*{\Hscr}{\mathscr{H}}

\newcommand{\abl}[2][]{\operatorname{abl}\paren[#1]{#2}}
\WithSuffix\newcommand\abl*[1]{\operatorname{abl}\paren*{#1}}

\DeclareMathOperator*{\argmax}{arg\,max}
\DeclareMathOperator{\Aut}{Aut}
\DeclarePairedDelimiter\card{\lvert}{\rvert}
\newcommand*{\chromatic}{\chi}
\DeclareMathOperator{\cl}{cl}
\DeclarePairedDelimiter\curly{\{}{\}}

\DeclareMathOperator{\diag}{diag}
\DeclareMathOperator{\Diag}{Diag}
\newcommand*{\drop}{\setminus}
\newcommand*{\eps}{\varepsilon}
\newcommand*{\ffrom}{\colon}
\newcommand*{\fto}{\to}
\newcommand{\gauge}[3][]{\gamma#1(#3\thinspace#1|\thinspace#2#1)}
\WithSuffix\newcommand\gauge*[2]{\gamma\left(#2\thinspace\middle|\thinspace#1\right)}
\newcommand*{\hprod}{\mathbin{\odot}}
\DeclareMathOperator{\Image}{Im}
\newcommand*{\incidvector}[1]{\mathbbm{1}_{#1}}

\DeclareMathOperator{\interior}{int}
\DeclarePairedDelimiter\anglebrackets{\langle}{\rangle}
\newcommand\iprod[3][]{\anglebrackets[#1]{#2,#3}}
\WithSuffix\newcommand\iprod*[2]{\anglebrackets*{#1,#2}}
\newcommand*{\iprodt}[2]{#1^{\transp}#2}
\newcommand\Iverson[2][]{\sqbrac[#1]{#2}}
\WithSuffix\newcommand\Iverson*[1]{\sqbrac*{#1}}
\newcommand*{\lambdamin}{\lambda_{\min}}
\newcommand*{\lambdamax}{\lambda_{\max}}

\let\emptyset\varnothing

\DeclarePairedDelimiter\norm{\lVert}{\rVert}
\newcommand{\NormalCone}[3][]{\operatorname{Normal}\paren[#1]{#2;#3}}
\WithSuffix\newcommand\NormalCone*[2]{\operatorname{Normal}\paren*{#1;#2}}
\newcommand\pnorm[3][]{\norm[#1]{#2}_{#3}}
\WithSuffix\newcommand\pnorm*[2]{\norm*{#1}_{#2}}
\DeclareMathOperator{\Null}{Null}
\newcommand*{\ones}{\bar{e}}
\newcommand*{\oprod}[2]{#1^{}#2^{\transp}}
\newcommand*{\oprodsym}[1]{\oprod{#1}{#1}}
\DeclarePairedDelimiter\paren{(}{)}

\newcommand*{\Psd}[1]{\Symraw_+^{#1}}

\def\restriction{\mathord{\upharpoonright}}
\newcommand*{\qform}[2]{#2^{\transp}#1#2^{}}
\newcommand*{\Reals}{\mathbb{R}}
\newcommand*{\stab}{\alpha}
\DeclareMathOperator{\STAB}{STAB}
\DeclarePairedDelimiter\set{\{}{\}}
\newcommand\setst[3][]{\curly[#1]{\,{#2}\,\colon{#3}}}
\WithSuffix\newcommand\setst*[2]{\curly*{\,{#1}\,\colon{#2}}}
\DeclarePairedDelimiter\sqbrac{[}{]}

\DeclareMathOperator{\supp}{supp}
\newcommand{\suppf}[3][]{\delta^*#1(#3\thinspace#1|\thinspace#2#1)}
\WithSuffix\newcommand\suppf*[2]{\delta^*\left(#2\thinspace\middle|\thinspace#1\right)}
\newcommand*{\Sym}[1]{\Symraw^{#1}}
\newcommand*{\Symraw}{\mathbb{S}}
\newcommand*{\thalf}{\tfrac{1}{2}}
\DeclareMathOperator{\trace}{Tr}
\newcommand*{\transp}{\mathsf{T}}

\newlength{\claimmargin}
\newenvironment{claimeq}
{%
  \refstepcounter{equation}%
  \begin{list}
    {%
\hspace*{\textwidth}%    %% move all the way to the right
\hspace*{-\claimmargin}% %% move a bit back to the left
\hfill%                  %% add a glue
\textup{(\theequation)}% %% write equation tag
\hspace*{\claimmargin}%  %% move to the right again
\hspace*{-\textwidth}}%  %% then back to where we started
    {
      \setlength{\topsep}{\smallskipamount}
      \setlength{\leftmargin}{\claimmargin}
      \setlength{\rightmargin}{\claimmargin}
      \setlength{\labelsep}{0cm}
      \setlength{\labelwidth}{\claimmargin}
    }%
  \item%
  }%
  {%
  \end{list}%
}

\newcommand*{\tb}{\bar{t}}
\newcommand*{\wb}{\bar{w}}
\newcommand*{\Xb}{\bar{X}}
\newcommand*{\Xh}{\hat{X}}
\newcommand*{\Xt}{\tilde{X}}
\newcommand*{\xb}{\bar{x}}
\newcommand*{\Yh}{\hat{Y}}
\newcommand*{\yb}{\bar{y}}

\newcommand*{\AdjCone}{\mathbb{A}}
\newcommand*{\AdjDual}[1]{\overline{#1}}
\newcommand*{\AdjMatrix}[1]{A_{#1}}
\newcommand*{\Bb}{\bar{B}}
\newcommand*{\Bt}{\tilde{B}}
\newcommand*{\ComplPositive}[1]{\Ccal_{#1}^{*}}
\newcommand{\component}[3][]{\sqbrac[#1]{#2}_{#3}}
\WithSuffix\newcommand\component*[2]{\sqbrac*{#1}_{#2}}
\newcommand*{\ConvexSet}{\Cscr}
\newcommand*{\Copositive}[1]{\Ccal_{#1}^{}}
\newcommand*{\CongMap}[1]{\operatorname{Congr}_{#1}}
\newcommand*{\DiagScale}[2]{\Diag(#2)#1\Diag(#2)}
\newcommand*{\DiagScaleMap}[1]{\Dcal_{#1}}
\newcommand*{\dual}[1]{#1^*}
\newcommand*{\Euclidean}{\EuclideanA}
\newcommand*{\EuclideanA}{\mathbb{E}}
\newcommand*{\Face}{\Fscr}
\DeclareMathOperator{\FRAC}{FRAC}
\newcommand*{\fracChromatic}{\chi^*}
\newcommand\genTH[3][]{\TH\paren[#1]{#2,#3}}
\WithSuffix\newcommand\genTH*[2]{\TH\paren*{#1,#2}}
\newcommand*{\hb}{\bar{h}}
\newcommand\liftedGenTH[3][]{\widehat{\TH}\paren[#1]{#2,#3}}
\WithSuffix\newcommand\liftedGenTH*[2]{\widehat{\TH}\paren*{#1,#2}}
\newcommand*{\liftedQuadCone}{\widehat{\QuadCone}}
\newcommand*{\liftedThetaBodyA}{\widehat{\hspace{-1pt}\Cscr}}
\newcommand\LStheta[3][]{\upsilon\paren[#1]{#2;#3}}
\WithSuffix\newcommand\LStheta*[2][]{\upsilon\paren*{#1;#2}}
\newcommand\LSthetaC[3][]{\upsilon\paren[#1]{#2;#3}}
\WithSuffix\newcommand\LSthetaC*[2][]{\upsilon\paren*{#1;#2}}
\newcommand*{\MPinverse}[1]{#1^{\dagger}}
\newcommand*{\onelift}[2]{%
  \begin{bmatrix*}[l]
    1  & #1^{\transp}\thinspace \\
    #1 & #2                     \\
  \end{bmatrix*}}
\renewcommand*{\ones}{\mathbbm{1}}
\newcommand*{\PermMatrix}[1]{P(#1)}
\newcommand*{\PolyAdjCone}[2]{\mathbb{A}_{#1,#2}}
\newcommand*{\PolyAdjConeNode}[4]{\mathbb{A}_{#1,#2,#3,#4}}
\DeclareMathOperator{\QSTAB}{QSTAB}
\newcommand*{\QuadCone}{\mathbb{K}}
\newcommand*{\QuadConeTwo}[1]{\QuadCone_2^{#1}}
\newcommand\Schurlift[2][]{\operatorname{Schur}\paren[#1]{#2}}
\WithSuffix\newcommand\Schurlift*[1]{\operatorname{Schur}\paren*{#1}}
\newcommand\Psdlift[2][]{\operatorname{Psd}\paren[#1]{#2}}
\WithSuffix\newcommand\Psdlift*[1]{\operatorname{Psd}\paren*{#1}}
\newcommand*{\setlift}[1]{\set{0} \cup #1}
\DeclareMathOperator*{\Symmetrize}{Sym}
\newcommand*{\Symnonneg}[1]{\Symraw_{\scriptscriptstyle\geq 0}^{#1}}
\let\theta\vartheta
\let\vartheta\oldtheta
\newcommand*{\ThetaBodyA}{\Cscr}

\usepackage{ifpdf}
\ifpdf
\usepackage[pdftex,%
pdfauthor={Marcel Kenji de Carli Silva and Levent Tunçel},%
pdftitle={An Axiomatic Duality Framework%
  for the Theta Body and Related Convex Corners},%
colorlinks=true,%
\ifoptionfinal{
  linkcolor=black,%
  citecolor=black,%
}{}
hypertexnames=true,bookmarks=true,pagebackref]{hyperref}
\else
\fi

\theoremstyle{definition}
\newtheorem{corollary}{Corollary}
\newtheorem{proposition}[corollary]{Proposition}
\newtheorem{theorem}[corollary]{Theorem}
\newtheorem{lemma}[corollary]{Lemma}

\numberwithin{equation}{section}

\title[%
  An Axiomatic Duality Framework for the Theta Body and Related Convex Corners
]{%
  An Axiomatic Duality Framework\\
  for the Theta Body and Related Convex Corners
}

\author{Marcel K.\ de Carli Silva}
\thanks{%
  Part of the results in this paper appeared in the PhD
  thesis of the first author.\\
  \indent
  Research of the first author was supported in part by a Sinclair
  Scholarship, a Tutte Scholarship, Discovery Grants from NSERC, and
  by ONR research grant N00014-12-10049, while at the Department of
  Combinatorics and Optimization, University of Waterloo, and by grants
  2013/20740-9 and 2013/03447-6, São Paulo Research Foundation (FAPESP), while at
  the Institute of Mathematics and Statistics, University of São
  Paulo.
}
\author{Levent Tunçel}
\thanks{%
  Research of the second author was supported in part by a research
  grant from University of Waterloo, Discovery Grants from NSERC and
  by ONR research grant N00014-12-10049.
}

\date{\today}

\def\TH{\operatorname{TH}}

\begin{document}

\begin{abstract}
  Lovász theta function and the related theta body of graphs have been
  in the center of the intersection of four research areas:
  combinatorial optimization, graph theory, information theory, and
  semidefinite optimization.  In this paper, utilizing a modern convex
  optimization viewpoint, we provide a set of minimal conditions
  (axioms) under which certain key, desired properties are
  generalized, including the main equivalent characterizations of the
  theta function, the theta body of graphs, and the corresponding
  antiblocking duality relations.  Our framework describes several semidefinite and polyhedral relaxations of the stable set polytope of a graph as generalized theta bodies.  As a by-product of our approach, we
  introduce the notion of ``Schur Lifting'' of cones which is dual to
  PSD Lifting (more commonly used in SDP relaxations of combinatorial
  optimization problems) in our axiomatic generalization.  We also
  generalize the notion of complements of graphs to diagonally
  scaling-invariant polyhedral cones.  Finally, we provide a weighted generalization of the copositive formulation of the fractional chromatic number by Dukanovic and Rendl.
\end{abstract}

\maketitle

\ifoptionfinal{}{%
  \setcounter{tocdepth}{2}
  \tableofcontents
}

\maketitle

\section{Introduction}

The Lovász theta function is one of the most elegant highlights in
combinatorial and semidefinite optimization.  First introduced in the
seminal paper by Lovász~\cite{Lovasz79a} to solve a problem in
information theory, the theta function was further developed in the
1980's along with applications of the ellipsoid
method~\cite{GroetschelLS88a, GroetschelLS93a}, leading to the definition of the object
known as the theta body of a graph as a semidefinite relaxation of its
stable set polytope.  This relaxation is tight for perfect graphs, and
it leads to the only known (strongly) polynomial algorithm for finding
optimal stable sets and colorings in such graphs.  Since then, the
theory surrounding the Lovász theta function has been further
extended~\cite{LovaszS91a, Galtman00a, LuzS05a, GvozdenovicL08a}, and it has been used
in the design of approximation algorithms~\cite{KargerMS98a,
  KleinbergG98a, AlonMMN06a, ChanL12a}, in complexity theory~\cite{Tardos88a,
  FeigeL92a, BacikM95a, Bacik97a}, in information
theory~\cite{Marton93a, Simonyi01a, CKLMS1990}, and in extremal geometry~\cite{BachocNOV09a}.

For a graph \(G = (V,E)\), the \emph{theta body} of~\(G\) may be
defined as the set
\begin{equation}
  \label{eq:TH-intro}
  \TH(G)
  =
  \setst*{
    x \in \Reals^V
  }{
    \exists X \in \Sym{V},\,
    X_{ii} = x_i\,\forall i \in V,\,
    X_{ij} = 0\,\forall ij \in E,\,
    \begin{bmatrix}
      1 & x^{\transp}\thinspace \\
      x & X
    \end{bmatrix}
    \in \Psd{\setlift{V}}
  },
\end{equation}
where \(\Sym{V}\) denotes the set of \(V \times V\) symmetric matrices
and \(\Psd{\setlift{V}}\) is the set of symmetric positive
semidefinite matrices on the index set \(\setlift{V}\); we assume that
\(0\) is not an element of~\(V\).  The \emph{stable set polytope}
of~\(G\), denoted by \(\STAB(G)\), is defined as the convex hull of
incidence vectors of stable sets of~\(G\); a set \(S \subseteq V\) is
\emph{stable} in~\(G\) if no edge of~\(G\) joins two elements
of~\(S\).  It is not hard to check that, if \(x\) is the incidence
vector of a stable set of~\(G\), then \(X \coloneqq \oprodsym{x}\)
satisfies the constraints on the RHS of~\eqref{eq:TH-intro}, so \(x
\in \TH(G)\).  Thus, the theta body is a relaxation of~\(\STAB(G)\),
and the \emph{theta function}
\begin{equation}
  \label{eq:theta-intro}
  \theta(G) = \max\setst*{\textstyle\sum_{i \in V} x_i}{x \in \TH(G)}
\end{equation}
is an upper bound for the size of a largest stable set in~\(G\).  It
is also convenient to define a weighted version~\(\theta(G;w)\) of the
theta function, by multiplying~\(x_i\) in the objective function
of~\eqref{eq:theta-intro} by some weight \(w_i \geq 0\) for each \(i
\in V\).

Part of the broad applicability of the theta function owes to its
multitude of equivalent formulations, which led
Goemans to the conclusion that ``it seems all paths
lead to \(\theta\)!\hspace{1pt}''~\cite{Goemans97a}.  For instance, the eigenvalue
formulation
\begin{equation}
  \label{eq:lambda2-intro}
  \theta(G) = \min\setst*{
    \lambdamax\paren*{A + \oprodsym{\ones}}
  }{
    A \in \Sym{V},\,
    A_{ij} \neq 0 \implies ij \in E
  },
\end{equation}
where \(\lambdamax\) extracts the largest eigenvalue and \(\ones\) is
the vector of all-ones, is central for the approximate (vector)
coloring algorithm of Karger, Motwani and Sudan~\cite{KargerMS98a}.
Similarly, the non-convex formulation
\begin{equation*}
  \theta(G) = \max\setst*{
    \lambdamax(B)
  }{
    B \in \Psd{V},\,
    B_{ii} = 1\,\forall i \in V,\,
    B_{ij} = 0\,\forall ij \in E
  }
\end{equation*}
essentially says that \(\theta(G)\) is the best lower bound for the
chromatic number of~\(G\) from a family of bounds due to
Hoffman~\cite{Hoffman70a}; the \emph{chromatic number} of~\(G\) is the minimum size of a partition of~\(V\) into stable sets of~\(G\).  The description~\eqref{eq:TH-intro} itself
is also especially well suited for proving that all (nontrivial) facets
of~\(\TH(G)\) are determined by \emph{clique inequalities} (see, e.g.,
\cite[Theorem~67.13]{Schrijver03b}), i.e., inequalities of the form
\(\sum_{i \in K} x_i \leq 1\) for some clique~\(K\) of~\(G\); a set
\(K \subseteq V\) is a \emph{clique} of~\(G\) if every pair of
elements of~\(K\) is an edge of~\(G\).

Most of these alternative formulations also apply to some classical
variants of~\(\theta\), such as the functions~\(\theta'\)
and~\(\theta^+\) defined as in~\eqref{eq:theta-intro} over
corresponding variants of the theta body, denoted by~\(\TH'(G)\)
and~\(\TH^+(G)\), respectively (we define them in
Section~\ref{sec:theta-convex-corners}); see~\cite{McElieceRR78a,
  Schrijver79a, Szegedy94a}.  Many of such wealth of interesting
characterizations arguably come from Semidefinite Programming (SDP)
Strong Duality.  A particularly illuminating manifestation of this
duality is the identity (see, e.g.,
\cite[Theorem~67.12]{Schrijver03b})
\begin{subequations}
  \label{eq:abl-intro}
  \begin{equation}
    \label{eq:abl-TH-intro}
    \abl[\big]{\TH(G)} = \TH(\overline{G}),
  \end{equation}
  that is, the \emph{antiblocker} of \(\TH(G)\) is the theta body
  of~\(\overline{G}\), the complement of~\(G\).  Antiblocking duality
  is the notion of duality most appropriate for a class of convex sets
  known as convex corners (defined in
  Section~\ref{sec:theta-convex-corners}), which include all variants
  of theta bodies, as well as~\(\STAB(G)\).  Similar instances
  of~\eqref{eq:abl-TH-intro} include
  \begin{gather}
    \label{eq:abl-TH'-intro}
    \abl{\TH'(G)} = \TH^+(\overline{G}), \\
    \label{eq:abl-STAB-intro}
    \abl{\STAB(G)} = \QSTAB(\overline{G});
  \end{gather}
\end{subequations}
here, \(\QSTAB(G) \subseteq [0,1]^V\) is a classical polyhedral relaxation
for~\(\STAB(G)\) determined by clique inequalities (we define it in
Section~\ref{sec:TH-Cop}).  The relaxations in~\eqref{eq:abl-intro} are related by the following chain:
\begin{equation}
  \label{eq:STAB-chain}
  \STAB(G) \subseteq \TH'(G) \subseteq \TH(G) \subseteq \TH^+(G)
  \subseteq \QSTAB(G) \subseteq \FRAC(G);
\end{equation}
here, \(\FRAC(G) \subseteq [0,1]^V\) is determined by edge inequalities, i.e., clique inequalities where the clique is a single edge (we define it in Section~\ref{sec:schur-liftings}).  The beautiful, striking relationships in~\eqref{eq:abl-intro} continue to manifest themselves in certain lift-and-project methods \cite{AEN2002, GMS2003, LT2005}.

In this paper, we define a notion of generalized theta bodies and
develop a duality theory that: (1)~describes all sets in~\eqref{eq:STAB-chain} as theta bodies, (2) extends many of
the equivalent formulations for~\(\theta\) to the corresponding
generalized theta functions, and (3) extends the antiblocker
relations~\eqref{eq:abl-intro} to generalized theta bodies.  We shall
parameterize a generalized theta body, henceforth called just theta
body, by two convex cones. One of them, which we denote
by~\(\AdjCone\), shall encode the adjacency constraints \(X_{ij} = 0\)
in~\eqref{eq:TH-intro}, i.e., those constraints shall be replaced with
``\(X \in \AdjCone\)''.  The other cone, which we denote
by~\(\liftedQuadCone\), will replace~\(\Psd{\setlift{V}}\)
in~\eqref{eq:TH-intro}; it essentially constrains how \(X\) and \(\oprodsym{x}\) are related, so we may think of~\(\liftedQuadCone\) as a cone that encodes quadratic relations.  By varying these cones over some natural
families, we shall obtain a description of all sets in~\eqref{eq:STAB-chain}
as theta bodies, as well as the corresponding antiblocking duality
relations that links them in pairs as in~\eqref{eq:abl-intro}.  We thus unify the description of
all these relaxations and show that the ingenious though \emph{ad~hoc}
description~\eqref{eq:TH-intro} is in fact quite central, powerful, and natural.
The corresponding theory shall also make clearer the key role played
by the positive semidefinite cone in this generalized context, leading to
the most striking of the duality relations~\eqref{eq:abl-intro},
namely, the extremely symmetric relation~\eqref{eq:abl-TH-intro}.

The key axiom we shall need require from our parameter
cones~\(\AdjCone\) and~\(\liftedQuadCone\) will be their
\emph{diagonally scaling invariance}, i.e., these cones must be closed
under simultaneous left- and right-multiplications by any (and the
same) nonnegative diagonal matrix.  We shall then generalize the
proofs of equivalence of several formulations for~\(\theta\) to rely
(essentially) solely on this invariance property.  This axiomatic
approach also allows us to gauge the full power of the existing proof
methods; for instance, some of the equivalent formulations
for~\(\theta\) shall only work when the cone~\(\liftedQuadCone\) is
the positive semidefinite cone~\(\Psd{\setlift{V}}\).

Another advantage of unifying the equivalent formulations
of~\(\theta\) is that it provides less error-prone proofs of equivalent
formulations of the variants~\(\theta'\) and~\(\theta^+\).  Many such
formulations are listed in the literature without proof, with the
apparently implicit suggestion that, to prove them, it suffices to
repeat and slightly adapt the corresponding proofs for~\(\theta\).
While this may be true in most cases, it has already led to some
inaccuracies in the literature, as pointed out by the authors
in~\cite[Sec.~4.1]{CarliT13a}.

Finally, we are able to extend the copositive formulation of the
fractional chromatic number by Dukanovic and
Rendl~\cite{DukanovicR10a} to the weighted case, and we provide a
unified treatment of weighted generalizations of the convex quadratic
characterization of~\(\theta\) by Luz and Schrijver~\cite{LuzS05a} to
all of~\(\theta\), \(\theta'\), and~\(\theta^+\).

\subsection{Organization of the text}
\label{sec:organization}

The classical monograph~\cite{GroetschelLS93a}, which develops much of
the theory surrounding the theta function, defines weighted parameters
\(\theta_i(G;w)\) for each \(i \in \set{1,\dotsc,4}\) and shows that they are all
equal to~\(\theta(G;w)\) by proving the chain of inequalities
\begin{equation}
  \label{eq:theta-original-chain}
  \theta(G;w)
  \leq
  \theta_1(G;w)
  \leq
  \theta_2(G;w)
  \leq
  \theta_3(G;w)
  \leq
  \theta_4(G;w)
  \leq
  \theta(G;w)
\end{equation}
for a nonnegative weight function~\(w\) on~\(V\); see
also~\cite[Sec.~5]{Knuth94a}.  For instance, \(\theta_2\) is the
(weighted generalization of the) formulation on the RHS
of~\eqref{eq:lambda2-intro}.  In this paper, we shall generalize these
parameters to arbitrary theta bodies and prove that they are all equal
under some mild assumptions.  As in~\cite{GroetschelLS93a}, this shall
establish an antiblocking relation like those in~\eqref{eq:abl-intro}.

As we briefly hinted just before~\eqref{eq:STAB-chain}, we shall
parameterize our theta bodies using two cones, which we usually denote
by~\(\AdjCone\) and~\(\liftedQuadCone\); e.g., the membership
constraints \(X_{ij} = 0\) for every \(ij \in E\)
in~\eqref{eq:TH-intro} are replaced with the single membership
constraint ``\(X \in \AdjCone\)''.  One slightly confusing issue is
the fact that the cones~\(\AdjCone\) and~\(\liftedQuadCone\) do not
live in the same dimension:  \(\AdjCone\) lives in~\(\Sym{V}\), whereas
\(\liftedQuadCone\) lies in the higher-dimensional cone~\(\Sym{\setlift{V}}\), where \(0\) is
assumed not to be in~\(V\).  Throughout the paper, we label subsets
of~\(\Sym{\setlift{V}}\) with a wide hat, as in~\(\liftedQuadCone\),
and elements of such sets with a hat, e.g., \(\hat{X} \in
\liftedQuadCone\).

We shall see later that it is rather natural and convenient to define
a ``lifting'' of a cone~\(\QuadCone\) in~\(\Sym{V}\) to a
higher-dimensional cone in~\(\Sym{\setlift{V}}\) in a systematic way.
In fact, we shall define two such lifting operators for a cone \(\QuadCone
\subseteq \Sym{V}\), denoted by~\(\Psdlift{\cdot}\)
and~\(\Schurlift{\cdot}\), in such a way that the antiblocker of a
theta body parameterized by~\(\Psdlift{\QuadCone}\) is parameterized by
\(\Schurlift{\cdot}\) applied to the dual cone of~\(\QuadCone\).
Similarly, we will define a notion of duality for the cone
\(\AdjCone\), that encodes adjacencies, and the ``dual''
of~\(\AdjCone\) will be denoted by~\(\AdjDual{\AdjCone}\) to match the
occurrence of complements of graphs in~\eqref{eq:abl-intro}.

With these remarks in mind, we may now describe the organization of
the paper.  We list basic terminology and notation in
Subsection~\ref{sec:prelim}.  We then define arbitrary theta bodies
and prove their most basic properties in
Section~\ref{sec:theta-convex-corners}.  In
Section~\ref{sec:poly-diag-cones}, we study the structure of the
cones~\(\AdjCone\) that make sense in our theory, namely, the ones
that are diagonally scaling-invariant and polyhedral; we also quickly
develop their duality theory.  Next, in Section~\ref{sec:liftings}, we
describe the lifting operators for cones mentioned above, and we prove a weak duality theorem.  Some basic
reformulations of antiblocking duality are recalled and adapted to our
context in Section~\ref{sec:abl-reformulation}.  The latter three
sections also prove a number of equivalences among the~\(\theta_i\)'s
corresponding to~\eqref{eq:theta-original-chain}.  These equivalences
are put together in Section~\ref{sec:plethora} to prove our
generalization of the equivalence~\eqref{eq:theta-original-chain}, the
duality relations~\eqref{eq:abl-intro}, and further~\(\theta\)
results.  In Sections~\ref{sec:TH-Cop} and~\ref{sec:TH-Psd}, we study
properties of some theta bodies defined over some specific cones,
namely, the copositive and completely positive cones, and the
semidefinite cone.  In between those sections, we describe in
Section~\ref{sec:Hoffman} a further characterization of~\(\theta\)
related to Hoffman bounds for the chromatic number of a graph.

\subsection{Notation}
\label{sec:prelim}

We set the following notation.  Throughout the paper, \(V\) shall
denote a finite set.  We assume throughout that \(0 \not\in V\), as we
will often adjoin~\(0\) to~\(V\) to form an index set \(\setlift{V}\).
The family of subsets of~\(V\) of size~\(2\) is denoted
by~\(\tbinom{V}{2}\).  If \(E \subseteq \tbinom{V}{2}\), we set
\(\overline{E} \coloneqq \tbinom{V}{2} \drop E\).  For distinct \(i,j
\in V\), we denote \(ij \coloneqq \set{i,j}\).  The \emph{standard
  basis vectors} of~\(\Reals^V\) are \(\setst{e_i}{i \in V}\).  We
adopt \emph{Iverson notation}: for a predicate \(P\), we denote
\begin{equation*}
  \Iverson{P}
  \coloneqq
  \begin{cases*}
    1 & if \(P\) holds, \\
    0 & otherwise.      \\
  \end{cases*}
\end{equation*}
If \(P\) is false, then we consider \(\Iverson{P}\) to be ``strongly
zero,'' in the sense that we sometimes write expressions of the form
\(\Iverson{x \neq 0}(1/x)\) that evaluate to~\(0\) if \(x = 0\).

Most of the rest of our notation is listed over
tables~\ref{tbl:special-sets}, \ref{tbl:mat-vecs}, and \ref{tbl:conv}.

\bgroup
\renewcommand{\arraystretch}{1.2}
\begin{table}[htbp]
  \caption{Special sets.}
  \centering
  \begin{tabular}{r c p{11cm} }
    \toprule
    \(\Reals_+\) & \(\coloneqq\) & \(\setst{x \in \Reals}{x \geq 0}\)
    \\
    \(\Reals_{++}\) & \(\coloneqq\) & \(\setst{x \in \Reals}{x > 0}\)
    \\
    \(\Sym{V}\) & \(\coloneqq\) & the space of symmetric \(V \times
    V\) matrices
    \\
    \(\Psd{V}\) & \(\coloneqq\) & the cone of positive semidefinite
    matrices in~\(\Sym{V}\)
    \\[2pt]
    \(\Symnonneg{V}\) & \(\coloneqq\) & \(\setst{X \in \Sym{V}}{X \geq
      0}\), the cone of entrywise nonnegative matrices in~\(\Sym{V}\)
    \\
    \(\Copositive{V}\) & \(\coloneqq\) & \(\setst{X \in
      \Sym{V}}{\qform{X}{h} \geq 0\,\forall h \in \Reals_+^V}\), the cone
    of \emph{copositive matrices}
    \\
    \(\ComplPositive{V}\) & \(\coloneqq\) & the cone of \emph{completely
    positive matrices}, i.e., the dual cone of~\(\Copositive{V}\)
    \\
    \bottomrule
  \end{tabular}
  \label{tbl:special-sets}
\end{table}
\egroup                         % \arraystretch

\bgroup
\renewcommand{\arraystretch}{1.2}
\begin{table}[htbp]
  \caption{Notation for vectors and matrices.}
  \centering
  \begin{tabular}{r c p{11cm} }
    \toprule
    \(\iprod{X}{Y}\) & \(\coloneqq\) & \(\trace(XY^{\transp})\), the
    \emph{trace inner-product} on~\(\Sym{V}\)
    \\
    \(\diag\) & \(\coloneqq\) & the linear map that extracts the
    diagonal of a matrix
    \\
    \(\Diag\) & \(\coloneqq\) & the adjoint of \(\diag\)
    \\
    \(X[U]\) & \(\coloneqq\) & the principal submatrix of \(X \in
    \Reals^{V \times V}\) indexed by \(U \subseteq V\)
    \\
    \(\lambdamax(X)\) & \(\coloneqq\) & the largest eigenvalue of \(X
    \in \Sym{V}\)
    \\
    \(\lambdamin(X)\) & \(\coloneqq\) & the smallest eigenvalue of \(X
    \in \Sym{V}\)
    \\
    \(I\) & \(\coloneqq\) & the identity matrix in appropriate
    dimension
    \\
    \(\MPinverse{A}\) & \(\coloneqq\) &  the \emph{Moore-Penrose
      pseudoinverse} of~\(A \in \Reals^{V \times W}\);
    see~\cite{HornJ90a}
    \\
    \(\ones\) & \(\coloneqq\) & the vector of all-ones in the
    appropriate space
    \\
    \(\incidvector{U}\) & \(\coloneqq\) & the \emph{incidence vector}
    of \(U \subseteq V\) in \(\Reals^V\)
    \\
    \(\sqrt{w}\) & \(\coloneqq\) & the componentwise square root of
    \(w \in \Reals_+^V\), i.e., \(\component{\sqrt{w}\,}{i} \coloneqq
    \sqrt{w_i}\) for every \(i \in V\)
    \\
    \(x \oplus y\) & \(\coloneqq\) & the direct sum of vectors \(x \in
    \Reals^V\) and \(y \in \Reals^W\)
    \\
    \(\supp(x)\) & \(\coloneqq\) & \(\setst{i \in V}{x_i \neq 0}\),
    the \emph{support} of \(x \in \Reals^V\)
    \\
    \(x \hprod y\) & \(\coloneqq\) & the \emph{Hadamard product} of
    \(x, y \in \Reals^V\), i.e., \(\component{x \hprod y}{i} \coloneqq
    x_i y_i\) for every \(i \in V\)
    \\
    \bottomrule
  \end{tabular}
  \label{tbl:mat-vecs}
\end{table}
\egroup                         % \arraystretch

\bgroup
\renewcommand{\arraystretch}{1.2}
\begin{table}[htbp]
  \caption{Notation for Convex Analysis, mostly following~\cite{Rockafellar97a}.}
  \centering
  \begin{tabular}{r c p{11cm} }
    \toprule
    \(\Euclidean\) & & an \emph{Euclidean space}, i.e., a
    finite-dimensional real vector space equipped with an Euclidean
    inner-product
    \\
    \(\dual{\Euclidean}\) & & the \emph{dual space} of~\(\Euclidean\)
    \\
    \(\dual{\QuadCone}\) & \(\coloneqq\) & \(\setst[\big]{y \in
      \dual{\Euclidean}}{\iprod{x}{y} \geq 0\,\forall x \in
      \QuadCone}\), the \emph{dual cone} of a convex cone \(\QuadCone
    \subseteq \Euclidean\)
    \\
    \(\suppf{\ConvexSet}{w}\) & \(\coloneqq\) &
    \(\sup\setst[\big]{\iprod{w}{x}}{x \in \ConvexSet}\), the
    \emph{support function} of \(\ConvexSet \subseteq \Euclidean\)
    defined for each \(w \in \Euclidean^*\)
    \\
    \(\abl{\ConvexSet}\) & \(\coloneqq\) & \(\setst{y \in
      \Reals_+^V}{\iprod{y}{x} \leq 1\,\forall x \in \ConvexSet}\),
    the \emph{antiblocker} of \(\ConvexSet \subseteq \Reals_+^V\)
    \\
    \(\cl(\ConvexSet)\) & \(\coloneqq\) & the \emph{closure} of
    \(\ConvexSet \subseteq \Euclidean\)
    \\
    \(\interior(\ConvexSet)\) & \(\coloneqq\) & the \emph{interior} of
    \(\ConvexSet \subseteq \Euclidean\)
    \\
    \(\Aut(\ConvexSet)\) & \(\coloneqq\) & \(\setst{\Tcal \ffrom
      \Euclidean \fto \Euclidean}{\Tcal\text{ nonsingular linear
        map},\, \Tcal(\ConvexSet) = \ConvexSet}\), the
    \emph{automorphism group} of \(\ConvexSet \subseteq \Euclidean\)
    \\
    \bottomrule
  \end{tabular}
  \label{tbl:conv}
\end{table}
\egroup                         % \arraystretch

We will often consider subsets of \(\Sym{V}\)
and~\(\Sym{\setlift{V}}\).  Subsets of the latter, as well as their
elements, shall be decorated with a hat, e.g., \(\Xh \in
\liftedQuadCone \subseteq \Sym{\setlift{V}}\).  If \(\QuadCone
\subseteq \Euclidean\) is a pointed closed convex cone with nonempty
interior, then \(\QuadCone\) defines a partial
order~\(\succeq_{\QuadCone}\) on~\(\Euclidean\): we write \(x
\succeq_{\QuadCone} y\) to mean that \(x - y \in \QuadCone\).  We set
\(\succeq\, \coloneqq\: \succeq_{\Psd{V}}\).

\section{Theta Bodies}
\label{sec:theta-convex-corners}

\bgroup
\newcommand*{\Cone}{\mathbb{K}}
For each \(h \in \Reals^V\), define the \emph{diagonal scaling map}
\(\DiagScaleMap{h} \ffrom \Reals^{V \times V} \fto \Reals^{V \times
  V}\) as
\begin{equation}
  \label{eq:diag-scale-map}
  \DiagScaleMap{h}(X)
  \coloneqq
  \Diag(h) X \Diag(h)
  \qquad
  \forall X \in \Reals^{V \times V}.
\end{equation}
Note that each entry \(\component{\DiagScaleMap{h}(X)}{ij}\) is the
componentwise product \(X_{ij}[\oprodsym{h}]_{ij}\) for every
\(X \in \Sym{V}\) and \(h \in \Reals^V\).  A subset~\(\Cone\)
of~\(\Sym{V}\) is called \emph{diagonally scaling-invariant} if
\(\DiagScaleMap{h}(\Cone) \subseteq \Cone\) for every \(h \in
\Reals_+^V\).  The cones \(\Psd{V}\), \(\Symnonneg{V}\),
\(\Copositive{V}\), \(\ComplPositive{V}\) are all examples of
diagonally scaling-invariant subsets of~\(\Sym{V}\).  Some other
important examples are the sets of the form
\begin{equation}
  \label{eq:poly-adj-cone}
  \PolyAdjCone{E^+}{E^-}
  \coloneqq
  \setst{
    X \in \Sym{V}
  }{
    X_{ij} \geq 0\,\forall ij \in E^+,\:
    X_{ij} \leq 0\,\forall ij \in E^-
  },
\end{equation}
where \(E^+,E^- \subseteq \tbinom{V}{2}\).  Clearly, every diagonally
scaling-invariant set is a cone, and since the map
\(\DiagScaleMap{h}\) is self-adjoint, diagonal scaling invariance is
preserved under duality.  Moreover,
    if \(\Cone \subseteq \Sym{V}\) is diagonally scaling-invariant,
    then
    \(\setst{\DiagScaleMap{h}}{h \in \Reals_{++}^V}
    \subseteq \Aut(\Cone)\).

For sets \(\AdjCone \subseteq \Sym{V}\) and \(\liftedQuadCone
\subseteq \Sym{\setlift{V}}\), define
\begin{equation}
  \label{eq:liftedGenTH}
  \liftedGenTH{\AdjCone}{\liftedQuadCone}
  \coloneqq
  \setst*{
    \Xh \in \liftedQuadCone
  }{
    \Xh_{00} = 1,\,
    \Xh e_0 = \diag(\Xh),\,
    \Xh[V] \in \AdjCone
  }
\end{equation}
and
\begin{equation}
  \label{eq:genTH}
  \genTH{\AdjCone}{\liftedQuadCone}
  \coloneqq
  \setst*{
    \diag(\Xh[V])
  }{
    \Xh \in \liftedGenTH{\AdjCone}{\liftedQuadCone}
  }.
\end{equation}
We are interested in sets of the form
\(\genTH{\AdjCone}{\liftedQuadCone}\), where \(\AdjCone\)
and~\(\liftedQuadCone\) are diagonally scaling-invariant convex cones
with a few extra properties.  The most important known examples of
sets of this form are the theta body~\(\TH(G)\) of a graph \(G =
(V,E)\) and its variants~\(\TH'(G)\) and~\(\TH^+(G)\).  In fact, we
define
\begin{subequations}
  \label{eq:THs-as-genTHs}
  \begin{gather}
    \label{eq:TH-as-genTH}
    \TH(G)
    \coloneqq
    \genTH[\big]{
      \PolyAdjCone{E}{E}
    }{
      \Psd{\setlift{V}}
    },
    \\
    \label{eq:TH-as-genTH'}
    \TH'(G)
    \coloneqq
    \genTH[\big]{
      \PolyAdjCone{E \cup \overline{E}}{E}
    }{
      \Psd{\setlift{V}}
    },
    \\
    \label{eq:TH-as-genTH+}
    \TH^+(G)
    \coloneqq
    \genTH[\big]{
      \PolyAdjCone{\emptyset}{E}
    }{
      \Psd{\setlift{V}}
    }.
  \end{gather}
\end{subequations}
It thus makes sense to call sets of the
form~\(\genTH{\AdjCone}{\liftedQuadCone}\) as \emph{theta bodies} (the
terminology ``theta bodies'' was also used for another generalization
of the theta body by \cite{GouveiaPT10a}; our definition and approach
are very different).  To avoid
confusion, whenever we refer to the specific theta body~\(\TH(G)\), we
shall call it \emph{the} theta body \emph{of~\(G\)}.
\egroup

In the remainder of this section we shall prove that, under certain
simple hypotheses, every theta body is a \emph{convex corner}, i.e., a
compact, lower-comprehensive convex subset of the nonnegative orthant
with nonempty interior.  Recall that a subset \(\ConvexSet\)
of~\(\Reals_+^V\) is called \emph{lower-comprehensive} if, for any
\(x,y \in \Reals^V\), the chain of relations \(0 \leq y \leq x \in
\ConvexSet\) implies \(y \in \ConvexSet\).  In what follows, the extra
hypotheses~\eqref{eq:A-big-enough} and~\eqref{eq:TH-AK-hypotheses-K}
on~\(\AdjCone\) and~\(\liftedQuadCone\) may be thought of as requiring
that~\(\AdjCone\) is not ``too small'', and that~\(\liftedQuadCone\)
is neither ``too small'' nor~``too big.''
\begin{proposition}
  \label{prop:TH-AK-almost-convex-corner}
  Let \(\AdjCone \subseteq \Sym{V}\) and \(\liftedQuadCone \subseteq
  \Sym{\setlift{V}}\) be diagonally scaling-invariant closed convex
  cones.  Suppose that \(\AdjCone\) satisfies
  \begin{equation}
    \label{eq:A-big-enough}
    \Image(\Diag) \subseteq \AdjCone,
  \end{equation}
  and suppose that \(\liftedQuadCone\) satisfies
  \begin{subequations}
    \label{eq:TH-AK-hypotheses-K}
    \begin{gather}
      \label{eq:lifted-K-big-enough}
      \liftedQuadCone
      \supseteq
      \setst*{
        \oprodsym{(e_0 + e_i)}
      }{
        i \in V
      }
      \shortintertext{and}
      \label{eq:lifted-K-not-too-big}
      \diag\paren[\big]{
        \setst[\big]{
          \Xh \in \liftedQuadCone
        }{
          \Xh_{00} = 1,\,
          \Xh e_0 = \diag(\Xh)
        }
      }
      \subseteq
      [0,1]^{\setlift{V}}.
    \end{gather}
  \end{subequations}
  Then \(\genTH{\AdjCone}{\liftedQuadCone}\) is a convex
  lower-comprehensive subset of~\([0,1]^V\) with nonempty interior.
  In particular,
  \(\cl\paren[\big]{\genTH{\AdjCone}{\liftedQuadCone}}\) is a convex
  corner.
\end{proposition}

\begin{proof}
  Convexity of the projection \(\genTH{\AdjCone}{\liftedQuadCone}\)
  follows from that of~\(\liftedGenTH{\AdjCone}{\liftedQuadCone}\).
  It is clear from~\eqref{eq:lifted-K-not-too-big} that
  \(\genTH{\AdjCone}{\liftedQuadCone} \subseteq [0,1]^V\).  To prove
  that the convex set~\(\genTH{\AdjCone}{\liftedQuadCone}\) is lower-comprehensive,
  it suffices to
  show that if \(x \in \genTH{\AdjCone}{\liftedQuadCone}\) then \(x -
  x_i e_i \in \genTH{\AdjCone}{\liftedQuadCone}\) for each \(i \in
  V\).  Let \(\Xh \in \liftedGenTH{\AdjCone}{\liftedQuadCone}\) such
  that \(x = \diag(\Xh[V])\).  Let \(i \in V\).  Set \(\Yh \coloneqq
  \DiagScaleMap{1 \oplus h}(\Xh) \in \liftedQuadCone\) for \(h
  \coloneqq \ones - e_i \in \Reals^V\).  Then diagonal scaling
  invariance of~\(\AdjCone\) and~\(\liftedQuadCone\) imply that \(\Yh
  \in \liftedGenTH{\AdjCone}{\liftedQuadCone}\).  Thus, \(x - x_i e_i
  = \diag(\Yh[V]) \in \genTH{\AdjCone}{\liftedQuadCone}\).  This
  proves that \(\genTH{\AdjCone}{\liftedQuadCone}\) is
  lower-comprehensive.  It remains to show that
  \(\genTH{\AdjCone}{\liftedQuadCone}\) has nonempty interior.  Let
  \(i \in V\).  By~\eqref{eq:A-big-enough}, we have \(\oprodsym{e_i}
  \in \AdjCone\).  Thus, \(\oprodsym{(e_0+e_i)} \in
  \liftedGenTH{\AdjCone}{\liftedQuadCone}\)
  by~\eqref{eq:lifted-K-big-enough} whence \(e_i \in
  \genTH{\AdjCone}{\liftedQuadCone}\).  Now convexity of
  \(\genTH{\AdjCone}{\liftedQuadCone}\) implies that \(\frac{1}{n}
  \ones \in \genTH{\AdjCone}{\liftedQuadCone}\), where \(n \coloneqq
  \card{V}\).  Since \(\genTH{\AdjCone}{\liftedQuadCone}\) is
  lower-comprehensive, we find that \(\frac{1}{2n} \ones \in
  \interior(\genTH{\AdjCone}{\liftedQuadCone})\).
\end{proof}

The reason for using the hypothesis~\eqref{eq:A-big-enough} rather
than the slightly weaker condition \(\Diag(\Reals_+^V) \subseteq \AdjCone\)
shall be made clearer in Section~\ref{sec:poly-diag-cones}, where we
develop a notion of duality for such cones.  As for the closedness of
theta bodies, we are not aware of any example of a theta body that is
not closed, so the closure operator in
Proposition~\ref{prop:TH-AK-almost-convex-corner} shall remain as a
minor nuisance.  We shall now see that, under a mild condition on the
cone~\(\liftedQuadCone\), the theta body
\(\genTH{\AdjCone}{\liftedQuadCone}\) is actually closed, and hence a
convex corner itself.
\begin{corollary}
  \label{cor:TH-AK-convex-corner}
  Let \(\AdjCone \subseteq \Sym{V}\) and \(\liftedQuadCone \subseteq
  \Sym{\setlift{V}}\) be diagonally scaling-invariant closed convex
  cones such that~\eqref{eq:A-big-enough}
  and~\eqref{eq:lifted-K-big-enough} hold.  If
  \begin{equation}
    \label{eq:lifted-K-not-too-big-suff}
    \liftedQuadCone
    \subseteq
    \setst*{
      \Xh \in \Sym{\setlift{V}}
    }{
      \Xh[S] \succeq 0,\,
      \forall S \in \tbinom{\setlift{V}}{2}
    },
  \end{equation}
  then \eqref{eq:lifted-K-not-too-big} holds and
  \(\liftedGenTH{\AdjCone}{\liftedQuadCone}\) is compact.  In
  particular, \(\genTH{\AdjCone}{\liftedQuadCone}\) is closed, and
  hence a convex corner.
\end{corollary}

\begin{proof}
  \newcommand*{\Cone}{\widehat{\mathbb{M}}}
  Let \(\Cone\) be the set of all~\(\Xh\) in the RHS
  of~\eqref{eq:lifted-K-not-too-big-suff} such that \(\Xh_{00} = 1\)
  and \(\Xh e_0 = \diag(\Xh)\).  Then~\(\Cone\) is bounded.  To see
  this, first use sets \(S \in \tbinom{\setlift{V}}{2}\)
  containing~\(0\) to show that \(\diag(\Cone) \subseteq
  [0,1]^{\setlift{V}}\).  Thus, \eqref{eq:lifted-K-not-too-big} holds.  Next,
  use~\eqref{eq:lifted-K-not-too-big-suff} with sets \(S \in
  \tbinom{V}{2}\) to show that all off-diagonal entries of \(\Xh \in
  \Cone\) have absolute value bounded above by~\(1\).  Since
  \(\liftedGenTH{\AdjCone}{\liftedQuadCone} \subseteq \Cone\), it
  follows that \(\liftedGenTH{\AdjCone}{\liftedQuadCone}\) is compact.
  Now closedness of \(\genTH{\AdjCone}{\liftedQuadCone}\) follows from
  the fact that \(\genTH{\AdjCone}{\liftedQuadCone}\) is a linear
  image of the compact set
  \(\liftedGenTH{\AdjCone}{\liftedQuadCone}\).  The rest follows from
  Proposition~\ref{prop:TH-AK-almost-convex-corner}.
\end{proof}

For most of the theta bodies in this paper, the cone
\(\liftedQuadCone\) shall be a subset of \(\Psd{\setlift{V}}\), and
hence~\eqref{eq:lifted-K-not-too-big-suff} shall be satisfied; this
already includes the theta bodies described
in~\eqref{eq:THs-as-genTHs}.
An important diagonally scaling-invariant closed convex cone which
does not satisfy~\eqref{eq:lifted-K-not-too-big-suff} is the cone of
copositive matrices; indeed, note that \(\Copositive{\setlift{V}} \supseteq \Symnonneg{\setlift{V}}\).  We shall deal with
theta bodies arising from the copositive cone in
Section~\ref{sec:TH-Cop}, where we shall prove directly that the
corresponding theta body is closed.

\section{Polyhedral Diagonally Scaling-Invariant Cones}
\label{sec:poly-diag-cones}

When studying a theta body \(\genTH{\AdjCone}{\liftedQuadCone}\), we
think of~\(\AdjCone\) as an ``elementary'' cone,
while~\(\liftedQuadCone\) is (potentially) a ``sophisticated'' cone.
In the most important instances of theta bodies, namely the ones
described in~\eqref{eq:THs-as-genTHs}, the cone~\(\AdjCone\) is
polyhedral, whereas \(\liftedQuadCone\) is the nonlinear
cone~\(\Psd{\setlift{V}}\).  In general, it makes sense to focus on
the case where~\(\AdjCone\) is polyhedral.  At any rate, when defining
a theta body \(\genTH{\AdjCone}{\liftedQuadCone}\), any trace of
``non-polyhedrality'' may be (and should be) ``pushed'' away
from~\(\AdjCone\) and into~\(\liftedQuadCone\).  We shall show next
that requiring a closed convex cone to be both diagonally
scaling-invariant and polyhedral severely constrains its structure.

We shall need a family of cones slightly more refined than the
cones~\(\PolyAdjCone{E^+}{E^-}\) defined in~\eqref{eq:poly-adj-cone}.
Let \(V^+,V^- \subseteq V\) and \(E^+, E^- \subseteq \binom{V}{2}\).
Define
\begin{equation*}
  \PolyAdjConeNode{V^+}{V^-}{E^+}{E^-}
  \coloneqq
  \setst*{
    X \in \PolyAdjCone{E^+}{E^-}
  }{
    \diag(X[V^+]) \geq 0,\,
    \diag(X[V^-]) \leq 0
  }.
\end{equation*}
Clearly, every set of this form is diagonally scaling-invariant and
polyhedral.  In fact, every polyhedral diagonally scaling-invariant
cone is of this form:
\begin{proposition}
  \label{prop:A-Diag-closed-polyhedral}
  Let \(\AdjCone \subseteq \Sym{V}\) be a diagonally scaling-invariant
  closed convex cone.  If~\(\AdjCone\) is polyhedral, then
  \(\AdjCone\) is of the form \(\AdjCone =
  \PolyAdjConeNode{V^+}{V^-}{E^+}{E^-}\) for some subsets \(V^+, V^-
  \subseteq V\) and \(E^+, E^- \subseteq \tbinom{V}{2}\).
\end{proposition}
\begin{proof}
  Let \(\Symmetrize \ffrom \Reals^{V \times V} \fto \Sym{V}\) denote
  the orthogonal projection onto~\(\Sym{V}\), that is,
  \begin{equation}
    \label{eq:def-symmetrize}
    \Symmetrize(X)
    \coloneqq
    \thalf(X+X^{\transp})
    \qquad
    \forall X \in \Reals^{V \times V}.
  \end{equation}
  It suffices to show that
  \begin{equation}
    \label{eq:A-Diag-closed-polyhedral-aux1}
    \text{every extreme ray of }
    \dual{\AdjCone\!}
    \text{ is of the form }
    \pm\Reals_+\Symmetrize\paren{\oprod{e_i}{e_j}}
    \text{ for some }
    i,j \in V.
  \end{equation}
  We first show that,
  \begin{equation}
    \label{eq:A-Diag-closed-polyhedral-aux2}
    \text{%
      if \(\Reals_+X\) is an extreme ray of \(\dual{\AdjCone\!}\),
      then \(\card[\big]{\supp\paren[\big]{\diag(X)}} \leq 1\).
    }
  \end{equation}
  Suppose that \(\Reals_+X\) is an extreme ray
  of~\(\dual{\AdjCone\!}\) such that \(X_{ii} \neq 0 \neq X_{jj}\) for
  distinct \(i,j \in V\).  Since~\(\dual{\AdjCone\!}\) is also
  diagonally scaling-invariant, we have \(\setst{\Dcal_h}{h \in
    \Reals_{++}^V} \subseteq \Aut(\dual{\AdjCone\!})\).  For \(t \in
  \Reals_{++}\), define \(h(t) \coloneqq t e_i + t^{-1} e_j +
  \incidvector{V \drop \set{i,j}}\).  Thus, \(\setst{\Reals_+
    \DiagScaleMap{h(t)}(X)}{t \in \Reals_{++}}\) is an infinite set of
  extreme rays of~\(\dual{\AdjCone\!}\).  This contradicts the fact
  that~\(\dual{\AdjCone\!}\) is polyhedral and thus
  proves~\eqref{eq:A-Diag-closed-polyhedral-aux2}.

  To prove~\eqref{eq:A-Diag-closed-polyhedral-aux1}, let \(\Reals_+X\)
  be an extreme ray of~\(\dual{\AdjCone\!}\).  Let us show that
  \begin{equation}
    \label{eq:A-Diag-closed-polyhedral-aux3}
    X_{ij} \Symmetrize\paren{\oprod{e_i}{e_j}}
    \in \dual{\AdjCone\!}
    \qquad
    \forall i,j \in V.
  \end{equation}
  Let \(i,j \in V\).  If \(i = j\)
  then~\eqref{eq:A-Diag-closed-polyhedral-aux3} holds by diagonal
  scaling invariance of~\(\AdjCone\), so assume \(i \neq j\).
  By~\eqref{eq:A-Diag-closed-polyhedral-aux2}, at most one
  of~\(X_{ii}\) and~\(X_{jj}\) is nonzero.  We may assume by symmetry
  that \(X_{ii} = 0\).  For \(t \in \Reals_{++}\), define \(h(t)
  \coloneqq t e_i + t^{-1} e_j\) and note that
  \(\DiagScaleMap{h(t)}(X) \in \dual{\AdjCone\!}\) for every \(t \in
  \Reals_{++}\).  By driving~\(t\) to~\(\infty\) we find that \(2
  X_{ij} \Symmetrize\paren{\oprod{e_i}{e_j}} = \lim_{t\to \infty}
  \DiagScaleMap{h(t)}(X)\) lies in the closed
  set~\(\dual{\AdjCone\!}\).  This
  proves~\eqref{eq:A-Diag-closed-polyhedral-aux3}.

  Since
  \begin{equation}
    \label{eq:sym-decomp}
    X
    =
    \sum_{i,j \in V}
    2^{\Iverson{i \neq j}}
    X_{ij} \Symmetrize\paren{\oprod{e_i}{e_j}}
  \end{equation}
  and \(\Reals_+X\) is an extreme ray of~\(\dual{\AdjCone\!}\), it
  follows from~\eqref{eq:A-Diag-closed-polyhedral-aux3} that at most
  one of the terms in the RHS of~\eqref{eq:sym-decomp} is nonzero.
  This proves~\eqref{eq:A-Diag-closed-polyhedral-aux1} and concludes
  the proof.
\end{proof}

It follows that a polyhedral diagonally scaling-invariant cone
\(\AdjCone \subseteq \Sym{V}\) such that \(\Image(\Diag) \subseteq
\AdjCone\) must have the form shown in~\eqref{eq:poly-adj-cone}.  We
shall focus our attention on such cones for the first argument
of~\(\widehat{\TH}\) for the remainder of the paper.

We shall define some alternative notions of ``duality'' for the cones
\(\AdjCone\) and \(\liftedQuadCone\) to describe more conveniently the
antiblockers of theta bodies.  The appropriate duality notion for the
cones \(\AdjCone \subseteq \Sym{V}\) is defined as follows:
\begin{equation}
  \label{eq:dual-adjcone}
  \AdjDual{\AdjCone} \coloneqq \Image(\Diag) - \dual{\AdjCone\!}.
\end{equation}
It is easy to check that, for every \(E^+, E^- \subseteq
\tbinom{V}{2}\), we have
\begin{equation}
  \AdjDual{\PolyAdjCone{E^+}{E^-}}
  =
  \PolyAdjCone{\overline{E^+}}{\overline{E^-}}
\end{equation}
and, for every polyhedral diagonally scaling-invariant cone \(\AdjCone
\subseteq \Sym{V}\), we have \(\AdjDual{\AdjDual{\AdjCone}} = \AdjCone
+ \Image(\Diag)\), so this operation is an involution when restricted
to cones that contain \(\Image(\Diag)\).

\section{Liftings of Cones}
\label{sec:liftings}

In this section, we define two operators that lift a cone
in~\(\Sym{V}\) to a cone in~\(\Sym{\setlift{V}}\), as we briefly
mentioned in Subsection~\ref{sec:organization}.  Note that, in the
definition~\eqref{eq:liftedGenTH}, whenever we test membership of a
matrix~\(\hat{X}\) in~\(\liftedQuadCone\), the 0th column
of~\(\hat{X}\) is completely determined by~\(\hat{X}[V]\).  Thus, it
makes some sense to define a lifting of a cone~\(\QuadCone\)
in~\(\Sym{V}\) as a cone in~\(\Sym{\setlift{V}}\) in such a way that
the \(0\)th column is strongly related to~\(\QuadCone\).  We shall
need two such liftings, which are studied in the next two subsections.
We shall also see a glimpse of the duality relation involving these
two liftings in a Weak Duality theorem, as well as a natural description of the set
\(\FRAC(G)\) (that appeared in~\eqref{eq:STAB-chain}) as a theta body using one of
the liftings.

\subsection{PSD Liftings of Cones}
\label{sec:psd-liftings}

In this subsection, we define the PSD lifting of a cone~\(\QuadCone\)
in~\(\Sym{V}\) and we prove that, under mild hypotheses, the support
function of a theta body defined over the PSD lifting of \(\QuadCone\)
can be formulated as a simple conic optimization problem
over~\(\QuadCone\).  This shall correspond to the equivalence
\(\theta_3 = \theta_4\) in~\eqref{eq:theta-original-chain}.

Let \(\QuadCone \subseteq \Sym{V}\).  Define the \emph{PSD lifting
  of~\(\QuadCone\)} as
\begin{equation}
  \label{eq:psd-lift}
  \Psdlift{\QuadCone}
  \coloneqq
  \setst*{
    \Xh \in \Psd{\setlift{V}}
  }{
    \Xh[V] \in \QuadCone
  }.
\end{equation}
Note that if \(\QuadCone\) is diagonally scaling-invariant, then so
is~\(\Psdlift{\QuadCone}\).  Moreover, \(\Psdlift{\Psd{V}} =
\Psd{\setlift{V}}\).

Before using PSD liftings, we shall need the following straightforward
weighted generalization of \cite[Proposition~9]{Gijswijt05a}.
\bgroup
\newcommand*{\Cone}{\mathbb{M}}
\newcommand*{\Xopt}[1][]{X_{#1}^*}
\begin{lemma}
  \label{lemma:gijswijt-lambda-trick}
  Let \(\Cone \subseteq \Sym{V}\) be a diagonally scaling-invariant
  closed convex cone.  Suppose that
  \begin{subequations}
    \label{eq:gijswijt-lambda-trick-S-hyp}
    \begin{gather}
      \label{eq:gijswijt-lambda-trick-S-hyp1}
      \diag(\Cone) \subseteq \Reals_+^V,
      \\
      \label{eq:gijswijt-lambda-trick-S-hyp2}
      \text{if }X_{ii} = 0
      \text{ for some }X \in \Cone
      \text{ and }i \in V,
      \text{ then }X e_i = 0,
      \\
      \label{eq:gijswijt-lambda-trick-S-hyp3}
      \setst{X \in \Cone}{\trace(X) = 1}
      \text{ is compact}.
    \end{gather}
  \end{subequations}
  Let \(w \in \Reals_+^V\).  Let \(\Xopt\) be an optimal solution of
  \begin{equation}
    \label{eq:gijswijt-opt-problem}
    \max\setst*{
      \qform{X}{\sqrt{w}}
    }{
      \trace(X) = 1,\, X \in \Cone
    },
  \end{equation}
  and suppose that \(\qform{\Xopt}{\sqrt{w}} > 0\).  Set
  \begin{gather*}
    d
    \coloneqq
    \diag(\Xopt),
    \\
    \Xb
    \coloneqq
    \MPinverse{\Diag(\sqrt{d})} \Xopt \MPinverse{\Diag(\sqrt{d})},
    \\
    \lambda
    \coloneqq
    \lambdamax(\DiagScaleMap{\sqrt{w}}(\Xb)).
  \end{gather*}
  Then
  \begin{subequations}
    \label{eq:gijswijt-lambda-trick}
    \begin{gather}
      \label{eq:gijswijt-lambda-trick1}
      \supp(d) \subseteq \supp(w),
      \\
      \label{eq:gijswijt-lambda-trick2}
      \DiagScaleMap{\sqrt{w}} (\Xb)
      \sqrt{d}
      =
      \lambda
      \sqrt{d},
      \\
      \label{eq:gijswijt-lambda-trick3}
      \lambda
      =
      \qform{\Xopt}{\sqrt{w}},
      \\
      \label{eq:gijswijt-lambda-trick4}
      \Xopt \sqrt{w}
      =
      \lambda
      \MPinverse{
        \Diag(\sqrt{w})
      }
      d.
    \end{gather}
  \end{subequations}
\end{lemma}
\begin{proof}
  We first show~\eqref{eq:gijswijt-lambda-trick1}.  Let \(i \in
  \supp(d)\), so that \(\Xopt[ii] > 0\).  Suppose that \(w_i = 0\).
  If \(\Xopt[ii] = 1\), then \(\Xopt = \oprodsym{e_i}\)
  by~\eqref{eq:gijswijt-lambda-trick-S-hyp} whence
  \(\qform{\Xopt}{\sqrt{w}} = 0\).  If \(\Xopt[ii] < 1\), then
  \((1-\Xopt[ii])^{-1} \DiagScaleMap{\ones - e_i}(\Xopt)\) is feasible
  for~\eqref{eq:gijswijt-opt-problem} with objective value
  \((1-\Xopt[ii])^{-1} \qform{\Xopt}{\sqrt{w}}\), hence strictly
  larger than the objective value of~\(\Xopt\).  In either case, we
  get a contradiction.  This proves~\eqref{eq:gijswijt-lambda-trick1}.

  If \(d_i = 0\) for some \(i \in V\), we are done by induction
  on~\(\card{V}\).  Thus, from~\eqref{eq:gijswijt-lambda-trick1} we
  may assume that
  \begin{equation}
    \label{eq:theta3-opt-diag-aux2}
    \supp(d) = \supp(w) = V.
  \end{equation}
  Define \(d^{-1/2} \in \Reals^V\) such that \(d^{-1/2} \hprod
  \sqrt{d} = \ones\), so \(\Xb = \DiagScaleMap{d^{-1/2}}(\Xopt) \in
  \Cone\) and \(\diag(\Xb) = \ones\).  For every \(h \in \Reals_+^V\)
  with \(\norm{h} = 1\), the point \(\DiagScaleMap{h}(\Xb)\) is
  feasible for~\eqref{eq:gijswijt-opt-problem} with objective value
  \(\qform{\DiagScaleMap{h}(\Xb)}{\sqrt{w}} =
  \qform{\DiagScaleMap{\sqrt{w}}(\Xb)}{h}\).  Since \(\Xopt =
  \DiagScaleMap{\sqrt{d}}\paren{\Xb}\) is optimal
  for~\eqref{eq:gijswijt-opt-problem}, it follows that \(\sqrt{d}\) is
  an optimal solution for
  \(\max\setst*{\qform{\DiagScaleMap{\sqrt{w}}(\Xb)}{h}}{h \in
    \Reals_+^V,\, \norm{h} = 1}\).  In fact, since
  \(\component{\sqrt{d}\thinspace}{i} > 0\) for all \(i \in V\), we
  find that \(\sqrt{d}\) is a local optimal solution for
  \(\max\setst*{\qform{\DiagScaleMap{\sqrt{w}}(\Xb)}{h}}{h \in
    \Reals^V,\, \norm{h} = 1}\), hence also a global one (note that
  the sign of \(h\) is unconstrained here).  Thus,
  \(\DiagScaleMap{\sqrt{w}}(\Xb) \sqrt{d} = \lambda \sqrt{d}\).  This
  proves~\eqref{eq:gijswijt-lambda-trick2}.  Now we unroll:
  \begin{equation*}
    \begin{split}
      \lambda d
      & =
      \lambda \Diag(\sqrt{d}) \sqrt{d}
      =
      \Diag(\sqrt{d})
      \DiagScaleMap{\sqrt{w}} \paren*{
        \DiagScaleMap{d^{-1/2}}(\Xopt)
      }
      \sqrt{d}
      =
      \Diag(\sqrt{d})
      \DiagScaleMap{d^{-1/2}} \paren*{
        \DiagScaleMap{\sqrt{w}}(\Xopt)
      }
      \sqrt{d}
      \\
      & =
      \Diag(\sqrt{d})
      \DiagScale{
        \DiagScaleMap{\sqrt{w}}(\Xopt)
      }{d^{-1/2}}
      \sqrt{d}
      =
      \DiagScale{\Xopt}{\sqrt{w}} \ones
      =
      \Diag(\sqrt{w}) \Xopt \sqrt{w}.
    \end{split}
  \end{equation*}
  This proves~\eqref{eq:gijswijt-lambda-trick4}.  Finally, \(\lambda =
  \lambda \trace(\Xopt) = \lambda \iprodt{\ones}{d} = \ones^{\transp}
  \Diag(\sqrt{w}) \Xopt \sqrt{w} = \qform{\Xopt}{\sqrt{w}}\)
  so~\eqref{eq:gijswijt-lambda-trick3} is proved.
\end{proof}
\egroup

We can now show that the support function of some theta bodies of the
form \(\genTH{\AdjCone}{\Psdlift\QuadCone}\), which shall correspond
to~\(\theta_4\) in~\eqref{eq:theta-original-chain}, may be formulated
as a conic optimization problem over~\(\QuadCone\); the latter shall
correspond to~\(\theta_3\) from~\eqref{eq:theta-original-chain}.  Note
that the next result does not make use of Duality Theory.
\begin{theorem}
  \label{thm:theta3-theta4}
  Let \(\AdjCone \subseteq \Sym{V}\) and~\(\QuadCone \subseteq
  \Psd{V}\) be diagonally scaling-invariant closed convex cones such that \(\Image(\Diag) \subseteq \AdjCone\) and \(\Diag(\Reals_+^V) \subseteq \QuadCone\).  Let
  \(w \in \Reals_+^V\).  Then
  \begin{equation}
    \label{eq:theta3-theta4}
    \suppf{
      \genTH{\AdjCone}{\Psdlift{\QuadCone}}
    }{
      w
    }
    =
    \max\setst*{
      \iprod{
        \oprodsym{\sqrt{w}}
      }{
        X
      }
    }{
      \iprod{I}{X} = 1,\,
      X \in \AdjCone,\,
      X \in \QuadCone
    }.
  \end{equation}
  Moreover, both optimization problems in~\eqref{eq:theta3-theta4} have optimal solutions.
\end{theorem}

\begin{proof}
  We begin by proving~`\(\leq\)'.  Let \(y \in
  \genTH{\AdjCone}{\Psdlift{\QuadCone}}\) and let \(\Yh \in
  \liftedGenTH{\AdjCone}{\Psdlift{\QuadCone}}\) such that \(y =
  \diag(Y)\) for \(Y \coloneqq \Yh[V]\).  We will show that there
  exists a feasible solution~\(X\) for the RHS
  of~\eqref{eq:theta3-theta4} with objective value at
  least~\(\iprod{w}{y}\).  We may assume that \(\iprod{w}{y} > 0\);
  otherwise, take \(X = \oprodsym{e_i}\) for any \(i \in V\).  Set \(h
  \coloneqq \iprod{w}{y}^{-1/2} \sqrt{w} \geq 0\) and \(X \coloneqq
  \DiagScaleMap{h}(Y) \in \AdjCone \cap \QuadCone\).  Then \(\trace(X)
  = {\iprod{w}{y}}^{-1} \trace\paren{ \DiagScaleMap{\sqrt{w}}(Y) } =
  {\iprod{w}{y}}^{-1} \iprod{ \sqrt{w} \hprod \sqrt{w} }{ \diag(Y) } =
  1\), whence \(X\) is feasible on the RHS
  of~\eqref{eq:theta3-theta4}.  Moreover,
  \begin{equation*}
    \begin{bmatrix}
      1          & (h \hprod y)^{\transp}\thinspace \\
      h \hprod y & X                                \\
    \end{bmatrix}
    =
    \DiagScaleMap{1 \oplus h}
    \paren*{\onelift{y}{Y}}
    \in
    \DiagScaleMap{1 \oplus h}
    \paren*{\Psdlift{\QuadCone}}
    \subseteq
    \Psdlift{\QuadCone}
    \subseteq
    \Psd{\setlift{V}}.
  \end{equation*}
  Thus, by Schur complement, we get \(X \succeq
  \iprod{w}{y}^{-1}\oprodsym{\paren{\sqrt{w} \hprod y}}\) and so
  \begin{equation*}
    \qform{X}{\sqrt{w}}
    \geq
    {\iprod{w}{y}}^{-1}
    \qform{
      \oprodsym{
        \paren*{
          \Diag(\sqrt{w}\thinspace) y
        }
      }
    }{
      \sqrt{w}
    }
    =
    \frac{1}{\iprod{w}{y}}
    \iprod{w}{y}^2.
  \end{equation*}
  This completes the proof of~`\(\leq\)'.

  Now we prove~`\(\geq\)'.  For that, we will show that,
  \begin{equation}
    \label{eq:theta3-theta4-aux1}
    \text{if }X \in \AdjCone \cap \QuadCone
    \text{ and }X\sqrt{w} \geq 0,
    \text{ then }
    \qform{X}{\sqrt{w}} \leq \sqbrac{\trace(X)}\iprod{w}{y}
    \text{ for some }y \in \genTH{\AdjCone}{\Psdlift{\QuadCone}}.
  \end{equation}
  So, let \(X \in \AdjCone \cap \QuadCone\) such that \(X \sqrt{w}
  \geq 0\).  We may assume that \(\qform{X}{\sqrt{w}} > 0\); otherwise
  take \(y = 0\).  Since \(X \in \QuadCone \subseteq \Psd{V}\), there
  exists \(B \in \Reals^{V\times V}\) such that \(X = B^{\transp} B\).
  Define
  \begin{gather*}
    c
    \coloneqq
    \paren[\big]{
      \qform{X}{\sqrt{w}}
    }^{-1/2}
    B\sqrt{w},
    \\
    d
    \coloneqq
    \diag(X),
    \\
    \Bt
    \coloneqq
    B \MPinverse{\sqbrac[\big]{\Diag(\sqrt{d})}},
    \\
    \Bb
    \coloneqq
    \Bt \Diag(\Bt^{\transp} c),
    \\
    y
    \coloneqq
    \Bb^{\transp} c
    =
    \Diag(\Bt^{\transp} c) \Bt^{\transp} c
    =
    \paren{\Bt^{\transp} c} \hprod \paren{\Bt^{\transp} c}.
  \end{gather*}

  We will show that
  \begin{equation}
    \label{eq:theta3-theta4-aux2}
    y \in \genTH{\AdjCone}{\Psdlift{\QuadCone}}.
  \end{equation}
  Set \(Y \coloneqq \Bb^{\transp} \Bb\) and note that
  \begin{equation}
    \label{eq:psdlift-name}
    \Yh
    \coloneqq
    \onelift{y}{Y}
    =
    \begin{bmatrix}
      1               & c^{\transp} \Bb   \\
      \Bb^{\transp} c & \Bb^{\transp} \Bb \\
    \end{bmatrix}
    =
    \begin{bmatrix}
      c^{\transp}   \\
      \Bb^{\transp} \\
    \end{bmatrix}
    \begin{bmatrix}
      c & \Bb \\
    \end{bmatrix}
    \in
    \Psdlift{\QuadCone};
  \end{equation}
  to see that \(Y\) lies in \(\QuadCone \cap \AdjCone\), note that \(Y =
  \DiagScaleMap{h}(X)\) for some \(h \geq 0\) since
  \begin{equation}
    \label{eq:theta3-theta4-aux3}
    \Bt^{\transp} c \geq 0,
  \end{equation}
  which follows from \(\paren{\qform{X}{\sqrt{w}}}^{1/2} \Bt^{\transp}
  c = \MPinverse{\sqbrac[\big]{\Diag(\sqrt{d})}} B^{\transp} B
  \sqrt{w} = \MPinverse{\sqbrac[\big]{\Diag(\sqrt{d})}} X \sqrt{w}
  \geq 0\).  Finally,
  \begin{equation*}
    \begin{split}
      \diag(Y)
      & =
      \diag(\Bb^{\transp}\Bb)
      =
      \diag\paren[\Big]{
        \Diag(\Bt^{\transp} c)
        \Bt^{\transp}
        \Bt
        \Diag(\Bt^{\transp} c)
      }
      \\
      & =
      \paren{
        \Bt^{\transp} c
      }
      \hprod
      \diag\paren{
        \Bt^{\transp}
        \Bt
      }
      \hprod
      \paren{
        \Bt^{\transp} c
      }
      \\
      & =
      \paren{
        \Bt^{\transp} c
      }
      \hprod
      \diag\paren[\Big]{
        \MPinverse{\sqbrac[\big]{\Diag(\sqrt{d})}}
        B^{\transp}
        B
        \MPinverse{\sqbrac[\big]{\Diag(\sqrt{d})}}
      }
      \hprod
      \paren{
        \Bt^{\transp} c
      }
      \\
      & =
      \paren{
        \Bt^{\transp} c
      }
      \hprod
      \incidvector{\supp(d)}
      \hprod
      \paren{
        \Bt^{\transp} c
      }
      =
      \paren{
        \Bt^{\transp} c
      }
      \hprod
      \paren{
        \Bt^{\transp} c
      }
      =
      y,
    \end{split}
  \end{equation*}
  where we used for the second-to-last equation the fact that \(d_j =
  0\) implies that \((\Bt^{\transp} c)_j = e_j^{\transp} \Bt^{\transp}
  c = e_j^{\transp} \MPinverse{\sqbrac[\big]{\Diag(\sqrt{d})}}
  B^{\transp} c = 0^{\transp} B^{\transp} c = 0\).  Thus, \(\Yh \in
  \liftedGenTH{\AdjCone}{\Psdlift{\QuadCone}}\),
  and~\eqref{eq:theta3-theta4-aux2} is proved.

  We also have \(\Bt \Diag(\sqrt{d}) = B
  \MPinverse{\sqbrac[\big]{\Diag(\sqrt{d})}} \Diag(\sqrt{d}) = B
  \Diag(\incidvector{\supp(d)}) = B\) since \(d_i = 0\) implies \(B
  e_i = 0\).  Thus,
  \begin{equation*}
    \begin{split}
      \qform{X}{\sqrt{w}}
      & =
      \paren*{
        \frac{
          \sqrt{w}^{\transp} B^{\transp} B \sqrt{w}
        }{
          \paren{
            \qform{X}{\sqrt{w}}
          }^{1/2}
        }
      }^2
      =
      \paren*{
        \sqrt{w}^{\transp}
        B^{\transp} c
      }^2
      \\
      & =
      \paren*{
        \sqrt{w}^{\transp}
        \Diag(\sqrt{d})
        \Bt^{\transp}
        c
      }^2
      =
      \paren*{
        \sqrt{d}^{\transp}
        \Diag(\sqrt{w})
        \sqrt{y}
      }^2
      \\
      & \leq
      \norm[\big]{\sqrt{d}}^2
      \norm[\big]{\Diag(\sqrt{w}) \sqrt{y}}^2
      =
      \sqbrac*{\trace(X)}
      \iprod{w}{y}.
    \end{split}
  \end{equation*}
  This completes the proof of~\eqref{eq:theta3-theta4-aux1}.

  Let \(X\) be an optimal solution for the RHS
  of~\eqref{eq:theta3-theta4}; the latter set is compact since it is a closed subset of \(\setst{X \in \Psd{V}}{\trace(X) = 1}\).  By
  Lemma~\ref{lemma:gijswijt-lambda-trick}, we have \(X\sqrt{w} \geq
  0\).  Thus, \(\suppf{\genTH{\AdjCone}{\Psdlift{\QuadCone}}}{w} \geq
  \qform{X}{\sqrt{w}}\) by~\eqref{eq:theta3-theta4-aux1} and the
  proof of~`\(\geq\)' is complete.

  The latter paragraph showed that the RHS is attained.  Attainment of the LHS also follows by compactness by Corollary~\ref{cor:TH-AK-convex-corner}.
\end{proof}

\subsection{Schur Liftings of Cones}
\label{sec:schur-liftings}

In this subsection, we define a lifting operator that is, in a sense,
dual to the PSD lifting introduced in the previous subsection.  After
proving that the Schur lifting of certain convex cones are also
convex, we shall prove a Weak Duality result relating both liftings,
and then we shall describe \(\FRAC(G)\) as a theta body over the Schur
lifting of a simple cone.

Let \(\QuadCone \subseteq \Sym{V}\).  Define the \emph{Schur lifting
  of~\(\QuadCone\)} as
\begin{equation}
  \Schurlift{\QuadCone}
  \coloneqq
  \setst*{
    \begin{bmatrix}
      x_0 & x^{\transp}\thinspace \\
      x   & X                     \\
    \end{bmatrix}
    \in
    \Sym{\setlift{V}}
  }{
    X \in \QuadCone,\,
    x_0 \in \Reals_+,\,
    x_0 X
    \succeq_{\QuadCone}
    \oprodsym{x}
  }.
\end{equation}
Note that \(\Schurlift{\Psd{V}} = \Psd{\setlift{V}}\).  It is
instructive to rewrite the PSD lifting \(\Psdlift{\QuadCone}\) in the
following format similar to~\(\Schurlift{\QuadCone}\):
\begin{equation}
  \text{%
    if \(\QuadCone \subseteq \Psd{V}\), then
  }
  \Psdlift{\QuadCone}
  =
  \setst*{
    \begin{bmatrix}
      x_0 & x^{\transp}\thinspace \\
      x   & X                     \\
    \end{bmatrix}
    \in
    \Sym{\setlift{V}}
  }{
    X \in \QuadCone,\,
    x_0 \in \Reals_+,\,
    x_0 X
    \succeq
    \oprodsym{x}
  };
\end{equation}
note the difference in the last (conic) inequality.

Whereas the expression~\eqref{eq:psd-lift} makes it clear that the PSD
lifting of a closed convex cone is convex, the same can not be said
about the Schur lifting.  We shall now show that, under certain simple
conditions, the Schur lifting of a convex cone is also convex, and in
fact it satisfies the properties~\eqref{eq:TH-AK-hypotheses-K} of the cone~\(\liftedQuadCone\) in Proposition~\ref{prop:TH-AK-almost-convex-corner}:
\begin{theorem}
  \label{thm:schur-K}
  Let \(\QuadCone \subseteq \Sym{V}\) be a diagonally
  scaling-invariant closed convex cone such that \(\QuadCone \supseteq
  \Psd{V}\) and \(\diag(\QuadCone) \subseteq \Reals_+^V\).  Then
  \(\Schurlift{\QuadCone}\) is a diagonally scaling-invariant closed
  convex cone that satisfies~\eqref{eq:TH-AK-hypotheses-K}.  In
  particular, if \(\AdjCone \subseteq \Sym{V}\) is a diagonally
  scaling-invariant closed convex cone such
  that~\eqref{eq:A-big-enough} holds, then
  \(\cl\paren[\big]{\genTH{\AdjCone}{\Schurlift{\QuadCone}}}\) is a
  convex corner contained in~\([0,1]^V\).
\end{theorem}

\begin{proof}
  \newcommand*{\AuxCone}{\widehat{\mathbb{M}}}
  Closedness of \(\Schurlift{\QuadCone}\) follows from that
  of~\(\QuadCone\).  Using \(\diag(\QuadCone) \subseteq \Reals_+^V\),
  it is not hard to check that
  \begin{equation}
    \label{eq:schur-K-aux1}
    \Schurlift{\QuadCone}
    =
    \cl(\AuxCone\thinspace),
  \end{equation}
  where
  \begin{equation*}
    \AuxCone
    \coloneqq
    \setst*{
      \begin{bmatrix}
        x_0 & x^{\transp}\thinspace \\
        x   & X                     \\
      \end{bmatrix}
      \in
      \Sym{\setlift{V}}
    }{
      X \in \QuadCone,\,
      x_0 \in \Reals_{++},\,
      x_0 X
      \succeq_{\QuadCone}
      \oprodsym{x}
    }.
  \end{equation*}

  It is obvious that \(\Schurlift{\QuadCone}\) is a cone.  We shall
  prove that~\(\Schurlift{\QuadCone}\) is convex by showing that
  \begin{equation}
    \label{eq:schur-K-aux2}
    \AuxCone \text{ is convex}.
  \end{equation}
  Since
  \begin{equation*}
    \AuxCone
    =
    \setst*{
      \begin{bmatrix}
        x_0 & x^{\transp}\thinspace \\
        x   & X                     \\
      \end{bmatrix}
      \in
      \Sym{\setlift{V}}
    }{
      X \in \QuadCone,\,
      x_0 \in \Reals_{++},\,
      \iprod{H}{x_0 X - \oprodsym{x}} \geq 0\,
      \forall H \in \dual{\QuadCone}
    },
  \end{equation*}
  it suffices to show that, for each \(H \in \dual{\QuadCone}\),
  \begin{equation}
    \label{eq:schur-K-aux3}
    \text{the function }
    f_H
    \ffrom
    x_0 \oplus x
    \in
    \Reals_{++}
    \oplus
    \Reals^V
    \mapsto
    \frac{
      \qform{H}{x}
    }{
      x_0
    }
    \text{ is convex}.
  \end{equation}
  Let \(H \in \dual{\QuadCone}\).  The Hessian of~\(f_H\) is
  \begin{equation*}
    \nabla^2 f_H (x_0 \oplus x)
    =
    \frac{2}{x_0^2}
    \begin{bmatrix}
      \qform{H}{x}/x_0 & - (H x)^{\transp}\thinspace \\
      - H x            & x_0 H                       \\
    \end{bmatrix}.
  \end{equation*}
  \bgroup
  \newcommand*{\eigenset}{\Hscr}%
  From the hypothesis that \(\QuadCone \supseteq \Psd{V}\) we get
  \(\dual{\QuadCone} \subseteq \Psd{V}\) whence \(H \succeq 0\), so we
  may write \(H = \sum_{h \in \eigenset} \oprodsym{h}\) for a finite
  subset \(\eigenset\) of~\(\Reals^V\).  For \(u \coloneqq x_0^{1/2}
  \oplus x_0^{-1/2} \ones \in \Reals_{++}^{\set{0}} \oplus
  \Reals_{++}^V\), we have
  \begin{equation*}
    \frac{x_0^2}{2}
    \DiagScaleMap{u}\paren*{
      \nabla^2 f_H (x_0 \oplus x)
    }
    =
    \sum_{h \in \eigenset}
    \DiagScaleMap{\iprod{h}{x} \oplus \ones}
    \paren*{
      \begin{bmatrix}
        1  & - h^{\transp}\thinspace \\
        -h & \oprodsym{h}            \\
      \end{bmatrix}
    }
    \succeq 0.
  \end{equation*}
  Thus, \(\nabla^2 f_H(x_0 \oplus x) \succeq 0\), and
  this concludes the proofs of~\eqref{eq:schur-K-aux3} and~\eqref{eq:schur-K-aux2}.  Therefore,
  \(\Schurlift{\QuadCone}\) is convex by~\eqref{eq:schur-K-aux1}.
  \egroup

  Let
  \begin{equation*}
    \Xh
    \coloneqq
    \begin{bmatrix}
      x_0 & x^{\transp}\thinspace \\
      x   & X                     \\
    \end{bmatrix}
    \in \Schurlift{\QuadCone},
  \end{equation*}
  and let \(h_0 \oplus h \in \Reals_+^{\set{0}} \oplus \Reals_+^{V}\).
  The condition \(\DiagScaleMap{h_0 \oplus h}(\Xh) \in
  \Schurlift{\QuadCone}\) is equivalent to \(\DiagScaleMap{h}(X) \in
  \QuadCone\) and \(h_0^2 x_0 \DiagScaleMap{h}(X) \succeq_{\QuadCone}
  h_0^2 \DiagScaleMap{h}(\oprodsym{x})\), both of which follow from
  the diagonal scaling invariance of~\(\QuadCone\).  It is easy to
  check that \(\Schurlift{\QuadCone}\)
  satisfies~\eqref{eq:lifted-K-big-enough}.
  For~\eqref{eq:lifted-K-not-too-big}, note that \(x - (x \hprod x) =
  \diag(X - \oprodsym{x}) \geq 0\) since \(\diag(\QuadCone) \subseteq
  \Reals_+^V\) whence \(x \subseteq [0,1]^V\).  This completes the
  proof that~\eqref{eq:TH-AK-hypotheses-K} holds.  The remainder of the
  statement of the theorem follows from
  Proposition~\ref{prop:TH-AK-almost-convex-corner}.
\end{proof}

The hypothesis that \(\QuadCone \supseteq \Psd{V}\) holds cannot be
dropped from Theorem~\ref{thm:schur-K}.  Consider the
cone~\(\ComplPositive{V}\) of completely positive matrices.  Now take
\(V \coloneqq \set{1,\dotsc,n}\) for some \(n \geq 2\) and note that both
\(\oprodsym{\incidvector{\set{0,1}}} +
\oprodsym{\incidvector{\set{2}}}\) and
\(\oprodsym{\incidvector{\set{0,2}}} +
\oprodsym{\incidvector{\set{1}}}\) lie in
\(\Schurlift{\ComplPositive{V}}\), whereas their midpoint does not.

Similar constructions of higher dimensional cones from lower
dimensional cones exist in many other areas of mathematics and
mathematical sciences.  In addition to the obvious Schur complement
connection, there are at least two other instances: one in the
construction of \emph{Siegel Domains} (see
\cite{Guler1996,TruongTuncel2004} and the references therein to
start), another in convex optimization and analysis in certain
recursive quadratic reformulation of optimization problems (see
\cite[pp.~165--168]{NN1994}).  In both of these constructions (which
guarantee the convexity of the resulting cone), \(\oprodsym{x}\)
corresponds to the bilinear form.  A key condition on the bilinear
forms in both of these constructions, corresponds to the condition
\(\oprodsym{x} \in \QuadCone\) in our current context. Indeed, this
last condition is equivalent to \(\QuadCone \supseteq \Psd{V}\).

\vspace{3pt}

PSD and Schur liftings of cones are in a sense dual to each other.  In
the next result, we make this statement a bit clearer by showing a
containment relation between theta bodies defined using these two
liftings.  The relation may be regarded as a form of Weak Duality, and
we shall later prove that equality, and hence a form of Strong
Duality, holds.
\begin{proposition}
  \label{prop:psd-schur-duality}
  Let \(\AdjCone \subseteq \Sym{V}\) be a diagonally scaling-invariant
  polyhedral cone such that \(\Image(\Diag) \subseteq \AdjCone\).  Let
  \(\QuadCone \subseteq \Psd{V}\) be a diagonally scaling-invariant
  closed convex cone such that \(\Diag(\Reals_+^V) \subseteq
  \QuadCone\).  Then
  \begin{equation}
    \label{eq:psd-schur-duality}
    \genTH{
      \AdjCone
    }{
      \Psdlift{\QuadCone}
    }
    \subseteq
    \abl[\big]{
      \cl\paren[\big]{
        \genTH[\big]{
          \thinspace\AdjDual{\AdjCone}
        }{
          \Schurlift{\dual{\QuadCone}}
        }
      }
    }.
  \end{equation}
\end{proposition}
\begin{proof}
  By continuity, it suffices to show that \(\iprod{x}{y} \leq 1\) if
  \(x \in \genTH{\AdjCone}{\Psdlift{\QuadCone}}\) and
  \(
  y \in
  \genTH[\big]{
    \thinspace\AdjDual{\AdjCone}
  }{
    \Schurlift{\dual{\QuadCone}}
  }
  \).
  Let \(x
  \in \genTH{\AdjCone}{\Psdlift{\QuadCone}}\), and let \(\Xh \in
  \liftedGenTH{\AdjCone}{\Psdlift{\QuadCone}}\) such that \(x =
  \diag(X)\) for \(X \coloneqq \Xh[V]\).  Let
  \(y \in
  \genTH[\big]{
    \thinspace\AdjDual{\AdjCone}
  }{
    \Schurlift{\dual{\QuadCone}}
  }
  \),
  and let
  \(
  \Yh \in
  \liftedGenTH[\big]{
    \thinspace\AdjDual{\AdjCone}
  }{
    \Schurlift{\dual{\QuadCone}}
  }
  \)
  such that \(y = \diag(Y)\) for \(Y \coloneqq \Yh[V]\).  Write \(Y =
  \Diag(u) - B\) where \(B \in \dual{\AdjCone\!}\).  Since
  \(\dual{\AdjCone\!} \subseteq \Null(\diag)\), we have \(u = y\).
  Then
  \begin{equation*}
    \begin{split}
      0
      & \leq
      \iprod{X}{Y - \oprodsym{y}}
      =
      \iprod{X}{\Diag(u) - B}
      -
      \qform{X}{y}
      =
      \iprod{x}{y}
      -
      \iprod{X}{B}
      -
      \qform{X}{y}
      \\
      & \leq
      \iprod{x}{y}
      -
      \qform{(\oprodsym{x})}{y}
      =
      \iprod{x}{y}
      -
      \iprod{x}{y}^2.
    \end{split}
  \end{equation*}
  Hence, \(\iprod{x}{y} \leq 1\).
\end{proof}

We can now give an example of a natural theta body defined over the
Schur lifting of a cone.  Define the \emph{weak fractional stable set
  polytope} of a graph \(G = (V,E)\) as the polytope
\begin{equation}
  \label{eq:FRAC}
  \FRAC(G)
  \coloneqq
  \setst{
    x \in [0,1]^V
  }{
    x_i + x_j \leq 1\,\forall ij \in E
  }.
\end{equation}
We shall prove that~\(\FRAC(G)\) is a theta body over the Schur
lifting of the convex cone
\begin{equation}
  \QuadConeTwo{V}
  \coloneqq
  \setst*{
    X \in \Sym{V}
  }{
    X[e] \succeq 0\,\forall e \in \tbinom{V}{2}
  };
\end{equation}
note the similarity with the hypothesis of
Corollary~\ref{cor:TH-AK-convex-corner}.  We shall make essential use of
Theorem~\ref{thm:schur-K} in our proof.

\begin{theorem}
  \label{prop:FRAC-as-TH}
  Let \(G = (V,E)\) be a graph such that \(\card{V} \geq 2\).  Then
  \begin{equation}
    \FRAC(G)
    =
    \genTH*{
      \PolyAdjCone{E}{E}
    }{
      \Schurlift*{\QuadConeTwo{V}}
    }.
  \end{equation}
\end{theorem}

\begin{proof}
  We first prove~`\(\supseteq\)'.  Let \(x \in
  \genTH*{\PolyAdjCone{E}{E}}{\Schurlift*{\QuadConeTwo{V}}}\), and let
  \(
  \Xh
  \in
  \liftedGenTH*{
    \PolyAdjCone{E}{E}
  }{
    \Schurlift*{\QuadConeTwo{V}}
  }
  \)
  such that \(x = \diag(X)\) for \(X \coloneqq \Xh[V]\).  By
  Theorem~\ref{thm:schur-K}, we have \(x \in [0,1]^V\).  Let \(e = ij
  \in E\).  Set \(Y \coloneqq X[e]\) and \(y \coloneqq
  x\restriction_e\).  Then \(X \succeq_{\QuadConeTwo{V}}
  \oprodsym{x}\) implies \(Y \succeq \oprodsym{y}\) so
  \begin{equation*}
    \begin{bmatrix}
      1   & x_i & x_j \\
      x_i & x_i & 0   \\
      x_j & 0   & x_j \\
    \end{bmatrix}
    \succeq
    0
    \implies
    \begin{bmatrix}
      1    & -x_i & -x_j \\
      -x_i & x_i  & 0    \\
      -x_j & 0    & x_j  \\
    \end{bmatrix}
    \succeq
    0
    \implies
    1 - x_i - x_j
    =
    \iprod*{
      \begin{bmatrix}
        1    & -x_i & -x_j \\
        -x_i & x_i  & 0    \\
        -x_j & 0    & x_j  \\
      \end{bmatrix}
    }{
      \oprodsym{\ones}
    }
    \geq 0.
  \end{equation*}
  Thus \(x \in \FRAC(G)\), and~`\(\supseteq\)' is proved.

  For the reverse inclusion, it suffices by Theorem~\ref{thm:schur-K}
  to show that
  \(\genTH*{\PolyAdjCone{E}{E}}{\Schurlift*{\QuadConeTwo{V}}}\)
  contains all the extreme points of~\(\FRAC(G)\).  So let~\(x\) be an
  extreme point of~\(\FRAC(G)\).
  By~\cite[Theorem~64.7]{Schrijver03b}, all coordinates of~\(x\) lie
  in \(\set{0,\thalf,1}\).  Define
  \begin{equation*}
    \Xh
    \coloneqq
    \onelift{x}{X}
    \in
    \Sym{\setlift{V}}
  \end{equation*}
  by setting \(\diag(X) \coloneqq x\) and \(X_{ij} \coloneqq
  \Iverson[\big]{ij \in \overline{E}\thinspace} \Iverson[\big]{x_i +
    x_j > 1} x_i x_j\) for every \(ij \in \tbinom{V}{2}\).  Note that
  \(X \in \PolyAdjCone{E}{E} \cap \QuadConeTwo{V}\) holds, and that
  \(X \succeq_{\QuadConeTwo{V}} \oprodsym{x}\) is equivalent to
  \begin{equation*}
    Y^{ij}
    \coloneqq
    \begin{bmatrix}
      1   & x_i    & x_j    \\
      x_i & x_i    & X_{ij} \\
      x_j & X_{ij} & x_j    \\
    \end{bmatrix}
    \in
    \Psd{\setlift{\set{i,j}}},
  \end{equation*}
  for every \(ij \in \tbinom{V}{2}\).  So let \(ij \in \tbinom{V}{2}\).
  If \(x_i + x_j \leq 1\), then \(X_{ij} = 0\) and either \(0 \in
  \set{x_i,x_j}\) or \(x_i = x_j = \thalf\), so \(Y^{ij} \succeq 0\)
  is easily verified.  So assume \(x_i + x_j > 1\).  Then \(ij \in
  \overline{E}\), so \(X_{ij} = x_i x_j\).  If \(x_i = x_j = 1\), then
  \(Y^{ij} = \oprodsym{\ones} \succeq 0\).  If \(x_i = 1\) and \(x_j =
  \thalf\), then
  \begin{equation*}
    Y^{ij}
    =
    \begin{bmatrix}
      1      & 1      & \thalf \\[3pt]
      1      & 1      & \thalf \\[3pt]
      \thalf & \thalf & \thalf \\
    \end{bmatrix}
    =
    \DiagScaleMap{\ones-e_2/2}
    \paren*{
      \oprodsym{\ones}
      +
      \oprodsym{e_2}
    }
    \succeq 0.
  \end{equation*}
  Thus, \(\Xh \in \Schurlift*{\QuadConeTwo{V}}\) and the proof
  of~`\(\subseteq\)' is complete.
\end{proof}

\section{Reformulations of Antiblocking Duality}
\label{sec:abl-reformulation}

We saw in Section~\ref{sec:psd-liftings} that the support functions of
some theta bodies defined over the PSD lifting of a cone \(\QuadCone
\subseteq \Sym{V}\) may be expressed as a conic optimization problem
over~\(\QuadCone\).  In this section, we shall see that something
similar holds for Schur liftings.  This shall be essentially a
manifestation of antiblocking duality, namely, that
\(\abl{\abl{\ConvexSet}} = \ConvexSet\) for every convex
corner~\(\ConvexSet\); we shall make use of this fact throughout the
rest of the paper.  At the end of the section, we shall have an
expression for~\(\theta_2\) from~\eqref{eq:theta-original-chain}, and
we shall also introduce~\(\theta_1\) along the way.

In the next result, we follow the rules set
for~\cite[Eq.~(9.3.6)]{GroetschelLS93a} to interpret the quotient
\(w_i/s_i\), with \(w_i,s_i \in \Reals_+\):
\begin{claimeq}
  \label{eq:suppf-bdd-arithmetics}
  if \(w_i = 0\), then we take the fraction~\(w_i/s_i\) to be~\(0\), even if the
  denominator is~\(0\); if \(w_i > 0\) but the denominator is~\(0\),
  we take the fraction~\(w_i/s_i\) to be \(+\infty\).
\end{claimeq}

\begin{proposition}
  \label{prop:suppf-theta1-abl}
  Let \(\ConvexSet \subseteq \Reals^V\) be a convex corner.  Let \(w
  \in \Reals_+^V\).  Then
  \begin{equation}
    \label{eq:suppf-theta1}
    \suppf{\ConvexSet}{w}
    =
    \adjustlimits
    \min_{s \in \abl{\ConvexSet}}
    \max_{i \in V}
    \frac{w_i}{s_i}.
  \end{equation}
  In particular,
  \begin{equation}
    \label{eq:suppf-theta1-abl}
    \suppf{\abl{\ConvexSet}}{w}
    =
    \adjustlimits
    \min_{x \in \ConvexSet}
    \max_{i \in V}
    \frac{w_i}{x_i}.
  \end{equation}
  Moreover, all four optimization problems in~\eqref{eq:suppf-theta1} and~\eqref{eq:suppf-theta1-abl} have optimal solutions.
\end{proposition}
\begin{proof}
  We may assume that \(w \neq 0\).  Let us prove `\(\leq\)'.  Let \(x
  \in \ConvexSet\) and \(s \in \abl{\ConvexSet}\).  We may assume that
  the \(\max\) on the RHS is finite so that, by following the rules
  from~\eqref{eq:suppf-bdd-arithmetics}, we have \(W \coloneqq
  \supp(w) \subseteq \supp(s) \eqqcolon S\).  Then
  \begin{equation*}
    \begin{split}
      \iprod{w}{x}
      & =
      \sum_{
        i \in W
      }
      w_i x_i
      =
      \sum_{
        i \in S
      }
      \frac{w_i}{s_i}
      s_i x_i
      \leq
      \paren*{
        \max_{
          i \in S
        }
        \frac{w_i}{s_i}
      }
      \sum_{i \in V}
      s_i x_i \leq
      \max_{i \in V}
      \frac{w_i}{s_i},
    \end{split}
  \end{equation*}
  where~\eqref{eq:suppf-bdd-arithmetics} is only used in the rightmost
  term.  For the reverse inequality, let \(\theta \coloneqq
  \suppf{\ConvexSet}{w} > 0\).  Then \(s \coloneqq \tfrac{1}{\theta} w
  \in \abl{\ConvexSet}\).  Since \(\max_{i \in V} w_i/s_i = \theta\),
  we find that the RHS of~\eqref{eq:suppf-theta1} is bounded above by
  \(\theta = \suppf{\ConvexSet}{w}\).  This proves~`\(\geq\)'
  in~\eqref{eq:suppf-theta1}, as well as attainment for its RHS.
  Finally, \eqref{eq:suppf-theta1-abl} follows
  from~\eqref{eq:suppf-theta1} by antiblocking duality.
\end{proof}

We shall later formulate the parameter~\(\theta_1\) (see the
discussion in Subsection~\ref{sec:organization}) essentially
as the optimization problem in the RHS of~\eqref{eq:suppf-theta1-abl}
applied to a theta body.  In a way, that formulation is unnecessary
for the proof of the generalization of~\eqref{eq:abl-intro}, and it
may be further simplified as a line-search, i.e., by a gauge function:
\begin{proposition}
  \label{prop:theta1-pre-theta2}
  Let \(\ConvexSet \subseteq \Reals^V\) be a convex corner.  Let \(w
  \in \Reals_+^V\).  Then
  \begin{equation}
    \label{eq:theta1-pre-theta2}
    \adjustlimits
    \min_{x \in \ConvexSet}
    \max_{i \in V}
    \frac{w_i}{x_i}
    =
    \min\setst*{
      \lambda \in \Reals_+
    }{
      w \in \lambda \ConvexSet
    }.
  \end{equation}
  Moreover, the RHS is attained.
\end{proposition}

\begin{proof}
  We may assume that \(w \neq 0\).  First we show~`\(\leq\)'.  Let
  \(\lambda \in \Reals_+\) such that \(w \in \lambda \ConvexSet\).
  Then \(\lambda > 0\) since \(w \neq 0\).  Set \(x \coloneqq
  \frac{1}{\lambda} w \in \ConvexSet\).  Then \(w_i/x_i = \Iverson{w_i
    \neq 0} \lambda\) for every \(i \in V\), according to the rules
  from~\eqref{eq:suppf-bdd-arithmetics}, so that \(\max_{i \in V}
  w_i/x_i = \lambda\), whence the LHS of~\eqref{eq:theta1-pre-theta2}
  is \(\leq \lambda\).  This proves~`\(\leq\)'.

  For the reverse inequality, let \(x \in \ConvexSet\) attain the LHS
  of~\eqref{eq:theta1-pre-theta2}, and let \(\lambda \coloneqq \max_{i
    \in V} w_i/x_i\).  Since \(w \neq 0\), we have \(\lambda > 0\).
  It is easy to check that \(y \coloneqq \tfrac{1}{\lambda} w\)
  satisfies \(y \leq x\).  Since \(0 \leq y \leq x \in \ConvexSet\)
  and \(\ConvexSet\) is lower-comprehensive, we find that \(y \in
  \ConvexSet\), i.e., \(w \in \lambda \ConvexSet\).  This
  proves~`\(\geq\)' on~\eqref{eq:theta1-pre-theta2}, as well as
  attainment in its RHS.
\end{proof}

The RHS of~\eqref{eq:theta1-pre-theta2} is, by definition, the gauge function
\(\gauge{\ConvexSet}{w}\) of~\(\ConvexSet\) at~\(w\), i.e.,
\(\gauge{\ConvexSet}{x}\) is defined as
\begin{equation}
  \label{eq:def-gauge}
  \gauge{\ConvexSet}{x}
  \coloneqq
  \inf\setst[\big]{
    \mu
  }{
    \mu \in \Reals_+,\,
    x \in \mu \ConvexSet
  }
  \qquad
  \forall x \in \Reals^V.
\end{equation}
From Propositions~\ref{prop:suppf-theta1-abl}
and~\ref{prop:theta1-pre-theta2}, we recover the fact that
\begin{equation}
  \label{eq:suppf-gauge}
  \text{%
    for a convex corner \(\ConvexSet \subseteq \Reals_+^V\),
    we have \(\suppf{\abl{\ConvexSet}}{\cdot} = \gauge{\ConvexSet}{\cdot}\)
    on \(\Reals_+^V\);
  }
\end{equation}
see~\cite[Theorem~14.5]{Rockafellar97a}.

A gauge function is oblivious to the upper surface of a set which is
``almost'' a convex corner:
\begin{proposition}
  \label{prop:gauge-non-closed}
  Let \(\ConvexSet \subseteq \Reals_+^V\) be a lower-comprehensive
  convex set with nonempty interior.  Then
  \begin{equation}
    \label{eq:gauge-non-closed}
    \gauge{\ConvexSet}{w}
    =
    \gauge{\cl(\ConvexSet)}{w}
    \qquad
    \forall w \in \Reals_+^V.
  \end{equation}
\end{proposition}

\begin{proof}
  The proof of~`\(\geq\)' is obvious.  For the reverse inequality, let
  \(w \in \Reals_+^V\) and let \(\lambda \in \Reals_+\) such that \(w
  \in \lambda \cl(\ConvexSet)\).  If \(\lambda = 0\), then \(w = 0\)
  and \(\gauge{\ConvexSet}{w} = 0 = \gauge{\cl(\ConvexSet)}{w}\), so
  assume \(\lambda > 0\).  We will show that \(w \in (\lambda +
  \eps)\ConvexSet\) for every \(\eps > 0\).  Let \(\eps > 0\).  Since
  \(\ConvexSet\) is lower-comprehensive and has nonempty interior,
  there exists \(M \in \Reals_{++}\) such that \(\ones/M \in
  \interior(\ConvexSet)\).  Thus, for every \(\mu \in \Reals\) such
  that \(0 < \mu \leq 1\), we have \(\frac{\mu}{M} \ones +
  \frac{1-\mu}{\lambda} w \in \interior(\ConvexSet)\).  For \(\mu
  \coloneqq \eps/(\lambda + \eps)\), this gives
  \(\frac{\eps}{M(\lambda + \eps)} \ones + \frac{1}{\lambda + \eps} w
  \in \interior(\ConvexSet)\), and since \(\ConvexSet\) is
  lower-comprehensive, we get \(w \in (\lambda + \eps)\ConvexSet\).
  Since \(\eps > 0\) was arbitrary, this proves~`\(\leq\)'
  in~\eqref{eq:gauge-non-closed}.
\end{proof}

We are now ready to show how an optimization problem
over~\(\Schurlift{\QuadCone}\) may sometimes be reduced to an
optimization problem over~\(\QuadCone\).  We shall use the following
simple fact:
\begin{equation}
  \label{eq:Schur-simplify}
  \text{
    if \(\QuadCone \supseteq \Psd{V}\), then
  }
  \begin{bmatrix}
    1 & x^{\transp} \\
    x & X
  \end{bmatrix}
  \in
  \Schurlift{\QuadCone}
  \text{ if and only if }
  X \succeq_{\QuadCone} \oprodsym{x}.
\end{equation}
\begin{proposition}
  \label{prop:theta-theta2}
  Let \(\AdjCone \subseteq \Sym{V}\) and \(\liftedQuadCone \subseteq
  \Sym{\setlift{V}}\) be diagonally scaling-invariant closed convex
  cones such that~\eqref{eq:A-big-enough}
  and~\eqref{eq:TH-AK-hypotheses-K} hold.  Let \(w \in \Reals_+^V\).
  Then
  \begin{equation}
    \label{eq:pre-theta2}
    \suppf{
      \abl{\cl\paren{\genTH{\AdjCone}{\liftedQuadCone}}}
    }{
      w
    }
    =
    \inf \setst*{
      \lambda \in \Reals_+
    }{
      W \in \AdjCone,\,
      \diag(W) = \lambda \ones,\,
      \begin{bmatrix}
        1        & \sqrt{w}^{\transp}\thinspace \\
        \sqrt{w} & W                            \\
      \end{bmatrix}
      \in
      \liftedQuadCone
    }.
  \end{equation}
  In particular, if \(\AdjCone\) is polyhedral, then
  \begin{equation}
    \label{eq:theta2-polyhedral}
    \suppf{
      \abl{\cl\paren{\genTH{\AdjCone}{\liftedQuadCone}}}
    }{
      w
    }
    =
    \inf \setst*{
      \lambda \in \Reals_+
    }{
      Y \in -\AdjCone \cap \Null(\diag),\,
      \begin{bmatrix}
        1        & \sqrt{w}^{\transp}\thinspace \\
        \sqrt{w} & \lambda I - Y                \\
      \end{bmatrix}
      \in
      \liftedQuadCone
    },
  \end{equation}
  and if \(\QuadCone \subseteq \Sym{V}\) is a diagonally
  scaling-invariant closed convex cone such that \(\QuadCone \supseteq
  \Psd{V}\) and \(\diag(\QuadCone) \subseteq \Reals_+^V\) then
  \begin{equation}
    \label{eq:theta2-polyhedral-Schur}
    \suppf{
      \abl{\cl\paren{\genTH{\AdjCone}{\Schurlift{\QuadCone}}}}
    }{
      w
    }
    =
    \inf \setst*{
      \lambda
    }{
      \lambda I \succeq_{\QuadCone} Y + \oprodsym{\sqrt{w}},\,
      Y \in - \AdjCone \cap \Null(\diag)
    }.
  \end{equation}
  Moreover, for each of~\eqref{eq:pre-theta2},
  \eqref{eq:theta2-polyhedral}, and
  \eqref{eq:theta2-polyhedral-Schur}, if the theta body on the LHS
  is closed, then the optimization problem on the RHS has an optimal
  solution.
\end{proposition}

\begin{proof}
  \newcommand*{\AdjConeB}{\AdjCone'}
  We may assume that \(w \neq 0\).  From
  Propositions~\ref{prop:TH-AK-almost-convex-corner}
  and~\ref{prop:gauge-non-closed} and from~\eqref{eq:suppf-gauge}, we
  have
  \begin{equation}
    \label{eq:theta2-suppf-gauge}
    \suppf{\abl{\cl\paren{\genTH{\AdjCone}{\liftedQuadCone}}}}{w}
    =
    \gauge{\genTH{\AdjCone}{\liftedQuadCone}}{w}
    =
    \inf \Lambda,
    \quad
    \text{where }
    \Lambda
    \coloneqq
    \setst*{
      \lambda \in \Reals_{++}
    }{
      w \in \lambda \genTH{\AdjCone}{\liftedQuadCone}
    }
  \end{equation}
  Note that \(\Lambda\) is convex, unbounded above, and bounded away
  from zero since the LHS of~\eqref{eq:theta2-suppf-gauge} is
  positive by Proposition~\ref{prop:TH-AK-almost-convex-corner}.
  Moreover, \(\Lambda\) is closed if
  \(\genTH{\AdjCone}{\liftedQuadCone}\) is closed by the attainment
  statement in Proposition~\ref{prop:theta1-pre-theta2}.  We may
  reformulate the set~\(\Lambda\) as
  \begin{equation}
    \label{eq:pre-theta2-aux1}
    \begin{split}
      \Lambda
      & =
      \setst*{
        \lambda \in \Reals_{++}
      }{
        W \in \AdjCone,\,
        \diag(W) = \tfrac{1}{\lambda} w,\,
        \begin{bmatrix*}
          1                   & \frac{1}{\lambda} w^{\transp}\thinspace \\
          \frac{1}{\lambda} w & W                                       \\
        \end{bmatrix*}
        \in \liftedQuadCone
      }
      \\
      & =
      \setst*{
        \lambda \in \Reals_{++}
      }{
        X \in \AdjCone,\,
        \diag(X) = \lambda \incidvector{\supp(w)},\,
        \begin{bmatrix}
          1        & \sqrt{w}^{\transp}\thinspace \\
          \sqrt{w} & X                            \\
        \end{bmatrix}
        \in
        \liftedQuadCone
      }
      \\
      & =
      \setst*{
        \lambda \in \Reals_{++}
      }{
        X \in \AdjCone,\,
        \diag(X) = \lambda \ones,\,
        \begin{bmatrix}
          1        & \sqrt{w}^{\transp}\thinspace \\
          \sqrt{w} & X                            \\
        \end{bmatrix}
        \in
        \liftedQuadCone
      },
    \end{split}
  \end{equation}
  where we used the diagonal scaling invariance of~\(\AdjCone\)
  and~\(\liftedQuadCone\) and the change of variable
  \begin{equation*}
    \Xh
    =
    \DiagScaleMap{1 \oplus \lambda w^{-1/2}}\paren*{
      \begin{bmatrix*}
        1                   & \frac{1}{\lambda} w^{\transp}\thinspace \\
        \frac{1}{\lambda} w & W                                       \\
      \end{bmatrix*}
    },
    \qquad
    \text{with }
    \component{w^{-1/2}}{i}
    \coloneqq
    \begin{cases*}
      w_i^{-1/2} & if \(w_i > 0\), \\
      1          & otherwise,      \\
    \end{cases*}
  \end{equation*}
  for the second equation, and the diagonal scaling invariance
  of~\(\AdjCone\) and~\(\liftedQuadCone\) and
  assumptions~\eqref{eq:A-big-enough}
  and~\eqref{eq:lifted-K-big-enough} for the last equation.  To
  prove~\eqref{eq:pre-theta2}, it now suffices to show that
  relaxing the constraint \(\lambda \in \Reals_{++}\) to \(\lambda \in
  \Reals_+\) in the RHS of~\eqref{eq:pre-theta2-aux1} does not
  change the set.  If it did, the relaxed set, which is convex, would
  contain~\(0\), and so~\(\Lambda\) would not be bounded away from
  zero, a contradiction.

  Suppose that \(\AdjCone\) is polyhedral.  It is easy to check that
  \(\AdjCone \cap \diag^{-1}(\lambda\ones) = \lambda I
  - \paren{-\AdjCone \cap \Null(\diag)}\); the inclusion
  `\(\supseteq\)' is obvious, whereas the reverse inclusion follows
  from Proposition~\ref{prop:A-Diag-closed-polyhedral}.  Thus,
  \eqref{eq:theta2-polyhedral} follows.
  Equation~\eqref{eq:theta2-polyhedral-Schur} follows
  from~\eqref{eq:theta2-polyhedral} and Theorem~\ref{thm:schur-K},
  using the equivalence~\eqref{eq:Schur-simplify}.  The constraint
  \(\lambda \in \Reals_+\) may be dropped since \(\diag(\QuadCone)
  \subseteq \Reals_+^V\).  In all cases, attainment if the theta body
  is closed follows from the closedness of~\(\Lambda\).
\end{proof}

\section{A Plethora of Theta Functions}
\label{sec:plethora}

We have now introduced all formulations of the
parameters~\(\theta_i\)'s from~\eqref{eq:theta-original-chain} and we
are ready to prove that they are all equal.  Let \(\AdjCone \subseteq
\Sym{V}\) be a diagonally scaling-invariant polyhedral cone such that
\(\Image(\Diag) \subseteq \AdjCone\). Let \(\QuadCone \subseteq
\Psd{V}\) be a diagonally scaling-invariant closed convex cone such
that \(\Diag(\Reals_+^V) \subseteq \QuadCone\) and
\(\interior(\QuadCone) \neq \emptyset\).  For each \(w \in
\Reals_+^V\), define:
\begin{gather*}
  \theta(\AdjCone, \QuadCone; w)
  \coloneqq
  \suppf[\Big]{
    \abl[\big]{\cl\paren[\big]{
        \genTH[\big]{
          \thinspace\AdjDual{\AdjCone}
        }{
          \Schurlift{\dual{\QuadCone}}}
      }
    }
  }{
    w
  },
  \\
  \theta_1(\AdjCone, \QuadCone; w)
  \coloneqq
  \inf \setst*{
    \max_{i \in V}
    \frac{w_i}{x_i}
  }{
    x \in \cl\paren[\big]{
      \genTH[\big]{
        \thinspace\AdjDual{\AdjCone}
      }{
        \Schurlift{\dual{\QuadCone}}
      }
    }
  },
  \\
  \theta_2(\AdjCone, \QuadCone; w)
  \coloneqq
  \inf\setst*{
    \lambda
  }{
    \lambda I \succeq_{\dual{\QuadCone}} Y + \oprodsym{\sqrt{w}},\,
    Y \in -\AdjDual{\AdjCone} \cap \Null(\diag)
  },
  \\
  \theta_3(\AdjCone, \QuadCone; w)
  \coloneqq
  \sup\setst*{
    \qform{X}{\sqrt{w}}
  }{
    \trace(X) = 1,\,
    X \in \QuadCone,\,
    X \in \AdjCone
  },
  \\
  \theta_4(\AdjCone, \QuadCone; w)
  \coloneqq
  \suppf[\big]{
    \genTH{\AdjCone}{\Psdlift{\QuadCone}}
  }{
    w
  }.
\end{gather*}
Here, the objective function for \(\theta_1(\AdjCone, \QuadCone; w)\)
is evaluated according to the arithmetic rules
from~\eqref{eq:suppf-bdd-arithmetics}.  For concreteness, we shall
finally define the Lovász theta number and the variants~\(\theta'\)
and~\(\theta^+\) as special cases of the above parameters.  Let \(G =
(V,E)\) be a graph.  For each \(w \in \Reals_+^V\), define
\begin{subequations}
  \label{eq:standard-theta-defs}
  \begin{gather}
    \theta(G;w) \coloneqq \theta\paren{ \PolyAdjCone{E}{E}, \Psd{V}; w
    },
    \\
    \theta'(G;w) \coloneqq \theta\paren{
      \PolyAdjCone{E\cup\overline{E}}{E}, \Psd{V}; w },
    \\
    \theta^+(G;w) \coloneqq \theta\paren{ \PolyAdjCone{\emptyset}{E},
      \Psd{V}; w }.
  \end{gather}
\end{subequations}

\begin{theorem}
  \label{thm:all-thetas-equal}
  Let \(\AdjCone \subseteq \Sym{V}\) be a diagonally scaling-invariant
  polyhedral cone such that \(\Image(\Diag) \subseteq \AdjCone\).  Let
  \(\QuadCone \subseteq \Psd{V}\) be a diagonally scaling-invariant
  closed convex cone such that \(\Diag(\Reals_+^V) \subseteq \QuadCone\) and \(\interior(\QuadCone) \neq \emptyset\).
  Let \(w \in \Reals_+^V\).  Then
  \begin{equation}
    \label{eq:thetaA-equal}
    \theta(\AdjCone, \QuadCone; w)
    =
    \theta_1(\AdjCone, \QuadCone; w)
    =
    \theta_2(\AdjCone, \QuadCone; w)
    =
    \theta_3(\AdjCone, \QuadCone; w)
    =
    \theta_4(\AdjCone, \QuadCone; w).
  \end{equation}
  Moreover, all optimization problems in~\eqref{eq:thetaA-equal} have
  optimal solutions except possibly for~\(\theta_2\), which has an
  optimal solution
  if~\(\genTH[\big]{\thinspace\AdjDual{\AdjCone}}{\Schurlift{\dual{\QuadCone}}}\)
  is closed.  Furthermore,
  \begin{equation}
    \label{eq:thetaA-abl}
    \abl[\big]{
      \cl\paren[\big]{
        \genTH[\big]{
          \thinspace\AdjDual{\AdjCone}
        }{
          \Schurlift{\dual{\QuadCone}}
        }
      }
    }
    =
    \genTH{
      \AdjCone
    }{
      \Psdlift{\QuadCone}
    }.
  \end{equation}
\end{theorem}

\begin{proof}
  The optimization problems that define~\(\theta_2(\AdjCone,
  \QuadCone; w)\) and~\(\theta_3(\AdjCone, \QuadCone; w)\) form a
  primal-dual pair of conic optimization problems; this follows from
  the polyhedrality of~\(\AdjCone\) and from \(\Image(\Diag) \subseteq
  \AdjCone\). Thus, the equation \(\theta_2(\AdjCone, \QuadCone; w) =
  \theta_3(\AdjCone, \QuadCone; w)\) follows by Conic Programming
  Strong Duality; see, e.g., \cite{BenTalN01a} or
  \cite[Theorem~1.1]{Carli13a}.  Although the conic formulation
  for~\(\theta_3(\AdjCone, \QuadCone; w)\) need not have a Slater
  point, the assumptions that~\(\AdjCone\) is polyhedral and
  \(\dual{\QuadCone} \supseteq \Psd{V}\) show that the optimization
  problem defining~\(\theta_2(\AdjCone, \QuadCone; w)\) has a
  restricted Slater point.  Equation~\eqref{eq:thetaA-equal} follows
  from \(\theta_2(\AdjCone, \QuadCone; w) = \theta_3(\AdjCone,
  \QuadCone; w)\), Propositions~\ref{prop:suppf-theta1-abl} and
  \ref{prop:theta-theta2}, and Theorem~\ref{thm:theta3-theta4}
  since
  \(
  \cl\paren[\big]{
    \genTH[\big]{
      \thinspace\AdjDual{\AdjCone}
    }{
      \Schurlift{\dual{\QuadCone}}
    }
  }
  \)
  is a convex corner by Theorem~\ref{thm:schur-K}.  Existence of optimal solutions follows from the corresponding statements in the previous results.
  Now~\eqref{eq:thetaA-abl} follows from conjugate duality applied to
  \(\theta(\AdjCone, \QuadCone; w) = \theta_4(\AdjCone, \QuadCone;
  w)\) for every \(w \in \Reals_+^V\).
\end{proof}

Theorem~\ref{thm:all-thetas-equal} implies~\eqref{eq:abl-TH-intro}
and~\eqref{eq:abl-TH'-intro} using the
descriptions~\eqref{eq:THs-as-genTHs} for every graph~\(G\).  Note
also that we could have mimicked the proof of the
chain~\eqref{eq:theta-original-chain} as in~\cite{GroetschelLS93a}
and~\cite{Knuth94a}; the proof that \(\theta_4(\AdjCone, \QuadCone; w)
\leq \theta(\AdjCone, \QuadCone; w)\) follows from
Proposition~\ref{prop:psd-schur-duality}.

In the context of Theorem~\ref{thm:all-thetas-equal}, the support
functions of the two theta bodies that appear in~\eqref{eq:thetaA-abl}
are gauges polar to each other; see~\cite[\S{}15]{Rockafellar97a} and recall the definition of gauge from~\eqref{eq:def-gauge}.
The corresponding polar inequality (that is, the corresponding
Cauchy-Schwarz inequality) for these gauges is stated next; compare
with~\cite[Proposition~8 and Theorem~18]{DukanovicR10a}.  For each
permutation \(\sigma\) on~\(V\), define the linear map
\(\PermMatrix{\sigma} \ffrom \Reals^V \fto \Reals^V\) as the linear
extension of the map \(e_i \in \Reals^V \mapsto e_{\sigma(i)}\).
For each \(L \in \Reals^{V \times V}\), define the \emph{congruence
  map}~\(\CongMap{L} \ffrom \Reals^{V \times V} \fto \Reals^{V \times
  V}\) as
\begin{equation}
  \CongMap{L}(X) \coloneqq L X L^{\transp}
  \qquad
  \forall X \in \Reals^{V \times V}.
\end{equation}
\bgroup
\newcommand*{\Group}{\Gamma}
\begin{corollary}
  \label{cor:theta-polar-ineq}
  Let \(\AdjCone \subseteq \Sym{V}\) be a diagonally scaling-invariant
  polyhedral cone such that \(\Image(\Diag) \subseteq \AdjCone\).  Let
  \(\QuadCone \subseteq \Psd{V}\) be a diagonally scaling-invariant
  closed convex cone such that \(\Diag(\Reals_+^V) \subseteq \QuadCone\) and \(\interior(\QuadCone) \neq \emptyset\).  If
  \(w, \wb \in \Reals_+^V\), then
  \begin{equation}
    \label{eq:theta-polar-gauges}
    \iprod{w}{\wb}
    \leq
    \suppf[\Big]{
      \genTH[\big]{\AdjCone}{\Psdlift{\QuadCone}}
    }{
      w
    }
    \cdot
    \suppf[\Big]{
      \cl\paren[\big]{
        \genTH[\big]{
          \thinspace\AdjDual{\AdjCone}
        }{
          \Schurlift{\dual{\QuadCone}}
        }
      }
    }{
      \wb
    }.
  \end{equation}
  Moreover, if there exists a transitive permutation group~\(\Group\)
  on~\(V\) such that
  \begin{equation}
    \label{eq:theta-polar-ineq-Aut}
    \setst{\CongMap{\PermMatrix{\sigma}}}{\sigma \in \Group}
    \subseteq
    \Aut(\AdjCone) \cap \Aut(\QuadCone),
  \end{equation}
  then
  \begin{equation}
    \label{eq:theta-polar-gauges-sym}
    \card{V}
    =
    \suppf[\Big]{
      \genTH[\big]{
        \AdjCone
      }{
        \Psdlift{\QuadCone}
      }
    }{
      \ones
    }
    \cdot
    \suppf[\Big]{
      \cl\paren[\big]{
        \genTH[\big]{
          \thinspace\AdjDual{\AdjCone}
        }{
          \Schurlift{\dual{\QuadCone}}
        }
      }
    }{
      \ones
    }.
  \end{equation}
\end{corollary}

\begin{proof}
  \newcommand*{\liftedGroup}{\widehat{\Group}}
  By Theorem~\ref{thm:schur-K}, we know that
  \(
  \cl\paren[\big]{
    \genTH[\big]{
      \thinspace\AdjDual{\AdjCone}
    }{
      \Schurlift{\dual{\QuadCone}}
    }
  }
  \)
  is a convex corner.  By~\eqref{eq:suppf-gauge} and
  Theorem~\ref{thm:all-thetas-equal}, the gauge function
  \(
  \gauge[\big]{
    \cl\paren[\big]{
      \genTH[\big]{
        \thinspace\AdjDual{\AdjCone}
      }{
        \Schurlift{\dual{\QuadCone}}
      }
    }
  }{
    \cdot
  }
  \)
  is the support function
  \(
  \suppf[\big]{
    \genTH[\big]{\AdjCone}{\Psdlift{\QuadCone}}
  }{
    \cdot
  }
  \).
  Hence, the support functions
  \(
  \suppf[\big]{
    \cl\paren[\big]{
      \genTH[\big]{
        \thinspace\AdjDual{\AdjCone}
      }{
        \Schurlift{\dual{\QuadCone}}
      }
    }
  }{
    \cdot
  }
  \)
  and
  \(
  \suppf[\big]{
    \genTH[\big]{
      \AdjCone
    }{
      \Psdlift{\QuadCone}
    }
  }{
    \cdot
  }
  \)
  are gauges polar to each other (when restricted to~\(\Reals_+^V\));
  see~\cite[Corollary~15.1.2]{Rockafellar97a}.
  Now~\eqref{eq:theta-polar-gauges} follows immediately.

  Next, we prove that~`\(\geq\)' holds
  in~\eqref{eq:theta-polar-gauges-sym} if \(w = \wb = \ones\)
  and~\eqref{eq:theta-polar-ineq-Aut} holds.  Assume the latter, and
  let \(\liftedGroup\) denote the permutation group
  \(
  \liftedGroup
  \coloneqq
  \setst[\big]{
    \hat{\sigma}
  }{
    \hat{\sigma}(0) = 0,\,
    \hat{\sigma}\restriction_{V} \in \Group
  }
  \)
  on \(\setlift{V}\).  It is clear that
  \(\setst[\big]{\CongMap{\PermMatrix{\hat{\sigma}}}}{\hat{\sigma} \in
    \liftedGroup} \subseteq \Aut\paren[\big]{\liftedQuadCone}\) for
  each \(\liftedQuadCone \in \set[\big]{\Psdlift{\QuadCone},
    \Schurlift{\dual{\QuadCone}}}\).  Together with
  \(\setst{\CongMap{\PermMatrix{\sigma}}}{\sigma \in \Group} \subseteq
  \Aut(\AdjCone)\), this yields
  \(\setst[\big]{\CongMap{\PermMatrix{\hat{\sigma}}}}{\hat{\sigma} \in
    \liftedGroup} \subseteq \Aut\paren[\big]{\liftedThetaBodyA}\) for
  each \(\thinspace\liftedThetaBodyA \in
  \set[\big]{\liftedGenTH[\big]{\AdjCone}{\Psdlift{\QuadCone}},
    \liftedGenTH[\big]{\,\AdjDual{\AdjCone}}{\Schurlift{\dual{\QuadCone}}}}\),
  whence
  \begin{equation*}
    \setst[\big]{\PermMatrix{\sigma}}{\sigma \in \Group}
    \subseteq
    \Aut\paren[\big]{\ThetaBodyA}
    \qquad
    \forall\thinspace
    \ThetaBodyA \in
    \set[\big]{
      \genTH[\big]{\AdjCone}{\Psdlift{\QuadCone}},
      \genTH[\big]{
        \thinspace\AdjDual{\AdjCone}
      }{
        \Schurlift{\dual{\QuadCone}}
      }
    }.
  \end{equation*}
  Thus, each support function on the RHS
  of~\eqref{eq:theta-polar-gauges-sym} is attained by a fixed point of
  the Reynolds operator
  \begin{equation*}
    x \in \Reals^V
    \mapsto
    \frac{1}{\card{\Group}}
    \sum_{\sigma \in \Group}
    \PermMatrix{\sigma} x.
  \end{equation*}
  Since \(\Group\) acts transitively on~\(V\), there exist \(\mu, \nu
  \in \Reals\) such that \(\mu \ones\) attains
  \(\suppf[\big]{\genTH[\big]{\AdjCone}{\Psdlift{\QuadCone}}}{\ones}\)
  and \(\nu \ones\) attains
  \(
  \suppf[\big]{
    \cl\paren[\big]{
      \genTH[\big]{
        \thinspace\AdjDual{\AdjCone}
      }{
        \Schurlift{\dual{\QuadCone}}
      }
    }
  }{
    \ones
  }
  \).
  By~\eqref{eq:thetaA-abl} from Theorem~\ref{thm:all-thetas-equal}, we
  get \(\iprod{\mu \ones}{\nu \ones} \leq 1\) so \(\mu \nu \card{V} \leq 1\).
  Thus,
  \begin{equation*}
    \suppf[\Big]{
      \genTH[\big]{\AdjCone}{\Psdlift{\QuadCone}}
    }{
      \ones
    }
    \cdot
    \suppf[\Big]{
      \cl\paren[\big]{
        \genTH[\big]{
          \thinspace\AdjDual{\AdjCone}
        }{
          \Schurlift{\dual{\QuadCone}}
        }
      }
    }{
      \ones
    }
    =
    \iprod{\ones}{\mu \ones}
    \iprod{\ones}{\nu \ones}
    =
    \mu \nu \card{V}^2
    \leq
    \card{V}.\qedhere
  \end{equation*}
\end{proof}
\egroup

\section{Theta Bodies over the Copositive and Completely Positive Cones}
\label{sec:TH-Cop}

In this section, we show that the stable set polytope of a graph and
one of its classical fractional relaxations are theta bodies.  The key
result we use to prove this is a completely positive formulation for
the stability number of a graph, due to de Klerk and
Pasechnik~\cite{KlerkP02a}.  As a consequence of the antiblocker
duality relation from Theorem~\ref{thm:all-thetas-equal}, we shall
derive a weighted generalization of a copositive formulation for the
fractional chromatic number of a graph, due to Dukanovic and
Rendl~\cite{DukanovicR10a}.

Let \(G = (V,E)\) be a graph.  For each \(w \in \Reals_+^V\), we set
\begin{equation}
  \label{eq:stab-weighted-def}
  \stab(G;w)
  \coloneqq
  \suppf{\STAB(G)}{w}.
\end{equation}
Recall that the stable set polytope~\(\STAB(G)\) was defined as the
convex hull of \(\setst{\incidvector{S}}{S \subseteq V \text{ stable
    in }G}\), that \(\Copositive{V}\) denotes the cone of
copositive matrices, and that \(\ComplPositive{V}\) is the cone of
completely positive matrices.  The key argument in the proof of the next result
comes from~\cite[Theorem~2.2]{KlerkP02a}:
\begin{proposition}
  \label{prop:stab-as-TH-AK}
  If \(G = (V,E)\) is a graph, then
  \begin{equation}
    \label{eq:stab-as-TH-AK}
    \genTH{
      \PolyAdjCone{E}{E}
    }{
      \Psdlift{\ComplPositive{V}}
    }
    =
    \STAB(G).
  \end{equation}
\end{proposition}

\begin{proof}
  To prove `\(\supseteq\)', note that, if \(S \subseteq V\) is a
  stable set of~\(G\), then \(\oprodsym{\paren{1 \oplus
      \incidvector{S}}} \in
  \liftedGenTH[\big]{\PolyAdjCone{E}{E}}{\Psdlift{\ComplPositive{V}}}\),
  whence \(\incidvector{S} \in
  \genTH[\big]{\PolyAdjCone{E}{E}}{\Psdlift{\ComplPositive{V}}}\).
  For the reverse inclusion it suffices by conjugate duality and
  Corollary~\ref{cor:TH-AK-convex-corner} to show that, for \(w \in
  \Reals_+^V\), we have
  \(
  \stab(G; w)
  \geq
  \suppf[\big]{
    \genTH[\big]{
      \PolyAdjCone{E}{E}
    }{
      \Psdlift{\ComplPositive{V}}
    }
  }{
    w
  }.
  \)
  Thus, it suffices by Theorem~\ref{thm:theta3-theta4} to show
  that, for \(w \in \Reals_+^V\), we have
  \begin{equation}
    \label{eq:stab-as-TH-AK-aux1}
    \stab(G; w)
    \geq
    \max \setst*{
      \qform{X}{\sqrt{w}}
    }{
      \trace(X) = 1,\,
      X \in \ComplPositive{V},\,
      X \in \PolyAdjCone{E}{E}
    }.
  \end{equation}
  Let \(w \in \Reals_+^V\).  We may assume that \(w \neq
  0\).  The extreme rays of the cone \(\ComplPositive{V} \cap
  \PolyAdjCone{E}{E}\) are of the form \(\oprodsym{x}\) with
  \(x \in \Reals_+^V\) and \(\supp(x)\) stable in~\(G\).  So there
  exists an optimal solution for the RHS
  of~\eqref{eq:stab-as-TH-AK-aux1} of the form \(\oprodsym{\xb}\) for
  some \(\xb \in \Reals_+^V\) such that \(\norm{\xb}^2 =
  \trace(\oprodsym{\xb}) = 1\) and \(\supp(\xb)\) is a stable set
  in~\(G\).  In fact, for any \(y \in \Reals_+^V\) such that
  \(\norm{y}^2 = 1\) and \(\supp(y) \subseteq \supp(\xb)\), the point
  \(\oprodsym{y}\) is feasible in the RHS
  of~\eqref{eq:stab-as-TH-AK-aux1} with objective value
  \(\iprod{\sqrt{w}}{y}^2\) whence the RHS
  of~\eqref{eq:stab-as-TH-AK-aux1} is equal to \(\max
  \setst[\big]{\iprod{\sqrt{w}}{y}^2}{y \in \Reals_+^V,\, \norm{y}^2 =
    1,\, \supp(y) \subseteq \supp(\xb)}\).  The optimality conditions
  for this optimization problem (i.e., Cauchy-Schwarz) show that an
  optimal solution is given by \(\yb \coloneqq
  \frac{\sqrt{u}}{\norm{\sqrt{u}}}\) where \(u \coloneqq w \hprod
  \incidvector{\supp(\xb)}\), and its objective value is
  \begin{equation*}
    \frac{
      \iprod{
        \sqrt{w}
      }{
        \sqrt{u}\thinspace
      }^2
    }{
      \norm{
        \sqrt{u}\thinspace
      }^2
    }
    =
    \frac{
      \iprod{
        \sqrt{u}
      }{
        \sqrt{u}\thinspace
      }^2
    }{
      \norm{
        \sqrt{u}\thinspace
      }^2
    }
    =
    \norm{
      \sqrt{u}\thinspace
    }^2
    =
    \iprod{w}{\incidvector{\supp(\xb)}}.
  \end{equation*}
  Since \(\supp(\xb)\) is stable, this concludes our proof
  of~\eqref{eq:stab-as-TH-AK-aux1}.
\end{proof}

Let \(G = (V,E)\) be a graph.  The \emph{fractional stable set
  polytope of~\(G\)} is defined as
\begin{equation}
  \QSTAB(G)
  \coloneqq
  \setst{
    x \in \Reals_+^V
  }{
    \iprod{\incidvector{K}}{x} \leq 1
    \text{ for every clique }K\text{ of }G
  }.
\end{equation}
Note that
\begin{equation}
  \QSTAB(G)
  =
  \abl[\big]{
    \STAB(\overline{G})
  }.
\end{equation}
For \(w \in \Reals_+^V\), the \emph{fractional chromatic number
  of~\(G\)} is
\begin{equation}
  \label{eq:frac-chromatic}
  \fracChromatic(G;w)
  \coloneqq
  \suppf[\big]{\QSTAB(\overline{G})}{w}.
\end{equation}

Proposition~\ref{prop:stab-as-TH-AK} yields immediately a weighted
generalization of~\cite[Corollary~5]{DukanovicR10a}:
\begin{corollary}
  Let \(G = (V,E)\) be a graph.  Let \(w \in \Reals_+^V\).  Then
  \begin{equation}
    \label{eq:chi-star-as-theta2}
    \fracChromatic(G; w)
    =
    \min \setst*{
      \lambda
    }{
       Y \in \PolyAdjCone{\overline{E}}{\overline{E}}^{\perp},\,
      \begin{bmatrix}
        1        & \sqrt{w}^{\transp}\thinspace \\
        \sqrt{w} & \lambda I - Y                \\
      \end{bmatrix}
      \in
      \Psdlift{\ComplPositive{V}}
    }.
  \end{equation}
\end{corollary}
\begin{proof}
  By Proposition~\ref{prop:stab-as-TH-AK}
  and~\eqref{eq:theta2-polyhedral} from
  Proposition~\ref{prop:theta-theta2}, we have
  \begin{equation*}
    \begin{split}
      \fracChromatic(G; w)
      & =
      \suppf{
        \QSTAB(\overline{G})
      }{
        w
      }
      =
      \suppf{
        \abl{\STAB(G)}
      }{
        w
      }
      =
      \suppf{
        \abl{
          \genTH{
            \PolyAdjCone{E}{E}
          }{
            \Psdlift{\ComplPositive{V}}
          }
        }
      }{
        w
      }
      \\
      & =
      \min \setst*{
        \lambda \in \Reals_+
      }{
        Y \in -\PolyAdjCone{E}{E} \cap \Null(\diag),\,
        \begin{bmatrix}
          1        & \sqrt{w}^{\transp}\thinspace \\
          \sqrt{w} & \lambda I - Y                \\
        \end{bmatrix}
        \in
        \Psdlift{\ComplPositive{V}}
      }
      \\
      & =
      \min \setst*{
        \lambda
      }{
        Y \in \PolyAdjCone{\overline{E}}{\overline{E}}^{\perp},\,
        \begin{bmatrix}
          1        & \sqrt{w}^{\transp}\thinspace \\
          \sqrt{w} & \lambda I - Y                \\
        \end{bmatrix}
        \in
        \Psdlift{\ComplPositive{V}}
      }.
    \end{split}
  \end{equation*}
  The constraint \(\lambda \in \Reals_+\) may be dropped since
  \(\diag(\ComplPositive{V}) \subseteq \Reals_+^V\).
\end{proof}

By the antiblocker relation from Theorem~\ref{thm:all-thetas-equal},
we know that~\(\QSTAB(G)\) is the closure of a theta body.  Unlike in
the cases presented so far, the fact that the latter theta body is
actually closed does not follow from our previous results.  Thus, we
proceed to prove its closedness separately.  We shall use an argument
from~\cite[Theorem~5]{GibbonsHPR97a} (more specifically, in the proof
of~\eqref{eq:copositive-bounded-lift-aux2.66} below).  We denote the maximum norm by \(\pnorm{\cdot}{\infty}\).
\begin{theorem}
  \label{thm-copositive-bounded-lift}
  Let \(\AdjCone \subseteq \Sym{V}\) be a diagonally scaling-invariant
  polyhedral cone such that \(\Image(\Diag) \subseteq \AdjCone\).  Then
  \begin{equation}
    \label{eq:6}
    \genTH{
      \AdjCone
    }{
      \Schurlift{\Copositive{V}}
    }
    =
    \setst*{
      \diag(\Xh[V])
    }{
      \Xh \in
      \liftedGenTH{
        \AdjCone
      }{
        \Schurlift{\Copositive{V}}
      },\,
      \pnorm{\Xh}{\infty} \leq 1
    }.
  \end{equation}
  Consequently, \(\genTH{\AdjCone}{\Schurlift{\Copositive{V}}}\) is a
  convex corner contained in \([0,1]^V\).
\end{theorem}
\begin{proof}
  The inclusion~`\(\supseteq\)' in~\eqref{eq:6} is trivial.  For the
  reverse inclusion, let \(x \in
  \genTH{\AdjCone}{\Schurlift{\Copositive{V}}}\), and let \(\Yh \in
  \liftedGenTH{\AdjCone}{\Schurlift{\Copositive{V}}}\) such that \(x =
  \diag(Y)\) for \(Y \coloneqq \Yh[V]\).  We shall
  use~\eqref{eq:Schur-simplify} with \(\QuadCone = \Copositive{V}\)
  throughout the proof without further mention.  Note that \(Y -
  \oprodsym{x} \in \Copositive{V}\) implies that \(x - (x \hprod x) =
  \diag(Y - \oprodsym{x}) \geq 0\) so
  \begin{equation}
    \label{eq:copositive-bounded-lift-aux1}
    x \in [0,1]^V.
  \end{equation}
  Let us prove that
  \begin{equation}
    \label{eq:7}
    \text{%
      we may assume that
      \(Y \in \Symnonneg{V}\)
      and
      \(Y = Y[\supp(x)] \oplus 0\).
    }
  \end{equation}
  Indeed, the principal submatrix \(Y = \Yh[V]\) from~\(\Yh\) may
  possibly be replaced with
  \begin{equation*}
    Y
    -
    2
    \sum\setst*{
      \Iverson[\big]{Y_{ij} < 0} Y_{ij}
      \Symmetrize\paren{\oprod{e_i}{e_j}}
    }{
      ij \in \tbinom{V}{2}
    }
  \end{equation*}
  (using the notation~\(\Symmetrize\) from~\eqref{eq:def-symmetrize})
  without affecting the relations \(\Yh[V] \in \AdjCone\) or \(\Yh[V]
  \succeq_{\Copositive{V}} \oprodsym{x}\), by
  Proposition~\ref{prop:A-Diag-closed-polyhedral} and the trivial fact
  that \(\Copositive{V} + \Symnonneg{V} = \Copositive{V}\).  Clearly,
  for \(S \coloneqq \supp(x)\) and \(\xb \coloneqq x\restriction_S\),
  we have \(Y[S] \succeq_{\Copositive{S}} \oprodsym{\xb}\).  Thus, by
  possibly replacing~\(\Yh[V]\) with \(\Yh[S] \oplus 0\) in~\(\Yh\),
  we shall have \(Y = Y[\supp(x)] \oplus 0\), and the proof
  of~\eqref{eq:7} is complete.  Thus, by possibly restricting our
  attention to the index set \(\supp(x)\),
  \begin{equation}
    \label{eq:8}
    \text{%
      we may assume that \(\supp(x) = V\).%
    }
  \end{equation}
  Write \(D \coloneqq \Diag(x)\) and \(B \coloneqq Y - D\).  Let \(G =
  (V,E)\) be the graph defined by \(E \coloneqq \setst[\big]{ij \in
    \tbinom{V}{2}}{B_{ij} > 0}\).  Define \(A \in \AdjCone \cap
  \Null(\diag)\) by setting \(A_{ij} \coloneqq \thalf \Iverson{ij \in
    E} \paren[\big]{1/x_i + 1/x_j}\) for each \(ij \in
  \tbinom{V}{2}\), where we used
  Proposition~\ref{prop:A-Diag-closed-polyhedral} to prove membership
  of~\(A\) in~\(\AdjCone\).  We claim that
  \begin{equation}
    \label{eq:copositive-bounded-lift-aux2}
    D^{-1} + A - \oprodsym{\ones}
    \in
    \Copositive{V}.
  \end{equation}
  We shall need to consider the following optimization problem in our
  proof:
  \begin{equation}
    \label{eq:copositive-bounded-lift-aux2.5}
    \min\setst*{
      \qform{
        \paren{
          D^{-1} + A
        }
      }{
        h
      }
    }{
      h \in \Reals_+^V,\,
      \iprod{\ones}{h} = 1
    }.
  \end{equation}
  Let us show that
  \begin{equation}
    \label{eq:copositive-bounded-lift-aux2.66}
    \text{%
      there exists an optimal solution~\(\hb\)
      for~\eqref{eq:copositive-bounded-lift-aux2.5} whose support is a
      stable set in~\(G\).%
    }
  \end{equation}
  Indeed, let \(\hb\) be an optimal solution
  for~\eqref{eq:copositive-bounded-lift-aux2.5} with minimal support.
  Note that an optimal solution exists by continuity and compactness.
  Suppose that \(ij \subseteq \supp(\hb)\) for some \(ij \in E\).  For
  each \(t \in \Reals\), define \(h_t \coloneqq \hb + t(e_i - e_j)\),
  and note that \(h_t\) is feasible
  for~\eqref{eq:copositive-bounded-lift-aux2.5} whenever \(t \in
  [-\hb_i,\hb_j]\).  The objective value of \(h_t\)
  in~\eqref{eq:copositive-bounded-lift-aux2.5} is, \(\qform{(D^{-1} +
    A)}{h_t} = \qform{(D^{-1} + A)}{\hb} + 2 t (e_i - e_j)^{\transp}
  (D^{-1} + A) \hb = \qform{(D^{-1} + A)}{\hb}\), where the final
  equation follows from the optimality of \(\hb = h_0\).  Since
  \(h_{\tb}\) is feasible in~\eqref{eq:copositive-bounded-lift-aux2.5}
  for \(\tb \coloneqq \hb_j\) and \(\supp(h_{\tb}) \subsetneq
  \supp(\hb)\), the proof
  of~\eqref{eq:copositive-bounded-lift-aux2.66} is complete.

  It follows from~\eqref{eq:copositive-bounded-lift-aux2.66} that
  \(\qform{A}{\hb} = 0\) and \(\qform{D^{-1}BD^{-1}}{\hb} = 0\).  Thus,
  since \(D^{-1} Y D^{-1} \succeq_{\Copositive{V}} D^{-1} \oprodsym{x}
  D^{-1}\) by the diagonal scaling invariance of~\(\Copositive{V}\),
  we get
  \begin{equation*}
    \begin{split}
      \qform{(D^{-1} + A)}{\hb}
      & =
      \qform{D^{-1}}{\hb}
      =
      \qform{(D^{-1}DD^{-1})}{\hb}
      =
      \qform{(D^{-1}\paren{D + B}D^{-1})}{\hb}
      \\
      & \geq
      \qform{D^{-1} \oprodsym{x} D^{-1}}{\hb}
      =
      \qform{\oprodsym{\ones}}{\hb}
      =
      1.
    \end{split}
  \end{equation*}
  Thus, \(\min\setst{\qform{(D^{-1} + A - \oprodsym{\ones})}{h}}{h \in
    \Reals_+^V,\, \iprodt{\ones}{h} = 1} \geq 0\)
  and~\eqref{eq:copositive-bounded-lift-aux2} is proved.  Set \(X
  \coloneqq \DiagScaleMap{x}(D^{-1} + A)\).
  Then~\eqref{eq:copositive-bounded-lift-aux2} implies \(X
  \succeq_{\Copositive{V}} \DiagScaleMap{x}(\oprodsym{\ones}) =
  \oprodsym{x}\).  Moreover, \(\diag(X) = x\) and, for \(ij \in E\), we
  have
  \begin{equation*}
    X_{ij}
    =
    \component[\big]{
      \DiagScaleMap{x}(A)
    }{
      ij
    }
    =
    \frac{x_i x_j}{2}
    \paren*{
      \frac{1}{x_i}
      +
      \frac{1}{x_j}
    }
    =
    \frac{x_j + x_i}{2}
    \leq 1
  \end{equation*}
  by~\eqref{eq:copositive-bounded-lift-aux1}.  Since \(X_{ij} = 0\) for
  \(ij \in \overline{E}\), it follows that
  \begin{equation*}
    \Xh
    \coloneqq
    \begin{bmatrix}
      1 & x^{\transp} \\
      x & X
    \end{bmatrix}
    \in
    \liftedGenTH{
      \AdjCone
    }{
      \Schurlift{\Copositive{V}}
    }
  \end{equation*}
  and \(\pnorm{\Xh}{\infty} \leq 1\).  This completes the proof
  of~\eqref{eq:6}.  It follows that the set
  \(\genTH{\AdjCone}{\Schurlift{\Copositive{V}}}\) is closed, since it
  is described by~\eqref{eq:6} as the linear image of a compact set.
  Thus, \(\genTH{\AdjCone}{\Schurlift{\Copositive{V}}}\) is a convex
  corner by Theorem~\ref{thm:schur-K}.
\end{proof}

\begin{corollary}
  \label{cor:qstab-as-TH-AK}
  Let \(G = (V,E)\) be a graph.  Then
  \begin{equation}
    \label{eq:qstab-as-TH-AK}
    \QSTAB(G)
    =
    \genTH{
      \PolyAdjCone{E}{E}
    }{
      \Schurlift{\Copositive{V}}
    }.
  \end{equation}
  In particular, for every \(w \in \Reals_+^V\), we have
  \begin{equation}
    \label{eq:chi-star-as-theta}
    \fracChromatic(G; w)
    =
    \max\setst*{
      \iprod{w}{x}
    }{
      X \in \PolyAdjCone{\overline{E}}{\overline{E}},\,
      \diag(X) = x,\,
      X \succeq_{\Copositive{V}} \oprodsym{x}
    }.
  \end{equation}
\end{corollary}
\begin{proof}
  We know that
  \(
  \abl[\big]{
    \genTH[\big]{
      \PolyAdjCone{E}{E}
    }{
      \Schurlift{\Copositive{V}}
    }
  }
  =
  \genTH[\big]{
    \PolyAdjCone{\overline{E}}{\overline{E}}
  }{
    \Psdlift{\ComplPositive{V}}
  }
  =
  \STAB(\overline{G})
  \)
  by Theorems~\ref{thm:all-thetas-equal}
  and~\ref{thm-copositive-bounded-lift} and
  Proposition~\ref{prop:stab-as-TH-AK}.  Thus, \eqref{eq:qstab-as-TH-AK}
  follows from antiblocking duality.  Now~\eqref{eq:chi-star-as-theta}
  follows from~\eqref{eq:qstab-as-TH-AK} and~\eqref{eq:Schur-simplify}
  since, for each \(w \in \Reals_+^V\), we have
  \(
  \fracChromatic(G; w)
  =
  \suppf[\big]{
    \QSTAB(\overline{G})
  }{
    w
  }
  =
  \suppf[\big]{
    \genTH[\big]{
      \PolyAdjCone{\overline{E}}{\overline{E}}
    }{
      \Schurlift{\Copositive{V}}
    }
  }{
    w
  }
  \).
\end{proof}

\section{Hoffman Bounds}
\label{sec:Hoffman}

The \emph{chromatic number} of a graph~\(G = (V,E)\), denoted
by~\(\chromatic(G)\), is the size of a smallest partition of~\(G\)
into stable sets.  Hoffman~\cite{Hoffman70a} proved the following
classical lower bound on~\(\chromatic(G)\):
\begin{equation}
  \label{eq:hoffman}
  \chromatic(G)
  \geq
  1
  -
  \frac{
    \lambdamax(\AdjMatrix{G})
  }{
    \lambdamin(\AdjMatrix{G})
  }.
\end{equation}
Here, \(\AdjMatrix{G}\) denotes the adjacency matrix of~\(G\).
Lovász~\cite{Lovasz79a} proved that the lower bound~\eqref{eq:hoffman}
on~\(\chromatic(G)\) remains valid if the adjacency
matrix~\(\AdjMatrix{G}\) is replaced with any matrix
in~\(\PolyAdjCone{E}{E}^{\perp}\), and that the tightest lower bound
on~\(\chromatic(G)\) arising in this manner is
precisely~\(\theta(\overline{G})\).  Knuth~\cite[Sec.~33]{Knuth94a}
defined another graph parameter, that he denoted by~\(\theta_6(G;w)\),
which is in fact equal to~\(\theta(G;w)\).  The parameter
\(\theta_6(G;w)\) is defined as an optimization problem, and the
objective function corresponding to \(\theta_6(\overline{G};\ones)\)
yields precisely the expression in the RHS of~\eqref{eq:hoffman} when
applied to an arbitrary matrix \(A \in \PolyAdjCone{E}{E}^{\perp}\).
We shall extend our framework in this direction.

Let \(\AdjCone \subseteq \Sym{V}\) and \(\QuadCone \subseteq
\Sym{V}\).  Following Knuth~\cite[Sec.~33]{Knuth94a}, we define
\begin{equation}
  \theta_6(\AdjCone, \QuadCone; w)
  \coloneqq
  \sup\setst[\big]{
    \lambdamax(B)
  }{
    \diag(B) = w,\,
    B \in \QuadCone,\,
    B \in \AdjCone
  }
\end{equation}
for every \(w \in \Reals_+^V\).  Note that the optimization problem on
the RHS above is not convex.  The next result relates the formulations for
\(\theta_6(\AdjCone, \QuadCone; w)\) and~\(\theta_3(\AdjCone,
\QuadCone; w)\).
\bgroup
\newcommand*{\Cone}{\mathbb{M}}
\newcommand*{\Bopt}{\bar{B}}
\newcommand*{\Xopt}{X^*}
\begin{theorem}
  \label{thm:theta3-theta6}
  Let \(\Cone \subseteq \Sym{V}\) be a diagonally scaling-invariant
  closed convex cone such
  that~\eqref{eq:gijswijt-lambda-trick-S-hyp} holds, and that
  \(\Diag(\Reals_+^V) \subseteq \Cone\).  Suppose that either
  \(\DiagScaleMap{h}(\Cone) \subseteq \Cone\) for every \(h \in
  \Reals^V\) or \(\Cone \subseteq \Symnonneg{V}\).  Let \(w \in
  \Reals_+^V\).  Then
  \begin{equation}
    \label{eq:theta3-theta6}
    \max\setst[\Big]{
      \lambda_{\max}(B)
    }{
      B \in \Cone,\,
      \diag(B) = w
    }
    =
    \max\setst[\Big]{
      \qform{X}{\sqrt{w}}
    }{
      \trace(X) = 1,\,
      X \in \Cone
    }.
  \end{equation}
  Moreover, both optimization problems have optimal solutions.
\end{theorem}

\begin{proof}
  Equation~\eqref{eq:theta3-theta6} when \(w = 0\) follows
  from~\eqref{eq:gijswijt-lambda-trick-S-hyp2}.  Thus, we may assume
  that \(w \neq 0\).  Then the RHS of~\eqref{eq:theta3-theta6} is
  positive, whence Lemma~\ref{lemma:gijswijt-lambda-trick} may be
  applied.  We start by proving~`\(\geq\)'
  in~\eqref{eq:theta3-theta6}.  Let \(\Xopt\) be an optimal solution
  for the RHS of~\eqref{eq:theta3-theta6}.  Define \(d\) and~\(\Xb\)
  as in the statement of Lemma~\ref{lemma:gijswijt-lambda-trick}.
  Then \(\Bb \coloneqq \DiagScaleMap{\sqrt{w}}(\Xb) + \Diag\paren*{w
    \hprod \incidvector{V \drop \supp(d)}}\) is feasible for the LHS
  and its objective value is \(\lambda_{\max}(\Bb) \geq
  \lambda_{\max}\paren*{\DiagScaleMap{\sqrt{w}}(\Xb)} =
  \qform{\Xopt}{\sqrt{w}}\) by~\eqref{eq:gijswijt-lambda-trick3}.

  Next we prove~`\(\leq\)' in~\eqref{eq:theta3-theta6}.  Let \(\Bopt\)
  be an optimal solution for the LHS of~\eqref{eq:theta3-theta6}; one
  exists by compactness, as a consequence
  of~\eqref{eq:gijswijt-lambda-trick-S-hyp3}.  Let \(\lambda \coloneqq
  \lambda_{\max}(\Bopt) > 0\) and let \(b \in \Reals^V\) be a unit
  vector such that \(\Bopt b = \lambda b\).  Note that \(\supp(b)
  \subseteq \supp(w)\) by~\eqref{eq:gijswijt-lambda-trick-S-hyp2}.
  The matrix \(\Xt \coloneqq \MPinverse{\Diag(\sqrt{w})} \Bopt
  \MPinverse{\Diag(\sqrt{w})}\) satisfies \(\diag(\Xt) =
  \incidvector{\supp(w)}\), whence \(\Xb \coloneqq
  \DiagScaleMap{b}(\Xt)\) satisfies \(\trace(\Xb) = \norm{b}^2 = 1\).
  If \(\DiagScaleMap{h}(\Cone) \subseteq \Cone\) for each \(h \in
  \Reals^V\), then \(\Xb \in \Cone\) follows from \(\Bopt \in \Cone\).
  If \(\Cone \subseteq \Symnonneg{V}\), then \(\Xb \in \Cone\) follows
  from \(\Bopt \in \Cone\) and by the diagonal scaling invariance
  of~\(\Cone\), since we may assume that \(b \geq 0\) by the
  Perron-Frobenius Theorem; see, e.g., \cite[Theorem~8.3.1]{HornJ90a}.
  In either case, we find that \(\Xb \in \Cone\), whence \(\Xb\) is
  feasible in the RHS of~\eqref{eq:theta3-theta6}.  Finally, its
  objective value in the RHS of~\eqref{eq:theta3-theta6} is
  \(\qform{\Xb}{\sqrt{w}} = \qform{\DiagScaleMap{b}(\Xt)}{\sqrt{w}} =
  \qform{\DiagScaleMap{\sqrt{w}}(\Xt)}{b} = \qform{\Bopt}{b} =
  \lambda\), where we used~\eqref{eq:gijswijt-lambda-trick-S-hyp2} to
  get \(\Bopt = \DiagScaleMap{\sqrt{w}}(\Xt)\).  This completes the
  proof of~\eqref{eq:theta3-theta6}.
\end{proof}
\egroup

\begin{corollary}
  \label{cor:theta3-theta6-explicit}
  Let \(\AdjCone \subseteq \Sym{V}\) be a diagonally scaling-invariant
  polyhedral cone such that \(\Image(\Diag) \subseteq \AdjCone\).  Let
  \(\QuadCone \subseteq \Psd{V}\) be a diagonally scaling-invariant
  closed convex cone such that \(\Diag(\Reals_+^V) \subseteq \QuadCone\).  Let \(w \in \Reals_+^V\).  If either
  \(\DiagScaleMap{h}(\AdjCone \cap \QuadCone) \subseteq \AdjCone \cap
  \QuadCone\) for all \(h \in \Reals^V\) or \(\AdjCone \cap \QuadCone
  \subseteq \Symnonneg{V}\), then
  \begin{equation}
    \label{eq:theta3-theta6-explicit}
    \theta_6(\AdjCone, \QuadCone; w)
    =
    \theta_3(\AdjCone, \QuadCone; w)
  \end{equation}
\end{corollary}
\begin{proof}
  Immediate from Theorem~\ref{thm:theta3-theta6}.
\end{proof}

Next we shall show that, when applied to \(w = \ones\), the objective
value of the LHS of~\eqref{eq:theta3-theta6} has the same form as the
RHS of~\eqref{eq:hoffman}, and thus generalizes it:
\begin{proposition}
  \label{prop:hoffman-bounds-expr}
  Let \(\AdjCone, \QuadCone \subseteq \Sym{V}\) be diagonally
  scaling-invariant closed convex cones.  Suppose that \(\AdjCone\) is
  polyhedral and that \(I \in \AdjCone \cap \QuadCone\).  Then
  \begin{equation}
    \label{eq:hoffman-bounds}
    \theta_6(\AdjCone, \QuadCone; \ones)
    =
    \max \setst*{
      1
      -
      \Iverson{\mu \neq 0}
      \frac{\lambda_{\max}(A)}{\mu}
    }{
      A \in \AdjCone \cap \Null(\diag),\,
      \mu \in - \Reals_+,\,
      A \succeq_{\QuadCone} \mu I
    }.
  \end{equation}
\end{proposition}

\begin{proof}
  We have
  \begin{equation*}
    \begin{split}
      &
      \max\setst[\big]{
        \lambda_{\max}(B)
      }{
        \diag(B) = \ones,\,
        B \in \AdjCone,\,
        B \in \QuadCone
      }
      \\
      & \qquad\qquad =
      \max\setst[\big]{
        \lambda_{\max}(I + A)
      }{
        \diag(A) = 0,\,
        I + A \in \AdjCone,\,
        I + A \in \QuadCone
      }
      \\
      & \qquad\qquad =
      \max\setst[\big]{
        1 + \Iverson{\nu \neq 0} \nu \lambda_{\max}(A)
      }{
        A \in \AdjCone \cap \Null(\diag),\,
        \nu \in \Reals_+,\,
        \nu A \succeq_{\QuadCone} -I
      }
      \\
      & \qquad\qquad =
      \max\setst*{
        1
        -
        \Iverson{\mu \neq 0}
        \frac{\lambda_{\max}(A)}{\mu}
      }{
        A \in \AdjCone \cap \Null(\diag),\,
        \mu \in - \Reals_+,\,
        A \succeq_{\QuadCone} \mu I
      }.
    \end{split}
  \end{equation*}
  Note that we used Proposition~\ref{prop:A-Diag-closed-polyhedral} on
  the second equation.  That equation also uses \(I \in \AdjCone\),
  whereas the third one makes use of \(I \in \QuadCone\).
\end{proof}

\begin{corollary}
  \label{cor:hoffman-theta}
  Let \(G = (V,E)\) be graph.  Then
  \begin{subequations}
    \label{eq:hoffman-theta}
    \begin{align}
      \theta\paren*{\overline{G}; \ones}
      & =
      \max\setst*{
        1 - \Iverson{A \neq 0} \dfrac{\lambdamax(A)}{\lambdamin(A)}
      }{
        A \in \dual{\PolyAdjCone{E}{E}}
      },
      \\
      \theta'\paren*{\overline{G}; \ones}
      & =
      \max\setst*{
        1 - \Iverson{A \neq 0} \dfrac{\lambdamax(A)}{\lambdamin(A)}
      }{
        A \in \dual{\PolyAdjCone{E \cup \overline{E}}{E}}
      },
      \\
      \stab\paren*{\overline{G}; \ones}
      & =
      \max\setst*{
        1 - \Iverson{\mu \neq 0} \dfrac{\lambdamax(A)}{\mu}
      }{
        A \in \dual{\PolyAdjCone{E}{E}},\,
        \mu \in -\Reals_+,\,
        A \succeq_{\ComplPositive{V}} \mu I
      }.
    \end{align}
  \end{subequations}
  Moreover, all the optimization problems in~\eqref{eq:hoffman-theta}
  have optimal solutions.
\end{corollary}
\begin{proof}
  Immediate from Theorem~\ref{thm:all-thetas-equal},
  Corollary~\ref{cor:theta3-theta6-explicit}, and
  Propositions~\ref{prop:hoffman-bounds-expr}
  and~\ref{prop:stab-as-TH-AK}.
\end{proof}

Finally, note that, for a graph \(G = (V,E)\), we have
\begin{alignat*}{2}
  \max\setst*{
    1
    -
    \Iverson{A \neq 0}
    \frac{
      \lambdamax(A)
    }{
      \lambdamin(A)
    }
  }{
    A \in \dual{\PolyAdjCone{E}{E}}
  }
  & =
  \suppf[\big]{
    \genTH[\big]{
      \PolyAdjCone{\overline{E}}{\overline{E}}
    }{
      \Schurlift{\Psd{V}}
    }
  }{
    \ones
  }
  &
  \quad
  &
  \text{by Corollary~\ref{cor:hoffman-theta}},
  \\
  & \leq
  \suppf[\big]{
    \genTH[\big]{
      \PolyAdjCone{\overline{E}}{\overline{E}}
    }{
      \Schurlift{\Copositive{V}}
    }
  }{
    \ones
  }
  &
  \quad
  &
  \text{since \(\Psd{V} \subseteq \Copositive{V}\)},
  \\
  & =
  \suppf[\big]{
    \QSTAB(\overline{G})
  }{
    \ones
  }
  &
  \quad
  &
  \text{by Corollary~\ref{cor:qstab-as-TH-AK}},
  \\
  & =
  \fracChromatic(G;\ones)
  \leq
  \chromatic(G).
\end{alignat*}
This proves that the best bound from this family of lower bounds
for~\(\chromatic(G)\) is~\(\theta(\overline{G})\), as was already
shown by Lovász~\cite[Theorem~6]{Lovasz79a}; see also~\cite{Bilu06a}.

\section{Theta Bodies over the Positive Semidefinite Cone}
\label{sec:TH-Psd}

The development of the theory makes it clear that the positive
semidefinite cone plays a key role in theta bodies.  For instance, \(\Psd{V}\)
delineates the range of applicability of the lifting operators in several results (e.g., Theorems~\ref{thm:theta3-theta4}, \ref{thm:schur-K}, and~\ref{thm:all-thetas-equal}) and it
provides the most symmetric antiblocking
relation~\eqref{eq:thetaA-abl}, in the form
of~\eqref{eq:abl-TH-intro}.  In this section, we focus on some special
properties of theta bodies defined over the semidefinite cone.
Clearly, the most interesting such families are the ones defined
in~\eqref{eq:THs-as-genTHs}.  We shall reprove two classical results
about these families of theta bodies using our unifying framework, and
we conclude the section and the paper with a weighted extension of the
convex characterization of Luz and Schrijver~\cite{LuzS05a} to all
semidefinite variants of~\(\theta\).

We start by reproving that every facet of a theta body over the
semidefinite cone is defined by a clique inequality.  We briefly
recall some basic concepts of the facial structure of convex sets.  Let
\(\ConvexSet \subseteq \Euclidean\) be a convex set.  A convex subset \(\Face
\subseteq \ConvexSet\) is a \emph{facet} of~\(\ConvexSet\) if
\(\dim(\Face) = \dim(\ConvexSet)-1\) and \(\Face = \argmax_{x \in
  \ConvexSet} \iprod{c}{x}\) for some nonzero \(c \in
\dual{\Euclidean}\); in this case, we say that the facet \(\Face\) is
\emph{determined} by the inequality \(\iprod{c}{x} \leq
\suppf{\ConvexSet}{c}\).

The proof below is a slight modification
of~\cite[Theorem~67.13]{Schrijver03b}.  Note how it uses the
complementarity established in
Proposition~\ref{prop:psd-schur-duality}:
\begin{theorem}
  Let \(\AdjCone \subseteq \Sym{V}\) be a diagonally scaling-invariant
  polyhedral cone such that \(\Image(\Diag) \subseteq \AdjCone\).  Then each
  facet of~\(\genTH[\big]{\AdjCone}{\Psd{\setlift{V}}}\) is determined
  either by \(x_i \geq 0\) for some \(i \in V\), or by \(\iprod{w}{x}
  \leq 1\) for some \(w \in
  \genTH[\big]{\thinspace\AdjDual{\AdjCone}}{\Psd{\setlift{V}}} \cap
  \set{0,1}^V\).
\end{theorem}

\begin{proof}
  By Corollary~\ref{cor:TH-AK-convex-corner},
  \(\genTH[\big]{\AdjCone}{\Psd{\setlift{V}}}\) is a convex corner, and its antiblocker is \(\genTH[\big]{\thinspace\AdjDual{\AdjCone}}{\Psd{\setlift{V}}}\) by Theorem~\ref{thm:all-thetas-equal}.
  By a well-known dual characterization of facets of convex corners
  (see, e.g., \cite[Theorem~8]{CarliT13b}), it suffices to show that,
  if the inequality \(\iprod{w}{x} \leq 1\) determines a facet
  of~\(\genTH[\big]{\AdjCone}{\Psd{\setlift{V}}}\) for some \(w \in
  \genTH[\big]{\thinspace\AdjDual{\AdjCone}}{\Psd{\setlift{V}}}\), then \(w
  \in \set{0,1}^V\).

  So let \(w \in
  \genTH[\big]{\thinspace\AdjDual{\AdjCone}}{\Psd{\setlift{V}}}\) such
  that \(\iprod{w}{x} \leq 1\) determines a facet~\(\Face\)
  of~\(\genTH[\big]{\AdjCone}{\Psd{\setlift{V}}}\), and let \(\hat{W}
  \in
  \liftedGenTH[\big]{\thinspace\AdjDual{\AdjCone}}{\Psd{\setlift{V}}}\)
  such that \(w = \diag(W)\) for \(W \coloneqq \hat{W}[V]\).  Write
  \(\hat{W} = \sum_{i=1}^r \oprodsym{\paren{\alpha_i \oplus a_i}}\)
  for linearly independent vectors \(\setst{\alpha_i \oplus a_i}{i \in
    [r]} \subseteq \Reals^{\setlift{V}}\), where \([r] \coloneqq \set{1,\dotsc,r}\).  Let us show that,
  \begin{equation}
    \label{eq:1}
    \text{%
      if \(x \in \Face\), then \(\iprod{a_i}{x} = \alpha_i\) for each
      \(i \in [r]\).
    }
  \end{equation}
  So let \(x \in \Face\) and let \(\hat{X} \in
  \liftedGenTH[\big]{\AdjCone}{\Psd{\setlift{V}}}\) such that \(x =
  \diag(X)\) for \(X \coloneqq \hat{X}[V]\).  Write \(W = \Diag(w) -
  B\) for some \(B \in \dual{\AdjCone\!}\).  Since \(\hat{W} \in
  \liftedGenTH[\big]{\thinspace\AdjDual{\AdjCone}}{\Psd{\setlift{V}}}\)
  and \(\hat{X} \in \liftedGenTH[\big]{\AdjCone}{\Psd{\setlift{V}}}\),
  we have
  \begin{equation*}
    \begin{split}
      0
      & \leq
      \iprod{W}{X - \oprodsym{x}}
      =
      \iprod{
        \Diag(w) - B
      }{
        X
      }
      -
      \qform{W}{x}
      =
      \iprod{w}{x}
      -
      \iprod{B}{X}
      -
      \qform{W}{x}
      \\
      & \leq
      \iprod{w}{x}
      -
      \iprod{w}{x}^2
      =
      1 - 1
      =
      0.
    \end{split}
  \end{equation*}
  Equality throughout implies that \(\qform{W}{x} = \iprod{w}{x}^2 =
  1\) and that \(\iprod{W}{X} = \qform{W}{x} = 1\).  Thus,
  \begin{equation*}
    \iprod*{
      \begin{bmatrix}
        1  & -x^{\transp}\, \\
        -x & X              \\
      \end{bmatrix}
    }{
      \begin{bmatrix}
        1 & w^{\transp}\, \\
        w & W             \\
      \end{bmatrix}
    }
    =
    0.
  \end{equation*}
  In particular,
  \begin{equation*}
    \qform{
      \begin{bmatrix}
        1  & -x^{\transp}\, \\
        -x & X              \\
      \end{bmatrix}
    }{
      \paren{
        \alpha_i \oplus a_i
      }
    }
    =
    0
    \quad
    \forall i \in [r],
  \end{equation*}
  which implies that
  \begin{equation*}
    \begin{bmatrix}
      1  & -x^{\transp}\, \\
      -x & X              \\
    \end{bmatrix}
    \paren{
      \alpha_i \oplus a_i
    }
    =
    0
    \quad
    \forall i \in [r],
  \end{equation*}
  and, in particular, \(\iprod{a_i}{x} = \alpha_i\) for all \(i \in
  [r]\).  This proves~\eqref{eq:1}.

  Since \(\Face\) has \(n\) affinely independent vectors, \eqref{eq:1} implies that \(r = 1\), i.e.,
  \(\hat{W}\) is rank-one.  Thus, \(\diag(W) = w\) and~\(\hat{W} \in
  \Psd{\setlift{V}}\) imply that \(w \in \set{0,1}^V\).
\end{proof}

\subsection{Geometric Representations from Theta Bodies}

The theta bodies described in~\eqref{eq:THs-as-genTHs} all have the
form \(\genTH[\big]{\PolyAdjCone{E^+}{E^-}}{\Psd{\setlift{V}}}\) for
some \(E^+,E^- \subseteq \tbinom{V}{2}\).  The elements of these sets
arise from certain vectors which may be regarded as geometric
representations of graphs (see, e.g., \cite{Lovasz09a}):
\begin{proposition}
  \label{prop:polyhedral-A-geomrep}
  Let \(E^+,E^- \subseteq \tbinom{V}{2}\).  Then
  \(\genTH{\PolyAdjCone{E^+}{E^-}}{\Psd{\setlift{V}}}\) consists of
  all vectors \(x \in \Reals^V\) of the form
  \begin{equation}
    \label{eq:polyhedral-A-xform}
    x_i
    =
    \iprod{u_0}{u_i}^2
    \qquad
    \forall i \in V
  \end{equation}
  for vectors \(\setst{u_i}{i \in \setlift{V}} \subseteq
  \Reals^{\setlift{V}}\) satisfying the following properties:
  \begin{subequations}
    \label{eq:polyhedral-A-geomrep}
    \begin{alignat}{2}
      \label{eq:polyhedral-A-geomrep-zerorow}
      &\iprod{u_0}{u_i} \geq 0 &       &\forall i \in V,           \\
      \label{eq:polyhedral-A-geomrep-norms}
      &\norm{u_i} = 1          &\qquad &\forall i \in \setlift{V}, \\
      \label{eq:polyhedral-A-geomrep-iprods1}
      &\iprod{u_i}{u_j} \geq 0 &       &\forall ij \in E^+,        \\
      \label{eq:polyhedral-A-geomrep-iprods2}
      &\iprod{u_i}{u_j} \leq 0 &       &\forall ij \in E^-.
    \end{alignat}
  \end{subequations}
\end{proposition}
\begin{proof}
  \newcommand*{\targetSet}{\Cscr}
  Denote by~\(\targetSet\) the set of all vectors \(x \in \Reals^V\)
  of the form given by~\eqref{eq:polyhedral-A-xform} for vectors
  \(\setst{u_i}{i \in \setlift{V}} \subseteq \Reals^{\setlift{V}}\)
  satisfying~\eqref{eq:polyhedral-A-geomrep}.

  We first verify that
  \begin{equation}
    \label{eq:polyhedral-A-geomrep-goal1}
    \targetSet
    \subseteq
    \genTH{
      \PolyAdjCone{E^+}{E^-}
    }{
      \Psd{\setlift{V}}
    }.
  \end{equation}
  Let~\(\setst{u_i}{i \in \setlift{V}} \subseteq
  \Reals^{\setlift{V}}\) satisfy~\eqref{eq:polyhedral-A-geomrep}.
  Define \(U \in \Reals^{\paren{\setlift{V}} \times V}\) by setting
  \(U e_i \coloneqq u_i\) for every \(i \in V\).  Next, set \(Y
  \coloneqq U \Diag(U^{\transp} u_0)\) and
  \begin{equation*}
    \Xh
    \coloneqq
    \onelift{x}{X}
    \coloneqq
    \begin{bmatrix}
      u_0^{\transp} u_0^{} & u_0^{\transp} Y \\[2pt]
      Y^{\transp} u_0      & Y^{\transp} Y   \\
    \end{bmatrix}
    =
    \begin{bmatrix}
      u_0^{\transp} \\[2pt]
      Y^{\transp}   \\
    \end{bmatrix}
    \begin{bmatrix}
      u_0 & Y
    \end{bmatrix}
    \in \Psd{\setlift{V}},
  \end{equation*}
  where we used~\eqref{eq:polyhedral-A-geomrep-norms}.  Let us
  verify that
  \begin{subequations}
    \begin{alignat}{2}
      \label{eq:polyhedral-A-geomrep-goal1a}
      X        & \in \PolyAdjCone{E^+}{E^-}, \\
      \label{eq:polyhedral-A-geomrep-goal1b}
      \diag(X) & = x,                        \\
      \label{eq:polyhedral-A-geomrep-goal1c}
      x_i      & = \iprod{u_0}{u_i}^2        & \qquad & \forall i \in V.
    \end{alignat}
  \end{subequations}
  We start with~\eqref{eq:polyhedral-A-geomrep-goal1a}.  Note that
  \(X = Y^{\transp} Y = \DiagScaleMap{U^{\transp}
    u_0} \paren*{U^{\transp} U}\) and \(U^{\transp} u_0 \geq 0\)
  by~\eqref{eq:polyhedral-A-geomrep-zerorow}.  Since
  \(\PolyAdjCone{E^+}{E^-}\) is diagonally scaling-invariant, it
  suffices to show that \(U^{\transp} U \in \PolyAdjCone{E^+}{E^-}\).
  However, this is immediate
  from~\eqref{eq:polyhedral-A-geomrep-iprods1}
  and~\eqref{eq:polyhedral-A-geomrep-iprods2}.  This
  proves~\eqref{eq:polyhedral-A-geomrep-goal1a}.
  For~\eqref{eq:polyhedral-A-geomrep-goal1c}, note that
  \begin{equation}
    \label{eq:polyhedral-A-geomrep-aux1}
    x
    =
    Y^{\transp} u_0
    =
    \sqbrac{
      \Diag(U^{\transp} u_0)
      U^{\transp}
    }
    u_0
    =
    \paren{U^{\transp} u_0}
    \hprod
    \paren{U^{\transp} u_0}.
  \end{equation}
  By~\eqref{eq:polyhedral-A-geomrep-norms}, we have
  \(\diag(U^{\transp} U) = \ones\).  Thus,
  \(
  \diag(X)
  = \diag\paren*{
    \DiagScaleMap{U^{\transp} u_0}\paren{
      U^{\transp} U
    }
  }
  =
  \paren{U^{\transp} u_0}
  \hprod
  \diag(U^{\transp} U)
  \hprod
  \paren{U^{\transp} u_0}
  =
  \paren{U^{\transp} u_0}
  \hprod
  \paren{U^{\transp} u_0}
  =
  x
  \)
  by~\eqref{eq:polyhedral-A-geomrep-aux1}, thus
  proving~\eqref{eq:polyhedral-A-geomrep-goal1b}.  It follows that
  \(x \in \genTH{\PolyAdjCone{E^+}{E^-}}{\Psd{\setlift{V}}}\), and the
  proof of~\eqref{eq:polyhedral-A-geomrep-goal1} is complete.

  Now we show that
  \begin{equation}
    \label{eq:polyhedral-A-geomrep-goal2}
    \genTH{
      \PolyAdjCone{E^+}{E^-}
    }{
      \Psd{\setlift{V}}
    }
    \subseteq
    \targetSet.
  \end{equation}
  Let \(x \in \genTH{\PolyAdjCone{E^+}{E^-}}{\Psd{\setlift{V}}}\), and
  let \(\Xh \in
  \liftedGenTH{\PolyAdjCone{E^+}{E^-}}{\Psd{\setlift{V}}}\) such that
  \(x = \diag(X)\) for \(X \coloneqq \Xh[V]\).  Let \(Y \in
  \Reals^{(\setlift{V}) \times (\setlift{V})}\) such that \(\Xh =
  Y^{\transp} Y\).  Set \(y_i \coloneqq Y e_i\) for each \(i \in
  \setlift{V}\).  Let \(Z \coloneqq \setst{i \in V}{y_i = 0}\).  Define
  \(u_i \coloneqq y_i/\norm{y_i}\) for each \(i \in \setlift{(V \drop
    Z)}\) and let \(\setst{u_i}{i \in Z}\) be an orthonormal basis for
  a subspace of \(\setst*{u_i}{i \in \setlift{(V \drop Z)}}^{\perp}\)
  of appropriate dimension.

  We must show that~\eqref{eq:polyhedral-A-geomrep} holds.  Note
  that~\eqref{eq:polyhedral-A-geomrep-zerorow} for \(i \in V \drop
  Z\) follows from \(\Xh e_0 = \diag(\Xh) \geq 0\), and for \(i \in
  Z\) it holds by construction.  We also know
  that~\eqref{eq:polyhedral-A-geomrep-norms} holds by
  construction.  Let us
  check~\eqref{eq:polyhedral-A-geomrep-iprods1}.  Let \(ij \in
  E^+\).  If~\(i\) or~\(j\) is in~\(Z\), then \(\iprod{u_i}{u_j} =
  0\), so we may assume that \(i,j \in V \drop Z\).  Then \(\norm{y_i}
  \norm{y_j} \iprod{u_i}{u_j} = \iprod{y_i}{y_j} = X_{ij} \geq 0\)
  since \(X \in \PolyAdjCone{E^+}{E^-}\).  This completes the proof
  of~\eqref{eq:polyhedral-A-geomrep-iprods1}.  The proof
  of~\eqref{eq:polyhedral-A-geomrep-iprods2} is analogous,
  so~\eqref{eq:polyhedral-A-geomrep} holds.

  Lastly, we show that \(x\) is given
  by~\eqref{eq:polyhedral-A-xform}.  Let \(i \in V\).  Since \(\Xh
  e_0 = \diag(\Xh)\), we have \(x_i = \component{Y^{\transp} Y}{0i} =
  \iprod{y_0}{y_i} = \norm{y_0} \norm{y_i} \iprod{u_0}{u_i} =
  X_{ii}^{1/2} \iprod{u_0}{u_i} = x_i^{1/2} \iprod{u_0}{u_i}\).  If
  \(x_i > 0\), then \(x_i^{1/2} = \iprod{u_0}{u_i}\).  Otherwise, \(u_i
  \perp u_0\) by construction, so \(x_i = 0 = \iprod{u_0}{u_i}^2\).
  This proves that~\(x\) is given by~\eqref{eq:polyhedral-A-xform}
  and completes the proof
  of~\eqref{eq:polyhedral-A-geomrep-goal2}.
\end{proof}

An \emph{orthonormal representation} of a graph~\(G = (V,E)\) is a
map~\(u\) that sends~\(V\) into the unit vectors of some Euclidean
space such that \(\iprod{u_i}{u_j} = 0\) whenever \(ij \in
\overline{E}\).  If, additionally, \(\iprod{u_i}{u_j} \geq 0\)
whenever \(ij \in E\), then \(u\) is called an \emph{acute orthonormal
  representation} of~\(G\).  Finally, an \emph{obtuse representation}
of~\(G\) is a map~\(u\) from~\(V\) to the unit vectors of some
Euclidean space so that \(\iprod{u_i}{u_j} \leq 0\) whenever \(ij \in
\overline{E}\).

Proposition~\ref{prop:polyhedral-A-geomrep} immediately leads to
the following well-known internal description of the sets
in~\eqref{eq:THs-as-genTHs}.
\begin{corollary}
  \label{cor:Acute-Ortho-Obtuse}
  Let \(G = (V,E)\) be a graph.  Let \(\ThetaBodyA \in \set{\TH(G),
    \TH'(G), \TH^+(G)}\).  Then \(\ThetaBodyA\) consists of all vectors
  \(x \in \Reals^V\) of the form \(x_i = \iprod{u_0}{u_i}^2\) for
  every \(i \in V\) for some unit vectors in \(\setst{u_i}{i \in
    \setlift{V}} \subseteq \Reals^{\setlift{V}}\) such that
  \begin{enumerate}[(i)]
  \item \(u\restriction_V\) is an orthonormal representation
    of~\(\overline{G}\), if \(\ThetaBodyA = \TH(G)\);
  \item \(u\restriction_V\) is an acute orthonormal representation
    of~\(\overline{G}\) and \(\iprod{u_0}{u_i} \geq 0\) for all \(i
    \in V\), if \(\ThetaBodyA = \TH'(G)\);
  \item \(u\restriction_V\) is an obtuse representation
    of~\(\overline{G}\) and \(\iprod{u_0}{u_i} \geq 0\) for all \(i
    \in V\), if \(\ThetaBodyA = \TH^+(G)\).
  \end{enumerate}
\end{corollary}
\begin{proof}
  Immediate from Proposition~\ref{prop:polyhedral-A-geomrep}.  When
  \(\ThetaBodyA = \TH(G)\), the constraint \(\iprod{u_0}{u_i} \geq 0\)
  may be dropped, since for each orthonormal representation~\(u\)
  of~\(G\) and \(i \in V\), the map obtained from~\(u\) by replacing
  some image~\(u_i\) by~\(-u_i\) is also an orthonormal representation
  of~\(G\).
\end{proof}

\subsection{Luz and Schrijver's Convex Quadratic Characterization}

In this subsection, we show how to generalize the convex quadratic
characterization from~\cite{LuzS05a} to the context of generalized
theta bodies over the cone \(\liftedQuadCone = \Psd{\setlift{V}}\).
This provides a generalization of their results to all weights \(w \in
\Reals_+^V\) to all of the functions~\(\theta\), \(\theta'\),
and~\(\theta^+\).  We remark that a convex quadratic characterization
of~\(\theta'\) was already known to Luz~\cite{Luz06a}, as well as a
weighted generalization of the convex quadratic characterization
of~\(\theta\), which appeared in an unpublished report (in Portuguese)
by Luz in~2005~\cite{Luz14a}.

Let \(C \in \Sym{V}\) such that \(\diag(C) = 0\).  For each \(w \in
\Reals_+^V\), define
\begin{equation}
  \label{eq:LuzS-upsilon-C}
  \LSthetaC{C}{w}
  \coloneqq
  \max\setst*{
    2\iprod{w}{x} - \qform{\DiagScaleMap{\sqrt{w}}(H_C + I)}{x}
  }{
    x \in \Reals_+^V
  },
\end{equation}
where
\begin{equation}
  H_C \coloneqq \Iverson*{C \neq 0} \dfrac{C}{-\lambda_{\min}(C)}.
\end{equation}
Note that \(\diag(C) = 0\) implies that \(H_C + I\succeq 0\), so the
quadratic program on the RHS of~\eqref{eq:LuzS-upsilon-C} is convex.
In particular, there is an optimal solution whenever the optimal value
is finite.  The necessary and sufficient conditions for optimality
are:
\begin{subequations}
  \label{eq:LuzS-KKT}
  \begin{gather}
    \label{eq:LuzS-KKT1}
    x \geq 0,
    \\
    \label{eq:LuzS-KKT2}
    \DiagScaleMap{\sqrt{w}}(H_C + I)x \geq w,
    \\
    \label{eq:LuzS-KKT3}
    \qform{\DiagScaleMap{\sqrt{w}}(H_C + I)}{x} = \iprod{w}{x}
    = \LSthetaC{C}{w}.
  \end{gather}
\end{subequations}
Let \(\AdjCone \subseteq \Sym{V}\) be a polyhedral diagonally
scaling-invariant cone.  Define
\begin{equation}
  \label{eq:LuzS-upsilon}
  \LStheta{\AdjCone}{w}
  \coloneqq
  \inf\setst*{
    \LSthetaC{C}{w}
  }{
    C \in \AdjCone \cap \Null(\diag)
  }.
\end{equation}

We first show that \(\LStheta{\AdjCone}{w}\) provides an upper bound
for \(\theta_1\paren[\big]{\thinspace\AdjDual{\AdjCone}, \Psd{V}}\).
Note the similarity with the proof of
Theorem~\ref{thm:theta3-theta4}.
\begin{proposition}
  \label{prop:LuzS-weak}
  Let \(\AdjCone \subseteq \Sym{V}\) be a diagonally scaling-invariant
  closed convex cone such that \(\Image(\Diag) \subseteq \AdjCone\).
  Let \(w \in \Reals_+^V\).  Then
  \begin{equation}
    \label{eq:LuzS-weak}
    \min\setst*{
      \max_{i \in V}
      \dfrac{w_i}{y_i}
    }{
      y \in \genTH[\big]{\AdjCone}{\Psd{\setlift{V}}}
    }
    \leq
    \LStheta{\AdjCone}{w}.
  \end{equation}
\end{proposition}
\begin{proof}
  \newcommand*{\xopt}{\bar{x}}
  Let \(C \in \AdjCone \cap \Null(\diag)\) such that \(\upsilon
  \coloneqq \LSthetaC{C}{w} < \infty\), and let \(\xopt\) be an
  optimal solution for the corresponding optimization
  problem~\eqref{eq:LuzS-upsilon-C}.  We shall use the optimality
  conditions~\eqref{eq:LuzS-KKT} without further mention.  We may
  assume that \(\supp(\xopt) \subseteq \supp(w)\), so we may also
  assume that \(\supp(w) = V\) and that \(\upsilon > 0\).  Write \(H_C
  + I \succeq 0\) as \(H_C + I = B^{\transp} B\) for some \(B \in
  \Reals^{V \times V}\).  Set
  \begin{gather*}
    c \coloneqq \upsilon^{-1/2} B \Diag(\sqrt{w}) \xopt, \\
    z \coloneqq B^{\transp} c,                           \\
    \Bb \coloneqq B \Diag(B^{\transp}c),                 \\
    y \coloneqq \Bb^{\transp} c = \Diag(z)z = z \hprod z.
  \end{gather*}
  Note that \(\DiagScaleMap{\sqrt{w}}\paren{H_C+I}\xopt \geq w\)
  implies that
  \begin{equation}
    \label{eq:LuzS-weak-aux1}
    z
    =
    \upsilon^{-1/2} (H_C+I) \Diag(\sqrt{w}) \xopt
    \geq
    \upsilon^{-1/2} \sqrt{w}.
  \end{equation}

  We claim that
  \begin{equation}
    \label{eq:LuzS-weak-aux2}
    \Yh
    \coloneqq
    \begin{bmatrix}
      1 & y^{\transp} \\
      y & Y
    \end{bmatrix}
    \coloneqq
    \begin{bmatrix}
      1 & y^{\transp}       \\
      y & \Bb^{\transp} \Bb \\
    \end{bmatrix}
    \in \genTH[\big]{\AdjCone}{\Psd{\setlift{V}}}.
  \end{equation}
  Positive semidefiniteness of the matrix in~\eqref{eq:LuzS-weak-aux2}
  follows from the factorization
  \begin{equation*}
    \begin{bmatrix}
      1 & y^{\transp}       \\
      y & \Bb^{\transp} \Bb \\
    \end{bmatrix}
    =
    \begin{bmatrix}
      c^{\transp}   \\
      \Bb^{\transp} \\
    \end{bmatrix}
    \begin{bmatrix}
      c & \Bb \\
    \end{bmatrix}.
  \end{equation*}
  Since we have \(H_C + I \in \AdjCone\) and \(z \geq 0\) follows
  from~\eqref{eq:LuzS-weak-aux1}, we get \(\Bb^{\transp} \Bb =
  \DiagScaleMap{z}(H_C+I) \in \AdjCone\).  Finally
  \begin{equation*}
    \diag(\Bb^{\transp} \Bb)
    =
    z
    \hprod
    \diag(B^{\transp} B)
    \hprod
    z
    =
    y
  \end{equation*}
  since \(\diag(C) = 0\).  This concludes the proof
  of~\eqref{eq:LuzS-weak-aux2}, and so we have \(y \in
  \genTH[\big]{\AdjCone}{\Psd{\setlift{V}}}\).  It follows
  from~\eqref{eq:LuzS-weak-aux1} that \(y_i \geq \upsilon^{-1} w_i\)
  for each \(i \in V\), so that \(\max_{i \in V} w_i/y_i \leq
  \upsilon\).
\end{proof}

Next we show how \(\LStheta{\AdjCone}{w}\) relates to
\(\theta_2\paren[\big]{\thinspace\AdjDual{\AdjCone}, \Psd{V}}\):

\bgroup
\newcommand*{\Yopt}{\bar{Y}}
\begin{theorem}
  \label{thm:LuzS-strong}
  Let \(\AdjCone \subseteq \Sym{V}\) be a polyhedral diagonally
  scaling-invariant cone such that \(\Image(\Diag) \subseteq
  \AdjCone\).  Let \(w \in \Reals_+^V\).  Let \(\Yopt\) be an optimal
  solution for
  \begin{equation}
    \label{eq:LuzS-strong-optproblem}
    \theta
    \coloneqq
    \min\setst*{
      \lambda_{\max}\paren[\big]{Y + \oprodsym{\sqrt{w}}}
    }{
      Y \in -\AdjCone \cap \Null(\diag)
    }.
  \end{equation}
  Then \(\theta \geq \LSthetaC{-\Yopt}{w}\).
\end{theorem}
\begin{proof}
  \newcommand*{\Copt}{\bar{C}}
  First note that the dual of~\eqref{eq:LuzS-strong-optproblem} has
  \(\frac{1}{\card{V}} I\) as a restricted Slater point, so an optimal
  solution for~\eqref{eq:LuzS-strong-optproblem} exists and \(\theta
  \geq 0\), with equality only if \(w = 0\).  We may thus assume that
  \(\theta > 0\).  Set \(\Copt \coloneqq - \Yopt\).  Clearly, \(\theta
  \geq \lambda_{\max}(\Yopt)\).  If \(\theta =
  \lambda_{\max}(\Yopt)\), then the objective value of an arbitrary
  \(x \in \Reals_+^V\) in the formulation for \(\LSthetaC{C}{w}\) is
  \begin{equation*}
    \begin{split}
      2\iprod{w}{x}
      -
      \qform{\DiagScaleMap{\sqrt{w}}\paren{H_{\Copt}+I}}{x}
      & =
      2\iprod{w}{x}
      -
      \theta^{-1}
      \qform{
        \DiagScaleMap{\sqrt{w}}
        \paren*{
          \theta I - \Yopt - \oprodsym{\sqrt{w}}
        }
      }{
        x
      }
      -
      \theta^{-1}\iprod{w}{x}^2
      \\
      & \leq
      2\iprod{w}{x}
      -
      \theta^{-1}\iprod{w}{x}^2
      \leq
      \theta
    \end{split}
  \end{equation*}
  since \(\paren{\theta^{1/2} - \theta^{-1/2}\iprod{w}{x}}^2 \geq 0\),
  and hence \(\LSthetaC{-\Yopt}{w} \leq \theta\).  Thus, we may
  assume that
  \begin{equation}
    \label{eq:LuzS-strong-aux1}
    \theta > \lambda_{\max}(\Yopt).
  \end{equation}

  \newcommand*{\Bopt}{\bar{B}}
  \newcommand*{\yopt}{\bar{y}}
  \newcommand*{\xopt}{\bar{x}}
  Let \(\Bopt\) be an optimal solution for the dual
  of~\eqref{eq:LuzS-strong-optproblem}, given by
  \(\sup\setst{\qform{B}{\sqrt{w}}}{\trace(B) = 1,\, B \in
    \AdjDual{\AdjCone} \cap \Psd{V}}\).  Note that \(\theta I -
  \Yopt\) is nonsingular by~\eqref{eq:LuzS-strong-aux1}, so the rank
  of \(\theta I - \Yopt - \oprodsym{\sqrt{w}}\) is \(\geq
  \card{V}-1\).  Since \(\iprod{\Bopt}{\theta I - \Yopt -
    \oprodsym{\sqrt{w}}} = 0\) by complementarity, it follows that
  \(\Bopt\) has rank one.  Write \(\Bopt = \theta^{-1}
  \oprodsym{\yopt}\) with \(\yopt\) having at least one positive
  component, set \(S \coloneqq \supp(\yopt)\) and \(\xopt \coloneqq
  \incidvector{S}\).  Let us show that
  \begin{subequations}
    \label{eq:LuzS-strong-auxs}
    \begin{gather}
      \label{eq:LuzS-strong-aux2}
      \yopt = \Diag(\sqrt{w}) \xopt,
      \\
      \label{eq:LuzS-strong-aux3}
      S \subseteq \supp(w),
      \\
      \label{eq:LuzS-strong-aux4}
      \iprod{\sqrt{w}}{\yopt} = \iprod{w}{\xopt} = \theta,
      \\
      \label{eq:LuzS-strong-aux4b}
      \norm{y}^2 = \theta.
    \end{gather}
  \end{subequations}
  We shall apply Lemma~\ref{lemma:gijswijt-lambda-trick}.  The
  inclusion~\eqref{eq:LuzS-strong-aux3} follows
  from~\eqref{eq:gijswijt-lambda-trick1}.
  Equation~\eqref{eq:gijswijt-lambda-trick4} yields
  \(\iprod{\sqrt{w}}{\yopt}\yopt = \theta \MPinverse{\Diag(\sqrt{w})}
  \Diag(\yopt) \yopt\).  If we multiply this equation on the left by
  \(\Diag(\sqrt{w}) \MPinverse{\Diag(\yopt)}\) and
  use~\eqref{eq:LuzS-strong-aux3} we get
  \begin{equation}
    \label{eq:2}
    \iprod{\sqrt{w}}{\yopt} \Diag(\sqrt{w}) \xopt = \theta \yopt.
  \end{equation}
  Thus, \(\yopt \geq 0\) or
  \(\yopt \leq 0\), so that \(\yopt \geq 0\) since \(\yopt\) has at
  least one positive component.  Then \(\theta =
  \qform{\Bopt}{\sqrt{w}}\) implies that \(\iprod{\sqrt{w}}{\yopt} =
  \theta\), which establishes~\eqref{eq:LuzS-strong-aux2} via~\eqref{eq:2} and half
  of~\eqref{eq:LuzS-strong-aux4}; the other half follows from the half
  already established and~\eqref{eq:LuzS-strong-aux2}.  Finally,
  \eqref{eq:LuzS-strong-aux4b} follows from \(\trace(\Bb) = 1\).

  We claim that
  \begin{equation}
    \label{eq:LuzS-strong-aux5}
    \xopt
    \text{
      satisfies the optimality conditions~\eqref{eq:LuzS-KKT} for
    }
    \theta = \LSthetaC{\Copt}{w}.
  \end{equation}
  Clearly \eqref{eq:LuzS-KKT1} holds.  Since complementarity yields
  \(\paren{\theta I - \Yopt - \oprodsym{\sqrt{w}}} \yopt = 0\), we
  have \(\theta \yopt = \paren{\oprodsym{\sqrt{w}} + \Yopt} \yopt\)
  and thus \(-\Yopt \yopt = \theta\paren*{\sqrt{w} - \yopt}\)
  using~\eqref{eq:LuzS-strong-aux4}.  Thus,
  \begin{equation}
    \label{eq:LuzS-strong-aux6}
    \DiagScaleMap{\sqrt{w}}
    \paren{
      H_{\Copt} + I
    }
    \xopt
    =
    \Diag(\sqrt{w})
    \sqbrac*{
      \dfrac{
        -\Yopt
      }{
        \lambda_{\max}(\Yopt)
      }
      + I
    }
    \yopt
    =
    \Diag(\sqrt{w})
    \sqbrac*{
      \dfrac{\theta}{\lambda_{\max}(\Yopt)}
      \paren{\sqrt{w} - \yopt}
      + \yopt
    }.
  \end{equation}
  Since \(\theta \geq \lambda_{\max}(\Yopt)\) and \(\sqrt{w} \geq \yb\), we get
  from~\eqref{eq:LuzS-strong-aux6} that
  \(\DiagScaleMap{\sqrt{w}} \paren{H_{\Copt} + I} \xopt \geq
  \Diag(\sqrt{w}) \sqbrac{\paren{\sqrt{w}-\yopt}+\yopt}) = w\),
  so~\eqref{eq:LuzS-KKT2} holds.  By hitting the LHS
  of~\eqref{eq:LuzS-strong-aux6} on the left with~\(\xb^{\transp}\),
  we get from the RHS using~\eqref{eq:LuzS-strong-auxs} that
  \(\qform{\DiagScaleMap{\sqrt{w}}(H_{\Copt}+I)}{\xb} = \theta\).
  Thus, \eqref{eq:LuzS-KKT3} follows from~\eqref{eq:LuzS-strong-aux4},
  and the proof of~\eqref{eq:LuzS-strong-aux5} is complete.
\end{proof}
\egroup

\begin{corollary}
  Let \(\AdjCone \subseteq \Sym{V}\) be a polyhedral diagonally
  scaling-invariant cone such that \(\Image(\Diag) \subseteq
  \AdjCone\).  Let \(w \in \Reals_+^V\).  Then
  \begin{equation}
    \label{eq:LuzS}
    \LStheta{\AdjCone}{w}
    =
    \theta\paren[\big]{
      \thinspace\AdjDual{\AdjCone}, \Psd{V}; w
    }.
  \end{equation}
  Moreover, the optimization problem for \(\LStheta{\AdjCone}{w}\)
  has an optimal solution.
\end{corollary}
\begin{proof}
  Immediate from Proposition~\ref{prop:LuzS-weak} and
  Theorem~\ref{thm:LuzS-strong}.
\end{proof}

\section{Conclusion}

We took an axiomatic viewpoint in our study of the Lovász theta
function and the related theta body of graphs.  We generalized the
binary encoding of graphs by the notion of diagonally
scaling-invariant polyhedral cones and the semidefinite cone with a
more general set of convex cones.  These generalization and viewpoint
led to graph complementation being replaced by convex polarity and to
the new notion of Schur-Lifting of cones as the dual of commonly used
PSD-Lifting of cones.  Our new general theory has many advantages: we
are able to treat the stable set polytope and many of its convex
relaxations uniformly and extend the most commonly used equivalent
characterizations of Lovász theta function and the most powerful
properties to these generalized theta bodies.

\appendix

\section{Antiblocking Duality}

Antiblocking duality is the form of duality most suitable for convex
corners, and we used extensively the basic facts of this duality
theory in this paper.  Even though such facts are well known, we are
not aware of any previous treatment in the literature of antiblocking
duality for non-polyhedral convex corners which meets the needs of this paper.  Thus, for the sake of
completeness, we include a brief, self-contained description of this
theory.

\bgroup
\newcommand*{\corner}{\mathscr{C}}
Recall from Section~\ref{sec:theta-convex-corners} that a convex
corner is a compact, lower-comprehensive convex subset of the
nonnegative orthant with nonempty interior, and that the antiblocker
of a convex corner~\(\corner \subseteq \Reals_+^n\) is \(\abl{\corner}
\coloneqq \setst{y \in \Reals_+^n}{\iprod{y}{x} \leq 1\,\forall x \in
  \corner}\).  Note that \(\abl{\corner}\) is also a convex corner.
Closedness, convexity and inclusion in~\(\Reals_+^n\) are clear,
whereas lower-comprehensiveness follows from the inclusion~\(\corner
\subseteq \Reals_+^n\).  Boundedness of~\(\abl{\corner}\) follows from
\(\corner \cap \Reals_{++}^n \neq \emptyset\), since
\(\interior\paren{\corner} \neq \emptyset\).  Finally, since
\(\corner\) is bounded, \(\abl{\corner}\) has nonempty interior: if
\(M \in \Reals\) is such that \(x \leq n^{-1} M \ones\) for all \(x
\in \corner\), then \(M^{-1}\ones \in \abl{\corner}\), so that
\(\tfrac{1}{2}M^{-1}\ones \in \interior(\abl{\corner})\) by
lower-comprehensiveness of~\(\abl{\corner}\).

The key fact of antiblocking duality is that \(\abl{\cdot}\) defines
an involution on the class of convex corners, which we prove next.
\begin{theorem}
  \label{thm:abl-duality}
  If \(\corner \subseteq \Reals_+^n\) is a convex corner, then
  \(\abl{\abl{\corner}} = \corner\).
\end{theorem}
\begin{proof}
  \newcommand*{\xopt}[1][]{x_{#1}^*}
  The inclusion `\(\supseteq\)' is trivial.  For the other inclusion,
  let \(u \in \Reals^n \drop \corner\).  If \(u \not\in \Reals_+^n\),
  then \(u \in \Reals^n \drop \abl{\abl{\corner}}\), so assume that
  \(u \geq 0\).  Let \(\xopt \in \corner\) minimize \(\norm{u-x}^2\)
  over \(x \in \corner\).  Set
  \begin{equation*}
    y \coloneqq u - \xopt
    \qquad
    \text{and}
    \qquad
    \mu \coloneqq \iprod{y}{\xopt}.
  \end{equation*}
  We claim that
  \begin{equation}
    \label{eq:abl-sep}
    \iprod{y}{x} \leq \mu
    \quad
    \forall x \in \corner,
    \qquad
    \text{and}
    \qquad
    \iprod{y}{u} > \mu.
  \end{equation}

  Let \(x \in \corner\).  If \(t \in \char40{}0,1\char93\), then \(\xopt
  + t(x - \xopt) \in \corner\), so that \(\norm{u-\xopt}^2 \leq
  \norm{u-\xopt-t(x-\xopt)}^2\), whence \(\iprod{u-\xopt}{x-\xopt} \leq
  t\norm{x-\xopt}^2/2\).  By sending \(t \downarrow 0\), we get
  \(\iprod{u-\xopt}{x-\xopt} \leq 0\), which implies the first half
  of~\eqref{eq:abl-sep}.  The second half is easily seen to be
  equivalent to \(\norm{y}^2 > 0\), which follows from the fact that
  \(\xopt \in \corner\) whereas \(u \not\in \corner\).  This concludes
  the proof of~\eqref{eq:abl-sep}.

  Next we show that
  \begin{equation}
    \label{eq:abl-good-sep}
    y \geq 0.
  \end{equation}
  Suppose that \(y_i < 0\) for some \(i \in \set{1,\dotsc,n}\).  Then
  \(\xopt[i] > u_i \geq 0\).  Let \(\eps > 0\) such that \(\xopt[i]
  \geq \eps\) and set \(\xb \coloneqq \xopt - \eps e_i \in \corner\).
  By~\eqref{eq:abl-sep}, we get \(\mu \geq \iprod{y}{\xb} =
  \iprod{y}{\xopt} - \eps\iprod{y}{e_i} = \mu - \eps y_i > \mu\), a
  contradiction.  This proves~\eqref{eq:abl-good-sep}, which yields
  \(\mu > 0\) when combined with~\eqref{eq:abl-sep} and
  \(\interior(\corner) \neq \emptyset\).  Now~\eqref{eq:abl-sep}
  and~\eqref{eq:abl-good-sep} show that \(\tfrac{1}{\mu} y \in
  \abl{\corner}\), so \(u \in \Reals^n \drop \abl{\abl{\corner}}\)
  by~\eqref{eq:abl-sep}.  This concludes the proof of~`\(\subseteq\)'.
\end{proof}

\bgroup
\newcommand*{\xopt}[1][]{x_{#1}^*}
\newcommand*{\yopt}[1][]{y_{#1}^*}
We thus obtain the following optimality conditions:
\begin{corollary}
  Let \(\corner \subseteq \Reals_+^n\) be a convex corner.  Let \(w
  \in \Reals_+^n\), and let \(\theta \in \Reals_+\).  Then
  \(\suppf{\abl{\corner}}{w} = \theta\) if and only if there exists
  \(\xopt \in \abl{\corner}\) and \(\yopt \in \corner\) such that
  \(\iprod{\yopt}{\xopt} = 1\) and \(\theta \yopt = w\).
\end{corollary}
\begin{proof}
  First we prove necessity.  If \(w = 0\), let \(\xopt \in
  \abl{\corner}\) maximize \(\norm{x}^2\) over \(x \in
  \abl{\corner}\), and set \(\yopt \coloneqq \xopt/\norm{\xopt}^2\);
  if \(x \in \abl{\corner}\), then \(\iprod{\yopt}{x} \leq \norm{x}
  \norm{\xopt}/\norm{\xopt}^2 \leq 1\) by Cauchy-Schwarz, so \(\yopt
  \in \corner\) by Theorem~\ref{thm:abl-duality}.  Assume that \(w
  \neq 0\), so \(\theta > 0\).  Let \(\xopt\) be
  maximize~\(\iprod{w}{x}\) over \(x \in \abl{\corner}\), and set
  \(\yopt \coloneqq w/\theta\).  If \(x \in \abl{\corner}\), then
  \(\iprod{\yopt}{x} \leq 1\) since \(\theta =
  \suppf{\abl{\corner}}{w}\), so \(\yopt \in \corner\) by
  Theorem~\ref{thm:abl-duality}.

  Now we prove sufficiency.  If \(x \in \abl{\corner}\), then
  \(\iprod{w}{x} = \theta \iprod{\yopt}{x} \leq \theta\) shows that
  \(\suppf{\abl{\corner}}{w} \leq \theta\), whereas \(\iprod{w}{\xopt}
  = \theta\iprod{\yopt}{\xopt} = \theta\) shows that equality holds.
\end{proof}
\egroup                         % \xopt, \yopt
\egroup                         % \corner

\end{document}